\documentclass[12pt]{amsart}
\usepackage{amsmath,amssymb}    
\usepackage{url}                
\usepackage{graphicx}           
\usepackage{amsthm}
\usepackage{amscd}
\usepackage{amsfonts}
\usepackage{latexsym}

\theoremstyle{definition}
\newtheorem{theorem}{Theorem}[section]
\newtheorem{lemma}[theorem]{Lemma}
\newtheorem{proposition}[theorem]{Proposition}
\newtheorem{corollary}[theorem]{Corollary}
\newtheorem{problem}[theorem]{Problem}
\newtheorem{definition}[theorem]{Definition}
\newtheorem{example}[theorem]{Example}

\newtheorem{remark}[theorem]{Remark}

\newcommand{\VV}{{\mathbf{V}}}
\newcommand{\ee}{{\mathbf{e}}}

\newcommand{\ZZ}{\mathbb{Z}}
\newcommand{\C}{\mathbb{C}}

\newcommand{\T}{\mathcal{T}}
\newcommand{\hh}{\mathfrak{h}}

\newcommand{\Hom}{\text{Hom}}
\newcommand{\Irrep}{\text{Irrep}}
\newcommand{\supp}{\text{supp}}
\newcommand{\rank}{\text{rank}}
\newcommand{\Sh}{\text{Sh}}
\newcommand{\Rep}{\text{Rep}}
\newcommand{\Lie}{\text{Lie}}
\newcommand{\Spec}{\text{Spec}}
\newcommand{\loc}{\text{loc}}

\newcommand{\Id}{\text{id}}

\newcommand{\g}{\mathfrak{g}}

\newcommand{\GL}{\operatorname{GL}}

\newcommand{\slf}{\mathfrak{sl}}

\newcommand{\D}{\mathcal{D}}
\newcommand{\U}{\mathcal{U}}
\newcommand{\A}{\mathcal{A}}
\newcommand{\z}{\mathfrak{z}}
\newcommand{\Pro}{\mathcal{P}}
\newcommand{\End}{\operatorname{End}}
\newcommand{\Str}{\mathcal{O}}
\renewcommand{\sl}{\mathfrak{sl}}
\newcommand{\taut}{\mathcal{T}}
\newcommand{\Ext}{\operatorname{Ext}}
\newcommand{\Dcal}{\mathcal{D}}
\newcommand{\LEnd}{\mathcal{E}nd}

\newcommand{\ii}{\rm i}
\newcommand{\h}{{\mathfrak{h}}}
\newcommand{\CC}{{\mathbb{C}}}

\DeclareMathOperator{\Tr}{Tr}
\DeclareMathOperator{\ch}{ch}
\DeclareMathOperator{\Ch}{Ch}

\DeclareMathOperator{\Coef}{Coef}

\DeclareMathOperator{\Res}{Res}
\DeclareMathOperator{\sgn}{sgn}
\DeclareMathOperator{\Ker}{Ker}
\DeclareMathOperator{\Alt}{Alt}

\date{}

\author{Pavel Etingof}
\address{Department of Mathematics, Massachusetts Institute of Technology,
Cambridge, MA 02139, USA}
\email{etingof@math.mit.edu}

\author{Eugene Gorsky}
\address{Mathematics Department, Stony Brook University,
Stony Brook NY, 11794-3651, USA}
\email{egorsky@math.sunysb.edu}

\author{Ivan Losev}
\address{Department of Mathematics,
Northeastern University, Boston, MA, 02115, USA}
\email{I.Loseu@neu.edu}

\title[Representations with minimal support and torus knots]{Representations of Rational Cherednik algebras with minimal
support and torus knots}

\begin{document}


\begin{abstract}
In this paper we obtain several results about
representations of rational Cherednik algebras, and
discuss their applications.
Our first result is the Cohen-Macaulayness property (as modules
over the polynomial ring) of Cherednik algebra modules with minimal
support. Our second result is an explicit formula
for the character of an irreducible minimal support module in
type $A_{n-1}$ for $c=\frac{m}{n}$,
and an expression of its quasispherical part (i.e.,
the isotypic part of ``hooks'') in terms of the HOMFLY
polynomial of a torus knot colored by a Young diagram.
We use this formula and the work of Calaque,
Enriquez and Etingof to give explicit formulas
for the characters
of the irreducible equivariant D-modules on
the nilpotent cone for $SL_m$.
Our third result is the construction
of the Koszul-BGG complex for the rational Cherednik
algebra, which generalizes the construction of the Koszul-BGG resolution
from \cite{BEG} and \cite{Go}, and the calculation
of its homology in type $A$. We also show in type $A$
that the differentials in the Koszul-BGG complex
are uniquely determined by the condition that they are
nonzero homomorphisms of modules over the Cherednik algebra.
Finally, our fourth result is
the symmetry theorem, which identifies
the quasispherical components in the representations
with minimal support over the rational Cherednik algebras
$H_{\frac{m}{n}}(S_n)$ and $H_{\frac{n}{m}}(S_m)$.
In fact, we show that the simple quotients of the
corresponding quasispherical subalgebras are isomorphic
as filtered algebras.
This symmetry was essentially
established in \cite{CEE} in the spherical case,
and in \cite{Gor} in the case $GCD(m,n)=1$, and
it has a natural interpretation in terms of invariants of torus
knots.
\end{abstract}

\maketitle

\section{Introduction}

The goal of this paper is to establish a number of properties
of representations of rational Cherednik algebras with minimal
support, and connect them to knot invariants.
Our motivation came from the connections of representations of Cherednik
algebras with quantum invariants of torus knots
and Hilbert schemes of plane curve singularities
(such as $x^m=y^n$, $GCD(m,n)=1$),
see \cite{GORS}.

\subsection{} Let $\h$ be a finite dimensional complex vector space,
$W\subset GL(\h)$ a finite subgroup,
$S\subset W$ the set of reflections, and $c: S\to \CC$ a conjugation invariant function.
Let $H_c(W,\h)$ be the rational Cherednik algebra attached to $W,\h$.
Let ${\mathcal{O}}_c={\mathcal O}_c(W,\h)$ be the category of modules over this algebra
which are finitely generated over $\CC[\h]=S\h^*$ and locally
nilpotent under $\h$. Typical examples of objects of this
category are $M_c(\tau)$, the Verma (a.k.a. standard)
module over $H_c(W,\h)$ with lowest
weight $\tau\in \Irrep\ W$, and $L_c(\tau)$, the
irreducible quotient of $M_c(\tau)$.

Any object $M\in {\mathcal{O}}_c$, being a finitely generated
$\CC[\h]$-module, has support $\supp(M)$ as a module over $\CC[\h]$, which
is a closed subvariety of $\h$.

\begin{definition} We say that $M\in {\mathcal O}_c$ has minimal support if
no subset of $\supp(M)$ of smaller dimension is the support of a nonzero object
of ${\mathcal O}_c$.
\end{definition}

Our first main result is

\begin{theorem}\label{cmth}
If $M$ has minimal support then it is a Cohen-Macaulay module
over $\CC[\h]$ of dimension $d=\dim\supp(M)$.
In other words, it is free over any polynomial
subalgebra $\CC[p_1,...,p_d]\subset \CC[\h]$ (with
homogeneous $p_i$) over which it is finitely
generated.
\end{theorem}

\begin{remark}
Note that the minimal support condition is needed.
For example, if $W=S_3$, $c=1/3$, and
$M=L_c(\h)$ is the irreducible module with lowest weight $\h$,
then $M$ is the augmentation ideal in $\CC[\h]$,
so it is not Cohen-Macaulay (as it is not free).
\end{remark}

\subsection{}
Our second result is the character formula for
irreducible minimally supported modules
for rational Cherednik algebras of $S_n$ for $c=\frac{m_0}{n_0}$, and its
consequences. Let $\h=\h_n$ be the reflection representation of
$S_n$, and consider the rational Cherednik algebra
$H_c(S_n):=H_c(S_n,\h_n)$, where $c=\frac{m_0}{n_0}$ and $m_0,n_0\in \ZZ_{\ge 1}$
are coprime.
Let $n=dn_0+r$, where $0\leqslant r<n_0$. Recall from \cite{Wi} that minimally
supported modules in the category ${\mathcal O}_c:={\mathcal O}_c(S_n,\h_n)$
are of the form $L_c(n_0\lambda+\lambda')$,
where $\lambda$ is a partition of $d$ and $\lambda'$ is a partition of $r$. Here $n_0\lambda+\lambda'$
is the partition given by $(n_0\lambda+\lambda')_i=n_0\lambda_i+\lambda'_i$.

To state the character formula, define the constants
$c_{\lambda,\lambda',n_0}^\nu$ by:
$$
s_\lambda(x_1^{n_0}, x_2^{n_0},...)s_{\lambda'}(x_1,x_2,\ldots)=\sum_\nu c_{\lambda,\lambda',n_0}^\nu s_\nu(x_1,
x_2,...),
$$
where $s_\lambda$ are the Schur polynomials. When we write $c_{\lambda,n_0}^\nu$, we mean $c_{\lambda,\varnothing,n_0}^\nu$.

\begin{theorem}\label{charfor}
In the Grothendieck group $K_0({\mathcal O}_c)$, we have
$$
[L_c(n_0\lambda+\lambda')]=\sum_{\nu: |\nu|=n} c_{\lambda,\lambda',n_0}^\nu M_c(\nu).
$$
In particular, the character of $L_c(n_0\lambda)$ is given by
the formula
$$
\Tr_{L_c(n_0\lambda+\lambda')}(\sigma q^{\mathbf{h}})=
\sum_{\nu: |\nu|=n} c_{\lambda,\lambda',n_0}^\nu
q^{\frac{n-1}{2}-c\kappa(\nu)}\chi_\nu(\sigma)
\det{_{\h}}(1-q\sigma)^{-1},
$$
where $\mathbf{h}$ is the scaling element of the rational Cherednik algebra,
$\kappa(\nu)$ is the content of $\nu$ (see formula (\ref{kappa})),
$\sigma\in S_n$, and $\chi_\nu$ is the character of the $S_n$-module
attached to the partition $\nu$.
\end{theorem}

This theorem implies the following explicit formula
for the character of the quasispherical part of
$L_c(n_0\lambda)$, which provides a connection to the theory
of knot invariants. Namely, let $T(m_0,n_0)$
be the torus knot corresponding to the relatively prime integers
$m_0,n_0$, and let $P_\lambda(T(m_0,n_0))(a,q)$
be its colored HOMFLY polynomial; if $a=q^N$ for large enough
$N$, it is computed as the Reshetikhin-Turaev invariant
for $U_q(\sl_N)$, by coloring the knot with the irreducible
representation of $\sl_N$ of highest weight
$\lambda$. Let
\begin{equation}\label{renorm}
\widetilde{P}_{\lambda}(T(m_0,n_0))(a,q)=
a^{\frac{d}{2}(m_0+n_0-m_0n_0)}\frac{q^{-1/2}-q^{1/2}}{1-a}P_\lambda(T(m_0,n_0))(a,q).
\end{equation}
We will call this polynomial the {\it renormalized} colored HOMFLY polynomial.

Using the formula by M. Rosso and V. Jones \cite{RJ} for this polynomial,
from Theorem \ref{charfor} we obtain:

\begin{corollary}\label{quasis}
$$
\sum_{k=0}^{n-1}(-a)^k\dim_{q}\Hom_{S_n}(\wedge^k\h_n,L_c(n_0\lambda))=q^{-m_0n_0\kappa(\lambda)}\widetilde{P}_{\lambda}(T(m_0,n_0))(a,q),
$$
where $\dim_q(E):=\Tr_E(q^{\mathbf{h}})$.
\end{corollary}

This shows, in particular, that the sum on the left hand side
is symmetric under interchanging $m$ and $n$, which
is not obvious from the representation theoretic viewpoint
(and is explained by Theorem \ref{symme} below).
It also shows that
$\widetilde{P}_\lambda(T(m_0,n_0))(-a,q)$ (and hence $P_\lambda(T(m_0,n_0))(-a,q)$, under a suitable normalization by a power of $-a$) is a
(Laurent) polynomial in $a$ and a power series in $q$
 with nonnegative coefficients,
which is not straightforward from
the knot theory point of view (in fact, the only proof we know uses
Cherednik algebras). Moreover, Theorem \ref{cmth} implies that
the reduced colored HOMFLY invariant
$\widetilde{P}_\lambda(T(m_0,n_0))\cdot \prod_{i=2}^d (1-q^i)$
is a polynomial with nonnegative coefficients.

\subsection{} The character formula of Theorem \ref{charfor} can be used to solve a
problem in Lie theory posed in \cite[Section 9]{CEE}, namely, to
compute the characters of certain equivariant D-modules on the nilpotent
cone of the group $SL_m$.

Let $G$ be a complex simply connected simple algebraic group with
Lie algebra $\g$, ${\mathcal N}\subset \g^*$ be its nilpotent
cone, and $\D_G({\mathcal N})$ be the category of $G$-equivariant
D-modules on ${\mathcal N}$. This category is known to be
semisimple, with simple objects $M_{O,\chi}$ parametrized by
nilpotent orbits $O$ and irreducible representations $\chi$ of
the fundamental group of $O$.  Using Kashiwara's lemma, we can
regard objects of this category as equivariant D-modules on
$\g^*$ supported on ${\mathcal N}$, and then they are precisely
the Fourier transforms of unipotent character D-modules on $\g$
(see \cite[Section 9]{CEE} and references therein, in
particular, \cite{Mir}).

Given $M\in \D_G({\mathcal N})$, regard it as a $D$-module on
$\g^*$, and consider its space of global sections, which we will
denote also by $M$ for brevity. Then $M$ carries an action of $G$
and a commuting action of the Lie algebra $\sl_2$
generated by the Laplace operator and the operator of
multiplication by the squared norm on $\g^*$, see \cite[Section 9]{CEE}. Moreover, it is shown in \cite[Subsection 9.4]{CEE},
that for simple $M$ and for any irreducible $G$-module $V$, the
multiplicity space $\Hom_G(V,M)$ is an $\sl_2$-module
in category $\mathcal O$. Thus, one can define the character of $M$ by the
formula
$$
\Ch_M(q,g)=\Tr_M(gq^{-H})=\sum_{\mu\in P_+}\Tr_{V_\mu}(g)
\psi_{M,\mu}(q)
$$
with
$$
\psi_{M,\mu}(q):=\Tr_{\Hom_G(V_\mu,M)}(q^{-H}), g\in G,
$$
where $H$ is the Cartan element of ${\mathfrak{sl}_2}$,
and $V_\mu$ is the irreducible representation of $G$ with highest weight $\mu$.
This leads naturally to the following interesting problem:

\begin{problem}\label{charco}
Compute the character $\Ch_M$ for every
simple object $M=M_{O,\chi}$ of $\D_G({\mathcal N})$.
\end{problem}

As far as we know, this problem is open for a general $G$.
In \cite{CEE} it was reduced for $G=SL_m$ to the computation of
characters of minimally supported modules for rational Cherednik
algebras, and solved for $G=SL_2$ and in the cuspidal case
for $G=SL_m$ using this reduction.
Thus, using Theorem \ref{charfor}, we now obtain the general
answer for $G=SL_m$.

Let $s\in [0,m-1]$, and $\theta_s$ be the corresponding character
of the center of $SL_m$. Let $d=GCD(m,s)$, $m_0=m/d$ and
$\lambda$ be a partition of $d$. Let $O_\mu$
be the nilpotent orbit corresponding to the partition
$\mu$ of $m$. Consider the orbit
$O_{m_0\lambda}$. This orbit carries a unique 1-dimensional local system
corresponding to the central character $\theta_s$, which we will
denote by ${\mathcal L}_s$.

\begin{theorem}\label{chardmod} 
If $M=M_{O_{m_0\lambda},{\mathcal L}_s}$ then
$$
\Ch_M(q,g)=(1-q)\lim_{n\to \infty}\sum_{\nu: |\nu|=n}
c_{\lambda,n_0}^\nu q^{\frac{n-1}{2}-\frac{m}{n}\kappa(\nu)}
s_\nu(x_1,...,x_m,qx_1,...,qx_m,q^2x_1,...),
$$
where $n=s+km$ with $k\in \ZZ_{\ge 0}$, $n_0=n/d$,
and $x_1,...,x_m$ are the eigenvalues of $g$.
\end{theorem}

Here the limit is understood in the sense of stabilization.
Namely, define an increasing filtration
on $M$ (labeled by $n=s+km$)
by setting $M^{(n)}$ to be the isotypic part of $M$
for the representations $V_\mu$ of $SL_m$
which occur in $V^{\otimes n}$.
Then
$$
\Ch_{M^{(n)}}(q,g)=(1-q)\sum_{\nu: |\nu|=n}
c_{\lambda,n_0}^\nu q^{\frac{n-1}{2}-\frac{m}{n}\kappa(\nu)}
s_\nu(x_1,...,x_m,qx_1,...,qx_m,q^2x_1,...),
$$
and $\Ch_M=\lim_{n\to\infty} \Ch_{M^{(n)}}$.

\subsection{}
The third result is the construction
of the Koszul-BGG complex and the study of its
homology. To define this complex, let us say that an
irreducible $W$-subrepresentation $V\subset M_c(\tau)$ is
singular if it is annihilated by the action of
$\h\subset H_c(W,\h)$. Then, given a singular
subrepresentation $V\subset M_c(\CC)$ for which $\rank(s-1)|_V=1$
for every reflection $s\in S$, we consider the Koszul complex
$K^\bullet(V)$ (in the sense of commutative algebra)
\footnote{Here we do not use, nor claim, that $\wedge^iV$ are simple
$W$-modules, even though this is true if $W$ is a Coxeter group and
$V$ is its reflection representation.}:
$$
M_c(\CC)\leftarrow M_c(V)\leftarrow M_c(\wedge^2V)\leftarrow...
$$
Our third main result is the following theorem.

\begin{theorem}\label{kc}
(i) (Proposition \ref{Dunkl vs Koszul} below)
The complex $K^\bullet$ is, in fact, a complex of $H_c(W,\h)$-modules.

(ii) (Theorem \ref{homolkoz} below)
If $W=S_n$, $c=\frac{m}{n}$, $GCD(m,n)=d<n$, and $V$ is the unique singular copy
of $\h$ in degree $m$ (see \cite{dunkl2,CE,ES}) then the homology $H_i(K^\bullet)$
vanishes if $i\ge d$, and is the irreducible $H_c(W,\h)$-module
$L_c(\lambda_i)$, where $\lambda_i=n_0(d-i,1^i)$, if
$i<d$.
\end{theorem}

The complex $K^\bullet(V)$ is analogous to the BGG resolution in
the representation theory of semisimple Lie algebras, and for
this reason it is called the Koszul-BGG complex.

\begin{remark}
In the case when $\dim V=\dim \h$ and the quotient module $M_c(\CC)/(V)$ is finite dimensional,
the Koszul complex $K^\bullet(V)$ (which is then exact in higher
degrees, i.e., a resolution) was considered in
\cite{BEG},\cite{Go} for real reflection groups, and
in \cite{CE} for complex reflection groups.
\end{remark}

We also show in type $A$
that the differentials in the Koszul-BGG complex
are uniquely determined up to scaling by the condition that they are
nonzero homomorphisms of modules over the Cherednik algebra
(Proposition \ref{uni}).

\subsection{} Finally, our fourth main result concerns symmetry
for Cherednik algebras of type A. Let $\ee_{i,n}$ be the Young
projector in $\CC S_n$ corresponding to the ``hook''
representation $\wedge^i\h_n$ (which is nonzero iff $0\le i\le
n-1$), and let $\overline{\ee}_n=\sum_{i=0}^{n-1}\ee_{i,n}$ be
the idempotent of $\wedge \h_n$.\footnote{When no confusion is possible, we
will often drop the subscript $n$ from the notation for these idempotents.}
The subalgebra $\overline{\ee}_{n}H_c(S_n)\overline{\ee}_{n}$ will be called
the quasispherical subalgebra.

Note that the algebra $H_c(W,\h)$ has the Bernstein
filtration, in which $\deg(\h)=\deg(\h^*)=1$, $\deg(W)=0$.
Also, the module $L_c(\tau)$ is graded by the eigenvalues
of the scaling element $\mathbf{h}\in H_c(W,\h)$,
and has a descending filtration
by the images of the powers of the maximal ideal
${\mathfrak{m}}\subset \CC[\h]^W$
(this filtration is discussed in \cite{GORS}).

It is shown in \cite{Lo} that the algebra $H_c(S_n)$ has a unique
maximal two-sided ideal $J_c(n)$.
Also, for $m\in \Bbb Z_{>0}$ with $GCD(m,n)=d$,
it follows from \cite{CEE}, \cite{BE} (see also \cite{Wi})
that if $\lambda$ is a partition of $d$
then the module $L_{\frac{m}{n}}(n_0\lambda)$ has minimal
support (its support can be explicitly computed from the
construction of \cite{CEE}, and it follows from \cite{BE} that
this support is minimal). This means that
the annihilator of $L_{\frac{m}{n}}(n_0\lambda)$
is the maximal ideal $J_{\frac{m}{n}}(n)$.

Our fourth main result is

\begin{theorem}\label{symme}
(i) (Corollary \ref{co3} below)
Let $\lambda$ be a partition of $d$.
Then for all $i$ there is an isomorphism
of vector spaces
$$
\phi_{n,m,i}: \ee_{i,n}L_{\frac{m}{n}}(n_0\lambda)\cong
\ee_{i,m}L_{\frac{n}{m}}(m_0\lambda)
$$
which preserves the grading and the filtration.
In particular, the two-variable
characters of these two spaces associated
to the grading and the filtration are equal.

(ii) (Theorem \ref{the1} below)
There exists an isomorphism of algebras
$$
{\bold \Phi}_{n,m}:
\overline{\ee}_{n}(H_{\frac{m}{n}}(S_n)/J_{\frac{m}{n}}(n))\overline{\ee}_{n}\to
\overline{\ee}_{m}(H_{\frac{n}{m}}(S_m)/J_{\frac{n}{m}}(m))\overline{\ee}_{m}
$$
preserving the Bernstein filtration and compatible with
$\phi_{n,m,i}$.
\end{theorem}

Note that this implies that if $i\ge \min(n,m)$
then in both parts of Theorem \ref{symme}, the spaces
and the algebras vanish (which is obvious only on one of the
two sides).

The proof of Theorem \ref{symme}
is based on comparing two
constructions of representations of rational Cherednik
algebras of type A from Lie theory (by reduction from
equivariant D-modules) - the Gan-Ginzburg
construction (\cite{GG}) and the construction from
\cite{CEE}. More precisely, we generalize
the Gan-Ginzburg construction to the case of hook
representations, and then the representations
in part (i) of Theorem \ref{symme} turn out to be realized on the
same vector space, yielding a proof of part (i), and the algebras
from part (ii) turn out to act on this space by the same
operators, yielding a proof of part (ii).

\subsection{}
The organization of the paper is as follows.

Section 2 contains the preliminaries.

In Section 3 we prove Theorem \ref{cmth}
(actually, we give two somewhat
different proofs), and give
some applications.

In Section 4, we prove Theorem \ref{charfor}
and Corollary \ref{quasis}, providing a link to knot invariants.

In Section 5, we prove Theorem \ref{chardmod}
on the characters of equivariant D-modules.

In Section 6, we develop the theory of the Koszul-BGG complex,
and prove Theorem \ref{kc}. We give two proofs, based on two
different approaches.

In Section 7, we generalize the Gan-Ginzburg
quantum reduction construction to the ``hook'' case, and
prove Theorem \ref{symme}.

Finally, in Section 8, we study the symmetrized Koszul-BGG complexes,
and give a third proof of Theorem \ref{kc}.

{\bf Acknowledgments.} The work of P. E. was partially supported
by the NSF grant DMS-1000113. The work of I. L.
was partially supported by the NSF grants DMS-1161584 and DMS-0900907.
The work of E. G. was partially supported
by the grants RFBR-10-01-678, NSh-8462.2010.1 and the Simons foundation.
We are very grateful to R. Bezrukavnikov, M. Feigin, S. Gukov, A. Oblomkov,
J. Rasmussen,  V. Shende and M. Stosic
for many useful discussions, without which this paper would not
have appeared.

\section{Preliminaries and notation}

\subsection{Rational Cherednik algebras}

Let $\h$ be a finite dimensional complex vector space, $W\subset GL(\h)$ a finite subgroup,
$S\subset W$ the set of reflections, and $c: S\to \CC$ a conjugation invariant function.
For $s\in S$, let $\alpha_s\in \h^*,\alpha_s^\vee\in \h$ be elements
such that $s\alpha_s=\lambda_s\alpha_s$, $\lambda_s\ne 1$, $s\alpha_s^\vee=\lambda_s^{-1}\alpha_s^\vee$, and
$(\alpha_s,\alpha_s^\vee)=2$.

\begin{definition} The rational Cherednik algebra $H_c(W,\h)$
  attached to $W,\h$
is the quotient of $\CC W\ltimes T(\h\oplus \h^*)$ by the relations
$$
[x,x']=[y,y']=0,\
[y,x]=(x,y)-\sum_{s\in S}c_s(\alpha_s,y)(\alpha_s^\vee,x)s,
$$
where $x,x'\in \h^*,y,y'\in \h.$
\end{definition}

If $W$ is a reflection group and $\h$ is its reflection
representation, we will also use the abbreviated notation
$H_c(W)$ for this algebra.

For a representation
$\tau$ of $W$, let $M_c(\tau)$ be the Verma (or standard) module over $H_c(W,\h)$ induced from $\tau$,
i.e., $M_c(\tau)=H_c(W,\h)\otimes_{\CC W\ltimes S\h}\tau$.
We have a natural isomorphism $M_c(\tau)\cong S\h^*\otimes \tau$
of $\CC W\ltimes S\h^*$-modules, and
$y\in \h$ act by Dunkl operators
$$
D_y=\partial_y-\sum_{s\in S}\frac{\tilde c_s(\alpha_s,y)}{\alpha_s}(1-s)\otimes s,
$$
where $\tilde c_s=2c_s/(1-\lambda_s)$.
The Verma module $M_c(\tau)$ has a unique irreducible quotient
$L_c(\tau)$.

Define the category ${\mathcal O}_c={\mathcal O}_c(W,\h)$ to be the category
of $H_c(W,\h)$-modules which are finitely generated over
$\CC[\h]=S\h^*$, and locally nilpotent under $\h$.
Clearly, $M_c(\tau)$ and $L_c(\tau)$ belong to this category.

The algebra $H_c(W,\h)$ contains the scaling element
$$
\mathbf{h}=\sum_i x_iy_i+\frac{\dim \h}{2}-\sum_{s\in S}\tilde c_s
s,
$$
where $\lbrace{y_i\rbrace}$ is a basis of $\h$, and $\lbrace{ x_i\rbrace}$ the dual
basis of $\h^*$. This element has the property
that $[\mathbf{h},x_i]=x_i$, $[\mathbf{h},y_i]=-y_i$,
$[\mathbf{h},w]=0$ for all $w\in W$.
It is known (\cite{GGOR}) that $\mathbf{h}$ acts locally finitely
on every module from category ${\mathcal O}_c$, and acts
semisimply in every standard and hence
every irreducible module. This implies that
any module in ${\mathcal O}_c$ is naturally graded
by (generalized) eigenvalues of $\mathbf{h}$, and in particular
every standard and irreducible module in this category is
$\CC^\times$-equivariant (we make $\CC^\times$ act trivially on the lowest
weight space).

It is known (\cite{GGOR}) that the category
${\mathcal O}_c$ is a highest weight category (with the ordering by
real parts of eigenvalues of $\mathbf{h}$). In particular,
it has enough projectives, and they admit a
filtration in which successive quotients are standard
modules. Such a filtration is called a standard filtration.

\subsection{Rational Cherednik algebras in type $A$}

Let $W=S_n$ be the symmetric group in $n$
letters, and $\h=\h_n$ be the reflection representation of $W$
(of dimension $n-1$). Then the reflections are just
transpositions, so we have a single conjugacy class.
Thus the parameter $c$ is a single complex number.
The space $\h$ is spanned by $y_1,...,y_n$
permuted by $S_n$, with $\sum_i y_i=0$, and
$\h^*$ is spanned by $x_1,...,x_n$
permuted by $S_n$ with $\sum_i x_i=0$.
The defining relations are:
$$
[x_i,x_j]=[y_i,y_j]=0;
$$
$$
[y_i,x_j]=-\frac{1}{n}+cs_{ij},\ i\ne j;
$$
$$
[y_i,x_i]=1-\frac{1}{n}-c\sum_{j\ne i}s_{ij}.
$$

\subsection{Idempotents}\label{idem}

We need to fix notation for some idempotents in $\CC S_n$.
Denote by $\ee_{i,n}$ (or shortly $\ee_i$ when no confusion is possible) the primitive projector of the
representation $\wedge^i\h_n$ (it is nonzero iff $0\le i\le
n-1$). Denote the symmetrizer $\ee_{0}$
by $\ee$ and the antisymmetrizer $\ee_{n-1}$ by
$\ee_-$. Also, set $e_i=\ee_{i}+\ee_{i-1}$ (the projector of $\wedge^i\CC^n)$, and
$\overline{\ee}=\sum_{i=0}^{n-1}\ee_{i}=\sum_{i\ge 0}e_{2i}=\sum_{i\ge
  0}e_{2i+1}$ (the projector of $\wedge \h_n$).

\subsection{The restriction functors}

The parabolic restriction functors for rational Cherednik algebras
were introduced in \cite{BE}. Namely, given a point $b\in \h$,
denote by $W_b$ the stabilizer of $b$ in $W$. Then one can define
the restriction functor $\Res_b: {\mathcal O}_c(W,\h)\to
{\mathcal O}_c(W_b,\h)$, as follows. Given $M\in {\mathcal
  O}_c(W,\h)$, let $\widehat{M}_b$ be the completion of $M$ at $b$ as a
$\CC[\h]$-module. Then $\widehat{M}_b$ is naturally a module over the
completion of the algebra $H_c(W_b,\h)$ at zero. By taking the
$y$-nilpotent vectors in $\widehat{M}_b$, we get a module
over $H_c(W_b,\h)$, which lies in ${\mathcal O}_c(W_b,\h)$,
and is denoted by $\Res_b(M)$.
\footnote{Note that this definition is slightly different from the one in
\cite{BE}, since here, unlike \cite{BE}, we don't replace $\h$ with $\h/\h^W$.}

The functor $\Res_b$ is exact.
It will be used below in several places.

\subsection{The results of \cite{Wi}}

Let us summarize the results of \cite{Wi} (essentially, Theorem
1.8 and Proposition 3.7 in \cite{Wi}) which will be used several times below.

Let $m,n,d$ be as above. Let $\pi_\mu$ be
the representation of $S_d$ corresponding to
a partition $\mu$ of $d$.
Let $X_{d,n/d}(n)$ be the affine variety which is the union of
all the $S_n$-translates of the subspace in $\CC^n$ defined by
the equations
$\sum_i x_i=0$ and
$$
x_1=\ldots=x_{\frac{n}{d}},
x_{\frac{n}{d}+1}=\ldots=x_{\frac{2n}{d}},\ldots,
x_{(d-1)\frac{n}{d}+1}=\ldots=x_n.
$$
Let $X_{d,n/d}(n)^\circ$ be the open subset of $X_{d,n/d}(n)$
where there are $d$ distinct values of $x_i$.
Then $X_{d,n/d}(n)^\circ/S_n$ is isomorphic to the configuration space
of $d$ unmarked points on the complex plane with barycenter at
the origin, so $\pi_1(X_{d,n/d}(n)^\circ/S_n)=B_d$, the braid group in $d$ strands.

\begin{theorem}\label{wilco} (\cite{Wi})
(i) The minimal support for
modules in the category ${\mathcal{O}}_{\frac{m}{n}}(S_n,\h_n)$
in $\h_n$ is the variety $X_{d,n/d}(n)$.
The minimally supported irreducible modules
are $L_{\frac{m}{n}}(\frac{n}{d}\mu)$, where $\mu$ is a partition
of $d$.

(ii) Let $Y$ be the simple finite dimensional module over
$H_{\frac{m}{n}}(S_{n/d})$.
Given a minimally supported module
$M\in {\mathcal{O}}_{\frac{m}{n}}(S_n,\h_n)$, let ${\mathcal L}_M$ be the local system
on $X_{d,n/d}(n)^\circ/S_n$ whose fiber at a point $b$
is $\Hom_{H_{\frac{m}{n}}(S_{n/d})^{\otimes d}}(Y^{\otimes
  d},{\rm Res}_b(M))$. Then the local system ${\mathcal L}_M$
corresponds to a representation of $B_d$ which factors through
the symmetric group $S_d$. Moreover,
the assignment $M\mapsto {\mathcal L}_M$
is an equivalence of categories between the category
${\mathcal{O}}_{\frac{m}{n}}(S_n,\h_n)_{ms}$
of minimally supported modules in
${\mathcal{O}}_{\frac{m}{n}}(S_n,\h_n)$
and ${\rm Rep}(S_d)$.
In particular, the category
${\mathcal{O}}_{\frac{m}{n}}(S_n,\h_n)_{ms}$
is semisimple.

(iii) The equivalence of (ii)
maps $L_{\frac{m}{n}}(\frac{n}{d}\mu)$
to $\pi_\mu$.
\end{theorem}

\section{Cohen-Macaulayness of modules of minimal support
over rational Cherednik algebras}

The goal of this section is to prove Theorem \ref{cmth}.
We propose two proofs, given in the two subsections below.
In the third subsection we give applications
in the case of the symmetric group.

\subsection{Proof via homological duality for rational Cherednik algebras}

Here is our first proof of Theorem \ref{cmth}.
Its idea was suggested to us by R. Bezrukavnikov.

For brevity let $H:=H_c(W,\h)$ and $R:=\CC[\h]$.
Let $n=\dim \h$.

\begin{proposition}\label{lem}
Let $M$ be a module over $H$ which is free of finite
rank over $R$. Then $\Ext^\bullet_H(M,H)$
lives in dimension $n$ (i.e., $\Ext^i_H(M,H)=0$ unless $i=n$),
and there is a natural isomorphism of $R$-modules
$$\Ext^n_H(M,H)\cong M^*\otimes\wedge^n\h^*,$$
where $M^*:=\Hom_R(M,R)$.
\end{proposition}

\begin{proof}
Consider the Koszul
complex of $M$ as an $H$-module:
\begin{equation}\label{kozco}
M\leftarrow H\otimes_{\CC W\ltimes R} M\leftarrow
H\otimes_{\CC W\ltimes R} (M\otimes \h)\leftarrow...\leftarrow
H\otimes_{\CC W\ltimes R}
(M\otimes \wedge^n\h)\leftarrow 0,
\end{equation}
with the differential defined by
$$
\partial(h\otimes m\otimes b)=\sum_j (hy_j\otimes m-h\otimes
y_jm)\otimes \iota_{x_j}(b), b\in \wedge^i\h,
$$
where $\lbrace{y_j\rbrace}$ is a basis of $\h$,
$\lbrace{ x_j\rbrace} $ the dual basis of
$\h^*$, and $\iota$ is the contraction operator.

\begin{lemma}
This differential is well defined.
\end{lemma}

\begin{proof} If $w\in W$ then
$$
\partial(hw\otimes m\otimes b)=
\sum_j(hwy_j\otimes m-hw\otimes y_jm)\otimes \iota_{x_j}b=
$$
$$
\sum_j(hw(y_j)\otimes wm-h\otimes w(y_j)wm)\otimes w\iota_{x_j}b=
$$
$$
\sum_j(hy_j\otimes wm-h\otimes y_jwm)\otimes \iota_{x_j}wb=
\partial(h\otimes wm\otimes wb).
$$
On the other hand,
$$
\partial(hx_i\otimes m\otimes b)=
\sum_j(hx_iy_j\otimes m-hx_i\otimes y_jm)\otimes \iota_{x_j}b=
$$
$$
\sum_j(hx_iy_j\otimes m-h\otimes x_iy_jm)\otimes \iota_{x_j}b=
$$
$$
\sum_j(hy_jx_i\otimes m-h\otimes y_jx_im)\otimes \iota_{x_j}b+
\sum_{j,s}c_s(x_i,\alpha_s^\vee)(y_j,\alpha_s)(hs\otimes
m-h\otimes sm)\otimes \iota_{x_j}b=
$$
$$
\partial(h\otimes x_im\otimes b)+
\sum_{s}c_s(x_i,\alpha_s^\vee)(hs\otimes
m-h\otimes sm)\otimes \iota_{\alpha_s}b.
$$
Thus, it suffices to show that for each $s$,
$$
(hs\otimes m-h\otimes sm)\otimes \iota_{\alpha_s}b=0.
$$
We have
$$
(hs\otimes m-h\otimes sm)\otimes \iota_{\alpha_s}b=
h\otimes sm\otimes (s-1)\iota_{\alpha_s}b.
$$
So it suffices to show that $(s-1)\iota_{\alpha_s}b=0$.
This is shown in Lemma \ref{vanis} below (in a slightly
more general
situation).
\end{proof}

The complex (\ref{kozco})
is a resolution (i.e., exact in nonzero degrees), since
its associated graded under the $y$-filtration (where $M$ sits in
degree $0$ and $\deg(y_i)=1$) is the usual Koszul complex of $M$ as a $\CC[\h\oplus \h^*]$-module with $y_i|_M=0$
(which is a resolution
since $M$ is free over $R$). Moreover, since $M$ is free over $R$,
this is a projective resolution over $R$, and we can use it
to compute the $\Ext$ groups of $M$ with other modules (in
particular, $H$).
Computing $\Hom$ of this resolution to
$H$, and using that
$$
\Hom_H(H\otimes_{\CC W\ltimes R} M,H)\cong\Hom_{\CC W\ltimes
  R}(M,H)\cong
$$
$$
\Hom_{\CC W\ltimes R}(M,(\CC W\ltimes R)
\otimes S)\cong M^*\otimes S\cong
M^*\otimes_{\CC W\ltimes R} H,
$$
where $S=S\h$,
we see that the dual is a similar Koszul complex
(of right $H$-modules) with $M$ replaced with $M^*\otimes \wedge^n\h^*$,
and the statement follows.
\end{proof}

\begin{corollary}\label{coro}
If $M$ is in the category ${\mathcal O}_c$, then
there is a natural isomorphism of $R$-modules
$\Ext^{n+i}_H(M,H)\cong \Ext^i_R(M,R)\otimes \wedge^n\h^*$.
\end{corollary}

\begin{proof}
It suffices to show that the corollary holds for projectives in category
${\mathcal O}_c$; then the statement follows in general since we can replace
every object by its projective resolution. But projectives
admit a standard filtration, so they are free over $R$, and Proposition \ref{lem}
applies.
\end{proof}

Now we are ready to prove Theorem \ref{cmth}.
Suppose that $M\in {\mathcal O}_c$ has minimal support.

Let $c^*(s)=c(s^{-1})$.
We have an antiisomorphism
$\dagger: H_c(W,\h)\to H_{c^*}(W,\h)$ defined by the formulas:
$x\mapsto x$ for $x\in \h^*$; $y\mapsto -y$ for $y\in \h$;
$s\mapsto s^{-1}$, for $s\in W$.
It is shown in \cite{GGOR}, Proposition 4.10,
that the homological duality functor
$M\mapsto \Ext^*_H(M,H)^\dagger$
defines a derived antiequivalence between the categories
${\mathcal O}_c$ and ${\mathcal O}_{c^*}$.
Moreover, it is clear from Corollary \ref{coro} that this
antiequivalence preserves
supports (in the sense that ${\rm Ext}^i_H(M,H)$ is supported on ${\rm
  supp}(M)$ for all $i$).
This implies that the minimal supports are the same in ${\mathcal O}_c$ and
${\mathcal O}_{c^*}$.

Suppose that the support of $M$ has dimension $d$.
Then $M$ is Cohen-Macaulay of dimension $d$ at generic points of its support,
so for any $i<n-d$, $\Ext^i_R(M,R)$ is supported in dimension $<d$.
On the other hand, by Corollary \ref{coro},
$\Ext^i_R(M,R)\otimes \wedge^n\h^*$ is a right $H_c$-module,
which can be turned into a left $H_{c^*}$-module $U_i$
from category ${\mathcal O}_{c^*}$ by the antiisomorphism $\dagger$.
Since the support of $U_i$ is a proper subvariety of $\supp(M)$,
by the minimality assumption for $\supp(M)$, we must have $U_i=0$.
This implies that $M$ is Cohen-Macaulay.

\subsection{Proof using cohomology with support}

Here is our second proof of Theorem \ref{cmth}.
Let $M\in {\mathcal O}_c$ have minimal support, and
assume that $M$ is not Cohen-Macaulay.
Let $Y$ be the non-Cohen-Macaulay locus of $M$ in $\h$
(which is a Zariski closed subset in $\h$) and let $u$
be the codimension of $Y$ in $\h$. Consider the $i$th
cohomology group $H^i_Y(M)$ of $M$ with support in $Y$. According to \cite[Expose VIII, Cor. 2.3]{Groth}, $H^i_Y(M)$
is a finitely generated $R$-module whenever $i<u$.
Similarly to the proof of Theorem 6.2.5 in \cite{vdb}
one needs to prove that $H^i_Y(M)=0$ for $i<u$. Indeed, the vanishing of $H^i_Y(M)$
implies $\Ext^i_R(\C[Y],M)=0$ for $i<u$ (see \cite[Expose VII, Prop. 1.2]{Groth}).
Thus, $M$ has depth $\ge u$ near a generic point of $Y$.
This contradicts the condition that $Y$ is the non-Cohen-Macaulay locus for $M$.

\begin{lemma}\label{Coh_supp_Cher}
For any closed $W$-stable subvariety $Z\subset \h$, any $H$-module $M$, and any integer $i$
the space $H^i_Z(M)$ admits a natural action of $H$ extending the $R$-action.
\end{lemma}

\begin{proof}
Consider the endofunctor
$\Gamma_Z$ on the category of $R$-modules
such that $\Gamma_Z(M)$
is the set of elements of $M$ set-theoretically supported on $Z$
(i.e., killed by some power of the ideal of $Z$). This functor is
representable by the module $R^{\wedge_Z}$, the completion of $R$ along $Z$. Since $Z$ is $W$-invariant, the space
$H^{\wedge_Z}:=H\otimes_R R^{\wedge_Z}$ has a natural algebra structure and admits algebra embeddings
$R^{\wedge_Z}\hookrightarrow H^{\wedge_Z}$ and $H\hookrightarrow
H^{\wedge_Z}$ (compare to \cite{BE}).
So if $M$ is an $H$-module, we get, using the Frobenius reciprocity, that
$\Gamma_Z(M)=\Hom_H(H^{\wedge_Z},M)$. Now,
$H_Z^i(\bullet)=R^i\Gamma_Z(\bullet)=\Ext^i_{R}(R^{\wedge_Z},\bullet)$,
so, using the Shapiro lemma, for an $H$-module $M$ we have
$H_Z^i(M)=\Ext^i_H(H^{\wedge_Z},M)$ (an isomorphism of $R$-modules).
The right hand side is definitely
an $H$-module.
\end{proof}

Now we can finish the proof of the theorem. We may assume that
$M$ is irreducible. In this case, it is shown by Ginzburg,
\cite{Ginzburg_prim}, that $\supp(M)/W$ is an irreducible subvariety of $\h/W$.
Hence, $\supp(M)$ is an equidimensional variety of the form $W\h_0$, where $\h_0$ is
a subspace of $\h$. The subvariety $Y$ is a proper subvariety in the support, $W$-stable
because $M$ is a $W$-equivariant $R$-module. From Lemma
\ref{Coh_supp_Cher} and the preceding discussion
we deduce that $H_Y^i(M)$ is an $H$-module, finitely generated over $R$
for $i<u$. Also the support of that $R$-module is contained in $Y$. Since $M$, by our assumptions,
has minimal support, we see that $H_Y^i(M)=0$, which gives the
desired contradiction and completes the proof of the theorem.

\begin{remark} In fact, our assumption of the minimality of support
concerned only the category ${\mathcal O}_c$. But this does not create a problem because $H_Y^i(M)$
automatically lies in the category ${\mathcal O}_c$. Indeed, $M$ is a $\C^\times$-equivariant $H$-module.
So $Y$ is $\C^\times$-stable and we have a natural $\C^\times$-action on $H^{\wedge_Y}$
and hence also on ${\rm Ext}^i_H(H^{\wedge_Y},M)=H^i_Y(M)$. So $H^i_Y(M)$ becomes a $\C^\times$-equivariant
$H$-module. Since this module is finitely generated over $R$, it
lies in the category $\mathcal{O}_c$ (see \cite{BE}).
\end{remark}

\subsection{Examples}

Consider now the case of type $A$, i.e. $W=S_n$.
Let $c=\frac{r}{\ell}$, where $r,\ell\in \ZZ_{\ge 1}$, $GCD(r,\ell)=1$.
In this case, we have the following result.

\begin{proposition} (see \cite[Example 3.25]{BE}).
For $i=0,...,[n/\ell]$, let $X_{i,\ell}(n)$ be the
union of all the $S_n$-translates of the
subspace $U_i$ in $\h$ defined by the equations
$$
x_1=...=x_\ell, x_{\ell+1}=....=x_{2\ell},...,
x_{(i-1)\ell+1}=...=x_{i\ell}.
$$
Then $X_{i,\ell}(n)$ occur as supports of modules from
category ${\mathcal O}_c$, and conversely, the support of any
irreducible module in ${\mathcal O}_c$ is $X_{i,\ell}(n)$ for some $i$.
\end{proposition}

In particular, since $X_{0,\ell}(n)\supset X_{1,\ell}(n)\supset...\supset X_{[n/\ell],\ell}(n)$,
we see that the only minimal support is $X_{[n/\ell],\ell}(n)$.
So we get

\begin{corollary}\label{minsupA}
Any module $M\in {\mathcal O}_c$ with support $X_{[n/\ell],\ell}(n)$ is
Cohen-Macaulay as a module over $\CC[x_1,...,x_n]$
(of dimension $n-1-[n/\ell](\ell-1)$).
\end{corollary}

In particular, consider the irreducible module
$L_c(\CC)$. We have the following known proposition:

\begin{proposition}\label{funcalg}
If $c=1/\ell$ then $L_c(\CC)\cong \CC[X_{[n/\ell],\ell}(n)]$.
\end{proposition}

\begin{proof}
Let $I_c$ be the ideal of $X_{[n/\ell],\ell}(n)$. Then
it is easy to see by completing at the generic point
of $X_{[n/\ell],\ell}(n)$ (as in \cite{BE}) that $I_c$
is invariant under the Dunkl operators (i.e., if a polynomial
$f$ vanishes on $X_{[n/\ell],\ell}(n)$, then so does the
polynomial $D_yf$ for any $y\in \h$, because it is so at a
generic point by the completion argument). Thus, $I_c$ is a submodule in
$M_c(\CC)=\CC[\h]$. The quotient $M_c(\Bbb
C)/I_c=\CC[X_{[n/\ell],\ell}(n)]$
is clearly an irreducible module,
since it has minimal support, and its multiplicity at generic
points of the support is $1$. This proves the proposition.
\end{proof}

This leads to the following corollary in commutative algebra,
which appears to be new (note that it is used in the recent paper
\cite{BGS}).

\begin{proposition}\label{cmx}
For any $\ell\le n$, the variety $X_{[n/\ell],\ell}(n)$
is Cohen-Macaulay.
\end{proposition}

\begin{proof} This follows from Proposition \ref{funcalg} and
  Corollary \ref{minsupA}.
\end{proof}

Now consider the situation when $\ell=n/d=n_0$, where
$d\in \ZZ_{\ge 1}$, and $r=m/d$ (so $c=\frac{m}{n}$).
Then the minimal support is $X_{d,n/d}(n)$, of dimension $d-1$,
and by the results of \cite{Wi} (see Theorem \ref{wilco}),
the simple modules in ${\mathcal O}_c$ with this support are precisely
$L_{\frac{m}{n}}(n_0\lambda)$, where $\lambda$ is a partition of $d$ (we identify the irreducible $S_n$-modules
with the corresponding partitions). Let $p_i$
be the $i$-th power sum polynomial. Then $\CC[X_{d,n/d}(n)]$ is finite
over $\CC[p_2,\ldots,p_d]=\CC[X_{d,n/d}(n)]^{S_n}$ by
the Hilbert-Noether
theorem. Thus, we get the
following result:

\begin{proposition}\label{lfree}
For any partition $\lambda$ of $d$,
$L_{\frac{m}{n}}(n_0\lambda)$ is a free finite rank
module over $\CC[p_2,\ldots,p_d]$.
\end{proposition}

This proposition is used below in the proof of Theorem \ref{redcol}.

\subsection{Cohen-Macaulayness of $X_{k,\ell}(n)$.}

Using the above results, one can actually completely
answer the question when the variety $X_{k,\ell}(n)$ is
Cohen-Macaulay.
 Namely, we have

\begin{proposition}
\footnote{We thank A. Polishchuk and S. Sam for discussions
which led to this result, and S. Sam for computation
of the cases $\ell=3,n=6,k=1$ and $\ell=2,n=6,7,k=2$ in Macaulay-2.}
The variety $X_{k,\ell}(n)$ for $k>0,\ell>1$ is Cohen-Macaulay
if and only if either $k=[\frac{n}{\ell}]$ or
$\ell=2$.\end{proposition}

\begin{proof}
Let $r=[\frac{n}{\ell}]$.
By Proposition \ref{cmx}, it suffices to show that

(1) for $\ell\ge 3$,
the variety $X_{k,\ell}(n)$ is not Cohen-Macaulay for
$k<r$; and

(2) $X_{k,2}(n)$ is Cohen-Macaulay for all $k,n$.

Let us prove (1). We first show that
$X_{r-1,\ell}(n)$ is not Cohen-Macaulay for $n$
divisible by $\ell\ge 3$ and $n>\ell$. To this end, consider a
generic point $a$ of $X_{r,\ell}(n)\subset
X_{r-1,\ell}(n)$. We have
$a=(a_1,...,a_1,a_2,...,a_2,...,a_r,...,a_r)\in \Bbb C^n$,
where the $a_i$ are distinct, and each occurs $\ell$ times.
Consider the formal neigborhood of
$a$ in $X_{r-1,\ell}(n)$.
When we pass from $a$ to a generic point of this neighborhood,
the equalities in exactly one group of $\ell$ equal coordinates
in $a$ have to become inequalities. This means that this neighborhood is
a product of a formal polydisc of the appropriate dimension
with the formal neighborhood of zero in
the union of the subspaces $W_1,...,W_r$ of dimension
$\ell-1$ inside $W_1\oplus...\oplus W_r$.
This union is not Cohen-Macaulay by Reisner's theorem (\cite{Re},
Theorem 1) if $\ell>2$.

Now suppose that $X_{k,\ell}(n)$ is not Cohen-Macaulay,
and consider a point in $X_{k,\ell}(n+1)$ of the form
$(b,...,b,-nb)$, where $b\ne 0$. The formal neighborhood of this
point in $X_{k,\ell}(n+1)$ is the product of the formal disc with
the formal neighborhood of zero in $X_{k,\ell}(n)$.
Thus, $X_{k,\ell}(n+1)$ is not Cohen-Macaulay, either.
This completes the proof of (1).

Now let us prove (2).
To this end, note that if $c=1/\ell$,
the defining ideal of $X_{k,\ell}(n)$ is inviariant under Dunkl operators
(this can be checked at the generic point of $X_{k,\ell}(n)$
by using restriction functors from \cite{BE},
as in the proof of Proposition \ref{funcalg}),
so $\Bbb C[X_{k,\ell}(n)]$ is a lowest weight
module over $H_c(S_n)$.
Therefore, specializing to $\ell=2$ and arguing as in the proof of
Theorem \ref{cmth}, we see that if $X_{k,2}(n)$ fails to be Cohen-Macaulay,
its non-Cohen-Macaulay locus has to be
$X_{s,2}(n)$ for some $s>k$ (as the non-Cohen-Macaulay locus
has to be the support of a module over the rational Cherednik
algebra at $c=1/2$, by the proof of Theorem \ref{cmth}).
Consider a generic point
of $X_{s,2}(n)$,
$v=(a_1,a_1,a_2,a_2,...,a_s,a_s,b_1,...,b_{n-2s})$,
where $a_i\ne a_j$, $b_i\ne b_j$ for $i\ne j$ and $a_i\ne
b_j$ for any $i,j$.
Consider the formal neighborhood of $X_{k,2}(n)$ at the point $v$.
When we pass from $v$ to a generic point of this formal
neighborhood, exactly $s-k$ of the equalities inside the $s$
pairs of equal coordinates have to become inequalities.
Considering differences of coordinates in the pairs, we
see that this formal neighborhood is the product
of a formal polydisk of the appropriate dimension
with the formal neighborhood of zero in the union of the
$s-k$-dimensional coordinate subspaces in $\Bbb C^s$.
This union is Cohen-Macaulay by Reisner's theorem
(\cite{Re}, Theorem 1), which is a contradiction.
This means that $X_{k,2}(n)$ is always Cohen-Macualay,
and (2) is proved.
\end{proof}

\section{Characters of minimally supported modules and colored invariants of torus knots}

In this section we first prove the character formula for minimally supported modules (Theorem \ref{charfor}), an then proceed to apply it to knot invariants.
Namely, in \cite[Theorem 3.6]{GORS} it was shown that the HOMFLY
polynomial of the $(m,n)$ torus knot can be realized as a bigraded  character
of
$$\mathcal{H}_{\frac{m}{n}}:=\Hom_{S_n}(\wedge^{\bullet}\h_n,L_{\frac{m}{n}}),$$
where $\h_n$ is the reflection representation of $S_n$ and
$L_{\frac{m}{n}}=L_{\frac{m}{n}}(\CC)$ is the unique finite-dimensional
irreducible representation of the rational Cherednik algebra of type $A_{n-1}$. The space $L_{\frac{m}{n}}$ has a canonical $q$-grading, and the second $a$-grading is defined on $\wedge^{\bullet}\h_n$ as exterior degree.
This section extends this description to the colored HOMFLY invariants of torus knots. As an application, we prove certain positivity results for these polynomials.

\subsection{Proof of Theorem \ref{charfor}}
First let us consider the case when $r=0$. In this case the computation was basically done in the proof
of \cite[Proposition 5.13]{SV}. We reproduce the proof for reader's convenience.

The category $\mathcal{O}_c(S_n,\h_n)$ is equivalent to the category of modules over the $q$-Schur algebra $\mathcal{S}_q(n)$,
where $q:=\exp(\pi\sqrt{-1}/{n_0})$ (where $n_0$ is the denominator of $c$ and $n=dn_0$).
The equivalence was proved by Rouquier, \cite{rouqqsch}, when $n_0>2$ and by the third author,
\cite{mult_cher}, in general. Under this equivalence the Verma module $\Delta(\nu):=M_c(\pi_\nu)$ goes to the Weyl module
$W(\nu)$. Let us represent $\mathcal{S}_q(n)$ as the quotient of $U_q(\mathfrak{gl}_N)$
with some $N\geqslant n$. Then the character of $W(\nu)$ is
the same as the character of the irreducible $\operatorname{GL}_N$-module
$V_{\nu}$ with highest weight $\nu$. The simple module
$L_c(n_0\nu)$ is obtained from $V_{\nu}$ under the pull-back
with respect to the quantum Frobenius. So the character of $L_c(n_0\nu)$ is obtained from that of $V_{\nu}$
by replacing each summand $e^{\mu}$ with $e^{n_0\mu}$.
This implies Theorem \ref{charfor} in the case when $r=0$.

Let us proceed to the case when $r>0$. The proof will follow if we check
that $L(n_0\lambda+\lambda')=\operatorname{Ind}_{S_{dn_0}\times S_r}^{S_n}L(n_0\lambda)\boxtimes L(\lambda')$,
where on the right hand side we have the Bezrukavnikov-Etingof functor associated to the parabolic
subgroup $S_{dn_0}\times S_r\subset S_{dn_0+r}$. First, we will provide an alternative realization
of $L(n_0\lambda+\lambda')$ in terms of $L(n_0\lambda)$ and $\lambda'$.

Recall that the category $\mathcal{O}_c=\bigoplus_{n=0}^{\infty} \mathcal{O}_c(n)$ carries a
categorical Kac-Moody action of $\hat{\mathfrak{sl}}_{n_0}$, see \cite{Shan}. In particular, we
have functors $F_i:\mathcal{O}_c(\bullet)\rightarrow \mathcal{O}_c(\bullet+1), i=0,1,\ldots,n_0-1$.
The functor $F_i$ maps $\Delta(\lambda)$ to a module that admits a filtration with standard quotients,
the quotients that occur are $\Delta(\mu)$ with $\mu$ being a diagram obtained from $\lambda$ adding
a box with content congruent to $i$ modulo $n_0$, each $\Delta(\mu)$ with such $\mu$ occurs with multiplicity $1$.
Also the categorical action is highest weight in the sense of \cite{cryst}.

Choose a Young tableau on $\lambda'$, let $c_1,\ldots,c_r$ be the residues of boxes in the order
they appear in the tableau. We claim that $L(n_0\lambda+\lambda')=F_{c_r}\ldots F_{c_2}F_{c_1}L(n_0\lambda)$.
Indeed, let $\lambda_j$, for $j=1,\ldots,r$, denote the diagram obtained from $n_0\lambda$ by adding
the first $j$ boxes. It is enough to prove that $F_{c_j}L(\lambda_{j-1})=L(\lambda_j)$.
As the third author proved in \cite{cryst}, the crystal of the categorical action on $\mathcal{O}_c$
coincides with the standard crystal of the Fock space. Therefore, the reduced $c_j$-signature of $\lambda_{j-1}$
is a single ``$-$'' and $\tilde{f}_{c_j}\lambda_{j-1}=\lambda_j$. According to \cite[Proposition 5.20]{CR},
$F_{c_j}L(\lambda_{j-1})=L(\lambda_j)$.

So it remains to prove that
\begin{equation}\label{eq:obj_coinc} F_{c_r}\ldots F_{c_2}F_{c_1}L(n_0\lambda)=\operatorname{Ind}_{S_{dn_0}\times S_r}^{S_n}L(n_0\lambda)\boxtimes L(\lambda').\end{equation}
By \cite[Proposition 5.15]{SV}, the functor $\operatorname{Ind}_{S_{dn_0}\times S_\bullet}^{S_dn_0+\bullet}L(n_0\lambda)\boxtimes ?:
\mathcal{O}_c(\bullet)\rightarrow \mathcal{O}_c(dn_0+\bullet)$ commutes with the functors $F_i$. Let us remark that
$L(\lambda')=\Delta(\lambda')$ as the category $\mathcal{O}_c(r)$ is semisimple and so
$L(\lambda')=F_{c_r}\ldots F_{c_1}L(\varnothing)$. This completes the proof of (\ref{eq:obj_coinc}).
$\square$

\subsection{Standard modules revisited}

Let us first recall some facts about representations of $S_d$.
Let $\lambda$ be a Young diagram, $d=|\lambda|$, and let $\pi_{\lambda}$ denote the irreducible representation of $S_d$ corresponding to $\lambda$.
It is well known (e.g. \cite[Exercise 4.17(c)]{fuha}, \cite{Hur})
that the central element $\Omega=\sum_{i<j} (i\ j)\in \CC S_d$ acts in $\pi_{\lambda}$
by the constant
\begin{equation}
\label{kappa}
\kappa(\lambda)=\sum_{(i,j)\in \lambda}(i-j)=\frac{1}{2}\sum_{j}(\lambda_j-2j+1)\lambda_j
\end{equation}
called the {\bf content} of $\lambda$.
Recall that the {\em Frobenius character} of a representation
$\pi$ of $S_d$ is defined by the formula
$$\ch \pi=\frac{1}{d!}\sum_{\sigma\in S_d}\Tr_{\pi}(\sigma)p_1^{k_1(\sigma)}\ldots p_{r}^{k_r(\sigma)},$$
where $p_{i}$ are power sums and $k_i(\sigma)$ is the number of cycles of length $i$ in $\sigma$.
The Frobenius character of $\pi_{\lambda}$ is given by the Schur polynomial $s_{\lambda}$.
The following lemma is obvious (and well known).

\begin{lemma}
\label{char poly}
Let $\h$  be the $(d-1)$-dimensional reflection representation of $S_d$, $\sigma\in S_d$.
Then
\begin{equation}
\det{_{\h}}(1-q\sigma)=\frac{1}{1-q}\prod_{i}(1-q^{i})^{k_i(\sigma)}
\end{equation}
\end{lemma}


\begin{lemma}
The following equation holds:
\begin{equation}
\label{standard module}
\sum_{k=0}^{d-1}(-a)^{k}\dim_{q}\Hom_{S_d}(\wedge^{k}\h,M_{c}(\lambda))=q^{\frac{d-1}{2}-c\kappa(\lambda)}\cdot
\frac{1-q}{1-a}\cdot
\theta_{a,q}(s_{\lambda}),
\end{equation}
where $\theta_{a,q}$ is the character of the ring of symmetric
functions defined by the formula
$\theta_{a,q}(p_{i}):=\frac{1-a^{i}}{1-q^{i}}$.
\end{lemma}

Here, for simplicity we write $M_c(\lambda)$ for $M_c(\pi_\lambda)$.

\begin{proof}
The character of $M_{c}(\lambda)$ was computed in \cite[eq. (1.5)]{BEG}:
$$\Tr_{M_{c}(\lambda)}(\sigma\cdot q^{\mathbf{h}})=\frac{q^{\frac{d-1}{2}-c\kappa(\lambda)} \Tr_{\pi_{\lambda}}(\sigma)}
{\det{_{\h}}(1-q\sigma)}.$$
By orthogonality of characters, we have
$$\dim_{q}\Hom_{S_d}(\wedge^{k}\h,M_{c}(\lambda))=
\frac{1}{d!}\sum_{\sigma \in S_d}\Tr_{M_{c}(\lambda)}(\sigma\cdot q^{\mathbf{h}})\Tr_{\wedge^{k}\h}(\sigma).$$
Since $\sum_{k=0}^{d-1}(-a)^{k}\Tr_{\wedge^{k}\h}(\sigma)={\det{_{\h}}(1-a\sigma)}$, one can rewrite the left hand side of (\ref{standard module}) as
$$\frac{1}{d!}\sum_{\sigma \in S_d}\Tr_{M_{c}(\lambda)}(\sigma\cdot q^{\mathbf{h}}){\det{_{\h}}(1-a\sigma)}=
q^{\frac{d-1}{2}-c\kappa(\lambda)}\frac{1}{d!}\sum_{\sigma \in S_d}\Tr_{\pi_{\lambda}}(\sigma)\frac{\det_{\h}(1-a\sigma)}{\det_{\h}(1-q\sigma)}.$$
By Lemma \ref{char poly} it is equal to
$$\frac{q^{\frac{d-1}{2}-c\kappa(\lambda)}(1-q)}{1-a}\frac{1}{d!}\sum_{\sigma \in S_d}\Tr_{\pi_{\lambda}}(\sigma)\prod_i \left(\frac{1-a^{i}}{1-q^{i}}\right)^{k_i}=$$ $$
\frac{q^{\frac{d-1}{2}-c\kappa(\lambda)}(1-q)}{1-a}\cdot
\theta_{a,q}(\ch \pi_{\lambda}).
$$
\end{proof}


We will need some facts on the colored HOMFLY invariants of knots
in the three-sphere. Given a knot $K$ and a Young diagram
$\lambda$, one can define a rational function
$P_{\lambda}(K)(a,q)$ in variables $a$ and $q$.  We refer the
reader to \cite{amorton},\cite{morton1},\cite{morton2},
\cite{resh} for the precise mathematical definitions. The colored
$\sl_N$ invariant $P_{\lambda,N}(K)(q)$ (which can be defined
using quantum groups, see e.g. \cite{resh}) coincides with the
specialization of the HOMFLY invariant:
$P_{\lambda,N}(K)(q)=P_{\lambda}(K)(q^N,q)$.

For example, the $\sl_N$ invariant of the unknot colored by a diagram $\lambda$ equals to the $q$-character of the corresponding irreducible representation $V_\lambda$, which is equal to
$$P_{N,\lambda}(q)=s_{\lambda}(q^{\frac{1-N}{2}},q^{\frac{2-N}{2}},\ldots,q^{\frac{N-1}{2}})=q^{\frac{1-N}{2}|\lambda|}\prod_{(i,j)\in \lambda}\frac{(1-q^{N+i-j})}{(1-q^{h(i,j)})},$$
where $h(i,j)$ is the hook-length for a box $(i,j)\in \lambda$.

\begin{proposition}
The HOMFLY polynomial of the unknot colored by a Young diagram $\lambda$ equals to
\begin{equation}
\label{plambda}
P_{\lambda}(a,q)=\left(\frac{q}{a}\right)^{\frac{|\lambda|}{2}}\prod_{(i,j)\in
  \lambda}\frac{(1-aq^{i-j})}{(1-q^{h(i,j)})}=\left(\frac{q}{a}\right)^{\frac{|\lambda|}{2}}
\theta_{a,q}(s_{\lambda}).
\end{equation}
\end{proposition}

\begin{proof}
Note that if $\{x_i\}=\{1,q,\ldots,q^{N-1}\}$ then $p_{i}=\frac{1-q^{iN}}{1-q^{i}},$
so $$
s_{\lambda}(1,q,\ldots,q^{N-1})=
\theta_{q^N,q}(s_\lambda).
$$
Since $s_{\lambda}$ is a homogeneous polynomial of degree $|\lambda|$,
we get
$$
P_{N,\lambda}(q)=s_{\lambda}(q^{\frac{1-N}{2}},q^{\frac{2-N}{2}},\ldots,q^{\frac{N-1}{2}})=q^{\frac{1-N}{2}|\lambda|}
\theta_{q^N,q}(s_{\lambda}).
$$
If we replace $q^{N}$ by $a$, we get
$$
P_{\lambda}(a,q)=\left(\frac{q}{a}\right)^{\frac{|\lambda|}{2}}
\theta_{a,q}(s_{\lambda}).
$$
\end{proof}


\begin{corollary}
\label{unknot}
The character
$\sum_{k=0}^{d-1}(-a)^{k}\dim_{q}\Hom_{S_d}\left(\wedge^{k}\h,M_{c}(\lambda)\right)$
of the
hook-labeled isotypic components of $M_{c}(\lambda)$ equals
$q^{-m_0n_0\kappa(\lambda)}\widetilde{P}_{\lambda}(a,q).$
\end{corollary}

\begin{proof}
Follows from equations (\ref{standard module}), (\ref{plambda}), and (\ref{renorm}).
\end{proof}

\begin{remark}
Corollary \ref{unknot} can be explained in more combinatorial way.
We have an isomorphism
$\Hom_{S_d}\left(\wedge^{\bullet}\h, M_{c}(\lambda)\right)=\Hom_{S_d}(\pi_\lambda,\CC[\h]\otimes \wedge^{\bullet}\h).$
The space $\CC[\h]\otimes \wedge^{\bullet}\h$ is naturally bigraded: the $q$-grading is the polynomial degree and the $a$-grading is the degree of an exterior form. It is known that the $q$-grading defined by eigenvalues of $\mathbf{h}$ differs from the polynomial grading by a constant.
The bigraded character of the isotypic component of $\pi_\lambda$ in this space was computed in \cite{KP} (see also \cite{molch}): it is equal to
$$\prod_{(i,j)\in \lambda}\frac{q^{i-1}+aq^{j-1}}{1-q^{h(i,j)}}.$$
It remains to compare this formula with (\ref{plambda}).
\end{remark}

\subsection{Representations with minimal support and torus knots}

Let  $\Lambda$ be the ring of symmetric polynomials in infinitely
many variables.
Let us define the {\em Adams operations} on $\Lambda$ by the formula
$$
\Psi_k(f)(x_1,x_2,\ldots)=f(x_1^{k},x_2^k,\ldots).
$$
Note that $\Psi_{k}:\Lambda\to \Lambda$ are ring homomorphisms and $\Psi_{k}\circ \Psi_{m}=\Psi_{km}.$
We refer the reader to \cite{adams} and references therein for more details on Adams operations.


\begin{definition}
Let us define the coefficients  $c_{\lambda, n_0}^{\mu}$  by the equation
\begin{equation}
\label{defc}
\Psi_{n_0}(s_{\lambda})=\sum_{|\mu|=n_0|\lambda|} c_{\lambda, n_0}^{\mu} s_{\mu}.
\end{equation}
\end{definition}

\begin{theorem}(\cite{RJ}, see also \cite{LZ},\cite{stevan})
\label{homfly}
The HOMFLY polynomial of the $\lambda$-colored $(m_0,n_0)$ torus knot
can be computed using the formula
$$P_{\lambda}(T({m_0,n_0}))=q^{m_0n_0\kappa(\lambda)}a^{\frac{m_0(n_0-1)|\lambda|}{2}}\sum_{\mu} c_{\lambda, n_0}^{\mu}t^{-\frac{m_0}{n_0}\kappa(\mu)}P_{\mu}(a,q),$$
where
$\kappa(\mu)$ is defined by (\ref{kappa}).
\end{theorem}

\noindent{\bf Proof of Corollary \ref{quasis}:}



By Theorem \ref{charfor}
\begin{equation}
\label{min support}
[L_{\frac{m}{n}}(n_0\lambda)]=\sum_{|\mu|=n} c_{\lambda, n_0}^{\mu} [M_{c}(\mu)].
\end{equation}
Consider a linear map $\mathcal{F}: K_0[\mathcal{O}_c(S_n,\h)] \to \CC[[a,q]]$ defined by the equation
$$\mathcal{F}([V])=\sum_{k=0}^{n-1}(-a)^{k}\dim_{q}\Hom_{S_n}\left(\wedge^{k}\h_{n},V\right).$$
By Corollary \ref{unknot}, we have
$$\mathcal{F}(M_{c}(\mu))=q^{-\frac{m}{n}\kappa(\mu)}a^{\frac{n}{2}}\frac{q^{-1/2}-q^{1/2}}{1-a}P_{\mu}(a,q),$$
so by (\ref{min support}) we get
$$\mathcal{F}(L_{\frac{m}{n}}(n_0\lambda))=a^{\frac{n}{2}}\frac{q^{-1/2}-q^{1/2}}{1-a}\sum_{|\mu|=n}c_{\lambda, n_0}^{\mu}q^{-\frac{m}{n}\kappa(\mu)}P_{\mu}(a,q)=$$ $$
q^{-m_0n_0\kappa(\lambda)}a^{\frac{(m_0+n_0-m_0n_0)d}{2}}\frac{q^{-1/2}-q^{1/2}}{1-a}P_{\lambda}(T(m_0,n_0))(a,q).$$
The last equation follows from Theorem \ref{homfly}.$\square$

\begin{corollary}
Consider the space
$$\mathcal{H}_{\frac{m}{n}}(\lambda)=\bigoplus_{k=0}^{n-1}\Hom_{S_{n}}\left(\wedge^{k}\h_{n},L_{\frac{m}{n}}(n_0\lambda)\right).$$
It carries a $q$-grading obtained from the $q$-grading on $L_{\frac{m}{n}}(n_0\lambda)$,
and an $a$-grading by the exterior degree of  $\h_{n}$.  Then the $(a,q)$-bigraded characters of
$\mathcal{H}_{\frac{m}{n}}(\lambda)$ and $\mathcal{H}_{\frac{n}{m}}(\lambda)$ coincide.
\end{corollary}

\begin{proof}
By Corollary \ref{quasis} these characters compute the $\lambda$-colored HOMFLY invariants of the $(m_0,n_0)$ and $(n_0,m_0)$ torus knots respectively. Since the knots are topologically equivalent in $S^3$,
their colored invariants coincide.
\end{proof}

\begin{remark}
Indeed, this coincidence of $(a,q)$-characters also follows from Theorem \ref{symme}, which shows an isomorphism between  $\mathcal{H}_{\frac{m}{n}}(\lambda)$ and $\mathcal{H}_{\frac{n}{m}}(\lambda)$.
\end{remark}

\begin{corollary}
Let $\widetilde{P}_{\lambda}(T(m_0,n_0))(a,q)$ denote, as above, the renormalized $\lambda$-colored {\em unreduced} HOMFLY invariant of the $(m_0,n_0)$ torus knot. Then $\widetilde{P}_{\lambda}(T(m_0,n_0))(-a,q)$ (and hence
$P_\lambda(T(m_0,n_0))(-a,q)$, with an appropriate normalization by a power of $-a$) is a polynomial in $a$ of degree
$\min(m_0,n_0)\cdot |\lambda| -1$ and a power series in $q$ with nonnegative coefficients.
\end{corollary}


\subsection{Invariants of torus links}

Let $m,n$ be two positive integers, $d=GCD(m,n)$. One can consider the $(m,n)$ torus link
with $d$ components and compute its quantum invariants.

\begin{theorem}
The uncolored HOMFLY polynomial of the $(m,n)$ torus link is
given as the
following linear combination of characters of the minimal support representations:
$$
P_{\square}(T(m,n))(a,q)=\sum_{|\lambda|=d}\dim
\pi_{\lambda}\cdot \ch_{a,q} \mathcal{H}_{\frac{m}{n}}(\lambda),
$$
where, as above,
$$\mathcal{H}_{\frac{m}{n}}(\lambda)=\bigoplus_{i=0}^{n-1}\Hom_{S_n}\left(\wedge^{i}\h_n,L_{\frac{m}{n}}\left(\frac{n}{d}\lambda\right)\right).$$
\end{theorem}

\begin{proof}
Let $C$ denote a cycle of length $n$ in $S_n$. By \cite[Theorem 8]{RJ}, the HOMFLY polynomial of $T(m,n)$ can be presented as following:
$$
P_{\square}(T(m,n))(a,q)=\sum_{|\mu|=n}q^{-\frac{m}{n}\kappa(\mu)}\Tr_{\pi_{\mu}}(C^{m})\cdot P_{\mu}(a,q).$$
Since $C^{m}$ is a product of $d$ cycles of length $n_0=n/d$, we have
$$\Tr_{\pi_{\mu}}(C^{m})=\left\langle (p_{n_0})^{d}, s_{\mu}\right\rangle = \left\langle \Psi_{n_0}(p_{1})^{d}, s_{\mu}\right\rangle =$$
$$\left\langle \Psi_{n_0}\sum_{|\lambda|=d}\dim(\pi_{\lambda})s_{\lambda}, s_{\mu}\right\rangle=
\sum_{|\lambda|=d}\dim(\pi_{\lambda})\cdot c_{\lambda,n_0}^{\mu}.$$
It remains to apply Theorem \ref{charfor} and Corollary \ref{quasis}.
\end{proof}

\begin{corollary}
The series $P_{\square}(T(m,n))(-a,q)$
(renormalized by a suitable power of $-a$) has nonnegative coefficients.
\end{corollary}

\begin{remark}
In fact, our argument implies a stronger statement:
the function
$$
(1+a)^{-1}P_{\square}(T(m,n))(-a,q)\cdot \prod_{i=1}^{d}(1-q^{i})
$$
is a polynomial with nonnegative coefficients.
\end{remark}

\begin{example}
Let us compute the HOMFLY polynomial for the Hopf link, i.e., the
$(2,2)$ torus link.
We have $m=n=d=2$, $c=1$, the Verma modules $M_{c=1}(\lambda)$ are irreducible and $\dim \pi_{\lambda}=1$.
One can check that
$$\ch_{a,q}(\mathcal{H}_{1}(2))=q^{-\frac{1}{2}}\frac{1+aq}{1-q^2},\ \ch_{a,q}(\mathcal{H}_{1}(1,1))=q^{\frac{5}{2}}\frac{1+aq^{-1}}{1-q^2}.$$
Therefore
$$\ch_{a,q}(\mathcal{H}_{1}(2))+\ch_{a,q}(\mathcal{H}_{1}(1,1))=q^{-\frac{1}{2}}\frac{1+q^3+aq(1+q)}{1-q^2}.$$
This coincides with the known answer for the HOMFLY polynomial (e.g. \cite[Example 4]{OS}). Note that one can cancel the factor $(1+q)$ to get $q^{-\frac{1}{2}}\frac{1-q+q^2+aq}{1-q},$ but this destroys the non-negativity of the coefficients in the numerator.
\end{example}

\subsection{Duality of characters}

Let $\omega$ be the involution of the ring of symmetric functions $\Lambda$ defined by the equation
$\omega(p_k)=(-1)^{k-1}p_k$. It is well known that $\omega(s_{\lambda})=s_{\lambda^t}.$

\begin{lemma}\label{ome}
Let $f$ be a symmetric function of degree $d$. Then
$$\omega(\Psi_m(f))=(-1)^{(m-1)d}\Psi_m(\omega(f)).$$
\end{lemma}

\begin{proof}
It is sufficient to check the statement for the power sums $f=p_d$:
$$\omega(\Psi_m(p_d))=(-1)^{md-1}p_{md},\ \Psi_m(\omega(p_d))=(-1)^{d-1}p_{md}=(-1)^{(m-1)d}\Psi_m(\omega(p_d)).$$
\end{proof}

\begin{corollary}
The coefficients $c_{\lambda,n_0}^{\nu}$ satisfy the equation
\begin{equation}
\label{c symm}
c_{\lambda^t,n_0}^{\nu^t}=(-1)^{(n_0-1)|\lambda|}c_{\lambda,n_0}^{\nu}.
\end{equation}
\end{corollary}

\begin{proof}
This follows from Lemma \ref{ome}.
\end{proof}

\begin{theorem}
\label{lambda transposed}
The  characters of $L_{c}(n_0\lambda)$ and of
$L_{c}(n_0\lambda^t)$ are related (as rational functions in $q$) by the equation
$$
\Tr_{ L_{c}(n_0\lambda^t)}(\sigma
q^{\mathbf{h}})=(-1)^{|\lambda|-1}\Tr_{L_{c}(n_0\lambda)}(\sigma
q^{-\mathbf{h}}).
$$
\end{theorem}

\begin{proof}
By Theorem \ref{charfor} we have
$$\Tr_{ L_{c}(n_0\lambda^t)}(\sigma q^{\mathbf{h}})=\sum_{\nu}c_{\lambda^t,n_0}^{\nu^t}q^{\frac{n-1}{2}-c\kappa(\nu^t)}\chi_{\nu^t}(\sigma)\det{_{\h}}(1-q\sigma)^{-1}.$$
By (\ref{c symm}) we can rewrite this as
$$\sum_{\nu}(-1)^{(n_0-1)|\lambda|}c_{\lambda,n_0}^{\nu}q^{\frac{n-1}{2}+c\kappa(\nu)}
(-1)^{\sgn \sigma}
\chi_{\nu}(\sigma)\det{_{\h}}(1-q\sigma)^{-1}.$$
On the other hand,
$$\det{_{\h}}(1-q^{-1}\sigma)=(-1)^{n-1-\sgn(\sigma)}q^{-(n-1)}\det{_{\h}}(1-q\sigma),$$
hence
$$\Tr_{ L_{c}(n_0\lambda)}(\sigma q^{-\mathbf{h}})=\sum_{\nu}c_{\lambda,n_0}^{\nu}q^{-\frac{n-1}{2}+c\kappa(\nu)}(-1)^{n-1-\sgn(\sigma)}q^{(n-1)}
\chi_{\nu}(\sigma)\det{_{\h}}(1-q\sigma)^{-1}.$$
\end{proof}

\begin{remark}
The statement of Theorem \ref{lambda transposed} should be understood as follows.
For nontrivial $\lambda$ the representation $L_{c}(n_0\lambda)$ is infinite-dimensional,
so its character is an infinite series in $q$. On the other hand, by Proposition \ref{lfree} this character is
a rational function in $q$ of the form
$$\ch L_{c}(n_0\lambda)=\frac{Q_{c}(n_0\lambda)(q)}{(1-q^2)\cdots (1-q^d)},$$
where $d=|\lambda|$. Theorem \ref{lambda transposed}
provides a functional equation for this rational function which is equivalent to
the functional equation for its numerator (which is a Laurent polynomial with nonnegative coefficients):
$$Q_{c}(n_0\lambda^t)(q)=q^{\frac{d(d+1)}{2}-1}Q_{c}(n_0\lambda)(q^{-1}).$$
\end{remark}

\subsection{Reduced colored invariants}

In knot theory one has a notion of the reduced HOMFLY invariant. By definition,
it is equal to the normalization of the unreduced $\lambda$-colored HOMFLY invariant
of a knot $K$ by the  unreduced $\lambda$-colored HOMFLY invariant of the unknot:
$$P_{\lambda}^{red}(K)=P_{\lambda}(K)/P_{\lambda}(T(1,0)).$$


Motivated by Proposition \ref{lfree}, we define {\em partially reduced} $\lambda$-colored HOMFLY invariants
by the formula
$$\widehat{P}_{\lambda}(K):=P_{\lambda}(K)\cdot \prod_{i=1}^{|\lambda|}(1-q^{i}).$$

\begin{theorem}\label{redcol}
The function $\widehat{P}_{\lambda}(K)$ has the following properties:
\begin{itemize}
\item[a)] $\widehat{P}_{\lambda}(K)$ is a polynomial in $a$ and $q$ for any knot $K$.
\item[b)] For a torus knot $T(m_0,n_0)$, all the coefficients of the polynomial $\widehat{P}_{\lambda}(T(m_0,n_0))(-a,q)$
(after renormalizing by a power of $-a$) are nonnegative.
\item[c)] The sum of the coefficients of $\widehat{P}_{\lambda}(T(m_0,n_0))(-a,q)$ equals to
\begin{equation}
\label{dim of fiber}
\widehat{P}_{\lambda}(T(m_0,n_0))(a=-1,q=1)=(\widehat{P}_{(1)}(T(m_0,n_0))(-1,1))^{d}\cdot
\dim \pi_{\lambda}=
(2\dim \mathcal{H}_{\frac{m_0}{n_0}})^d\cdot \dim \pi_\lambda,
\end{equation}
where $d=|\lambda|$ and $\pi_{\lambda}$ is the irreducible
representation of $S_d$ corresponding to $\lambda$.
\end{itemize}
\end{theorem}

\begin{remark}
In fact, our argument implies a stronger statement that (b):
the function $(1+a)^{-1}(1-q)\widehat{P}_{\lambda}(T(m_0,n_0))(-a,q)$
(after renormalizing by a power of $-a$) is a polynomial with nonnegative coefficients.
\end{remark}

\begin{proof}
a) We have
$$\widehat{P}_{\lambda}(K):=P_{\lambda}(K)\cdot \prod_{i=1}^{d}(1-q^{i})=P_{\lambda}^{red}(K)\cdot P_{\lambda}(T(1,0))\cdot \prod_{i=1}^{d}(1-q^{i}).$$
It is known that the function $P_{\lambda}^{red}(K)$ is a polynomial, and the product of the remaining factors is a polynomial too:
$$P_{\lambda}(T(1,0))\cdot \prod_{i=1}^{d}(1-q^{i})=\prod_{x\in \lambda}(1-aq^{c(x)})\cdot \frac{\prod_{i=1}^{d}(1-q^{i})}{\prod_{x\in \lambda}(1-q^{h(x)})}.$$ Indeed, e.g. by \cite{KP}
$$\dim_{q}(\pi_{\lambda})=\frac{\prod_{i=1}^{d}(1-q^{i})}{\prod_{x\in \lambda}(1-q^{h(x)})}$$
is a polynomial in $q$ with nonnegative coefficients.

b) By Proposition \ref{lfree} the module $L_{\frac{m}{n}}(n_0\lambda)$ is free over $\CC[p_2,\ldots,p_d]$:
$$L_{\frac{m}{n}}(n_0\lambda)=N(n_0,m_0,\lambda)\otimes \CC[p_2,\ldots,p_d],$$
where $N(n_0,m_0,\lambda)$ is a certain finite-dimensional graded $S_{n}$-module.
It remains to note that
\begin{equation}
\label{phat from fiber}
\widehat{P}_{\lambda}(T(m_0,n_0))=\sum_{j=0}^{n-1}(-a)^{j}\dim_{q}\Hom_{S_n}(\wedge^{j}\hh,N(n_0,m_0,\lambda)).
\end{equation}

c) By (\ref{phat from fiber}) the number $\widehat{P}_{\lambda}(T(m_0,n_0))(-1,1)$ equals
$$\widehat{P}_{\lambda}(T(m_0,n_0))(-1,1)=\dim\Hom_{S_n}(\wedge^{\bullet}\hh,N(n_0,m_0,\lambda)).$$
Let us use the restriction functor \cite{BE} to compute this dimension.
The stabilizer of a generic point $b$ in $X_{d,n/d}(n)$ is isomorphic to $(S_{n_0})^{d}$, and
$$
\Res_{(S_{n_0})^d}^{S_{n_0d}}L_{\frac{m_0}{n_0}}(n_0\lambda)\simeq
(L_{\frac{m_0}{n_0}}(\CC))^{\otimes d}\otimes E_{\lambda}
$$
for an $S_{d}$-module $E_{\lambda}$. It follows from \cite{Wi} that $E_{\lambda}=\pi_{\lambda}$.
\end{proof}

\begin{remark}
Equation (\ref{dim of fiber}) is similar to the "power growth" phenomenon
in the colored Khovanov-Rozansky homology, conjectured by S. Gukov and M. Stosic in \cite[Sec. 4.5.2]{GS}.
\end{remark}

\begin{example}
Consider the $(2,3)$ torus knot colored by $\lambda=(2,1)$. One
can check using Theorem \ref{homfly}
that (up to an overall monomial factor)
$$P_{\lambda}^{red}(T(2,3))=1 + 2 q^2 - q^3 + 2 q^4  + 2 q^6 - q^7 + 2 q^8 + q^{10}
-a (1 + 2 q^2 + 3 q^4 + 3 q^6 + 2 q^8 + q^{10}) + $$ $$
 +a^2 (q^2 + q^3 + q^4 + q^6 + q^7 + q^8) - a^3 q^5.$$
 We see that $P_{\lambda}^{red}(T(2,3))(-a,q)$ has two negative coefficients, while all coefficients in
 $$\widehat{P}_{\lambda}(T(2,3))=\frac{(1-a)(1-aq)(1-aq^{-1})(1-q)(1-q^2)(1-q^3)}{(1-q)^2(1-q^3)}P_{\lambda}^{red}(T(2,3))=$$
 $$(1-a)(1-aq)(1-aq^{-1})(1+q)P_{\lambda}^{red}(T(2,3))$$
 have the right sign, and
 $$\widehat{P}_{\lambda}(T(2,3))(-1,1)=8\cdot 54=(2\cdot 3)^{|\lambda|}\cdot 2.$$
 Indeed, $\dim \pi_{\lambda}=2$ and
$\widehat{P}_{(1)}(T(2,3))(a,q)=(1-a)(1+q^2-aq)$,
so $\widehat{P}_{(1)}(T(2,3))(-1,1)=2\cdot 3$.
\end{example}

\section{Characters of equivariant $D$-modules}

\subsection{Proof of Theorem \ref{chardmod}}

Let us use Theorem \ref{charfor} to prove Theorem \ref{chardmod}.
Let $M$ be an $SL_m$-equivariant $D$-module  with central character $\theta_{s}$, $GCD(m,s)=d$,
labeled by the Young diagram $d\lambda$. Let $M^{(n)}$ be the isotypic part of $M$
for the representations $V_\mu$ of $SL_m$
which occur in $V^{\otimes n}$.



Define the automorphism
$\varphi_{\frac{1}{1-q}}$
of the ring of symmetric functions
as follows:
\begin{equation}
\label{defphi}
\varphi_{\frac{1}{1-q}}(p_{k})=\frac{p_{k}}{1-q^{k}}.
\end{equation}
Note that
$$
\varphi_{\frac{1}{1-q}}(f)(x_1,...,x_m,0,0,...)=f(x_1,...,x_m,qx_1,...,qx_m,q^2x_1,...).
$$

By \cite[Theorem 9.8]{CEE},
$({\mathcal{F}}(M)\otimes (\CC^{m})^{n})^{\sl_m}\simeq L_{\frac{m}{n}}(n_0\lambda),$
where ${\mathcal{F}}$ denotes the Fourier transform.
Therefore, since the Fourier transform changes $H$ to $-H$, by the
Schur-Weyl duality, we have ${\rm
  Ch}_{\sl_m}(M^{(n)})=\ch_{S_n}L_{\frac{m}{n}}(n_0\lambda)$
(where the right hand side is the character of the
${\mathfrak{sl}}_m$-module, and the left hand side is the
Frobenius character of the corresponding $S_n$-module).

By the proof of Corollary \ref{quasis}, we have
$$\ch_{S_{n}} L_{\frac{m}{n}}(n_0\lambda)=\sum_{\nu}c_{\lambda,n_0}^{\nu}q^{\frac{n-1}{2}-\frac{m}{n}\kappa(\nu)}\ch_{S_n} M_{\frac{m}{n}}(\nu).$$
Since $M_{\frac{m}{n}}(\nu)=\pi_\nu\otimes \CC[\h_{n}]$, we have
$\ch_{S_n}
M_{\frac{m}{n}}(\nu)=\varphi_{\frac{1}{1-q}}(s_{\nu}).$
This implies Theorem \ref{chardmod}.

\subsection{Character formulas}

Recall that by \cite[Corollary 8.10]{CEE} the character of the $D$-module for $SL_m$ corresponding to the partition $(1^m)$ is given by the equation

\begin{equation}
\label{module1m}
M_{(1^m)}=q^{\frac{m^2-1}{2}}\sum_{\mu} \frac{V_{\mu}P_{\mu}(q)}{(1-q^2)\cdots (1-q^m)},
\end{equation}
where $V_{\mu}$ is the irreducible representation of $SL_m$ labelled by $\mu$ (so that $\mu$ is a partition with at most $m$ parts), and $P_{\mu}(q)$ is the $q$-analogue of the multiplicity of zero weight in $V_{\mu}$ (cf. \cite{Kostant},\cite{Lusztig},\cite{Gupta}).
By the Schur-Weyl duality one can get the character of the
corresponding representation of the Cherednik algebra:
$$\ch L_{\frac{1}{n}}(n (1^m))=q^{\frac{m^2-1}{2}}\sum_{|\mu|=nm}\frac{s_{\mu}P_{\mu}(q)}{(1-q^2)\cdots (1-q^m)},$$
where $s_{\mu}$ denotes the Schur polynomial labeled by $\mu$.

\begin{remark}
The polynomials $P_{\mu}(q)$ are a special case of the
Kostka-Foulkes polynomials. Indeed, the zero weight for $SL_m$ can be represented by the Young diagram $n(1^m)=(n^m)$, and
$$P_{\mu}(q)=K_{\mu,(n^m)}(q),\  \sum_{|\mu|=nm}s_{\mu}P_{\mu}(q)=Q'_{(n^m)}.$$
Here $Q'_{(n^m)}$ is a transformation of the corresponding Hall-Littlewood polynomial $Q_{(n^m)}$
 (\cite{macd}, \cite{DLT}):
$$Q'_{(n^m)}=\varphi_{\frac{1}{1-q}}(Q_{(n^m)}),$$
where the map $\varphi_{\frac{1}{1-q}}$ is defined by $(\ref{defphi})$.
Therefore $$\ch L_{\frac{1}{n}}(n (1^m))=\frac{q^{\frac{m^2-1}{2}}Q'_{(n^m)}}{(1-q^2)\cdots (1-q^m)}.$$
This agrees with the observation in \cite{MMS} that the ``extended HOMFLY polynomial'' of
the $(1,n)$ torus knot colored with the diagram $(1^m)$ is given by the Hall-Littlewood polynomial
$Q_{(n^m)}$.
\end{remark}

We can use Theorem \ref{lambda transposed} to get similar answers for $\lambda=(m)$:

\begin{equation}
\label{modulem}
M_{(m)}=q^{\frac{m-1}{2}}\sum_{\mu} \frac{V_{\mu}P_{\mu}(q^{-1})}{(1-q^2)\cdots (1-q^m)},
\end{equation}

$$\ch L_{\frac{1}{n}}(n (m))=q^{\frac{m-1}{2}}\sum_{|\mu|=nm} \frac{s_{\mu}P_{\mu}(q^{-1})}{(1-q^2)\cdots (1-q^m)}.$$

Finally, similarly to \cite[Theorem 9.18]{CEE} (and using Lemma
\ref{Lem:koz_groth} below) one gets the
following equation in the Grothendieck group of representations of $SL_m$:

\begin{equation}
\label{dmodule euchar}
\sum_{i=0}^{m-1}(-1)^{i}[M_{(m-i,1^{i})}]=\sum_{\mu}\frac{q^{-d(\mu)/2}}{[m]_{q}}[V_{\mu}]\dim_{q}V_{\mu},
\end{equation}
where $[m]_{q}=(1-q^{m})/(1-q)$, $\dim_q(V_\mu)$ is the
(non-symmetrized) $q$-dimension of $V_\mu$, and  $d(\mu)=\deg (\dim_{q}V_{\mu})-m+1.$

\subsection{Equivariant $D$-modules for $SL_2$}

We have $\mu=(l_1,l_2)$ with $l_1+l_2=2n$, and $V_{\mu}\simeq V_{l_1-l_2}$.
Since $l_1$ and $l_2$ have same parity, zero weight is present in $V_{\mu}$ with multiplicity $1$ and
$P_{\mu}(q)=q^{\frac{l_1-l_2}{2}}.$ Therefore by (\ref{module1m}) and (\ref{modulem}) one gets

$$M_{(2)}=\sum_{j=0}^{\infty} V_{2j}\frac{q^{-j+1/2}}{1-q^2},\ M_{(1,1)}=\sum_{j=0}^{\infty} V_{2j}\frac{q^{j+3/2}}{1-q^2},$$

$$\ch L_{\frac{1}{n}}(n (2))=\sum_{l_1+l_2=2n} s_{(l_1,l_2)}\frac{q^{(l_2-l_1+1)/2}}{1-q^2},\ \ch L_{\frac{1}{n}}(n (1,1))=\sum_{l_1+l_2=2n} s_{(l_1,l_2)}\frac{q^{(l_1-l_2+3)/2}}{1-q^2}.$$

Note that

$$[M_{(2)}-M_{(1,1)}]=\sum_{j=0}^{\infty} [V_{2j}]\frac{q^{-j+1/2}-q^{j+3/2}}{1-q^2}=$$ $$\sum_{j=0}^{\infty} \frac{q^{-j+1/2}(1-q^{2j+1})}{1-q^2}[V_{2j}]=\sum_{j=0}^{\infty}\frac{q^{-j+1/2}}{1+q}[V_{2j}]\dim_{q}V_{2j},$$
what agrees with (\ref{dmodule euchar}).

\subsection{Equivariant $D$-modules for $SL_3$}

We have $\mu=(l_1,l_2,l_3)$ with $l_1+l_2+l_3=3n$, and $V_{\mu}\simeq V_{(l_1-l_3,l_2-l_3)}=V_{(\mu_1,\mu_2)}$.
The $q$-dimension of $V_{\mu}$ equals
$$\dim_{q}V_{\mu}=\frac{[\mu_1+2]_{q}[\mu_1-\mu_2+1]_{q}[\mu_2+1]_{q}}{[2]_{q}}.$$
Let $x=\min(\mu_1-\mu_2+1,\mu_2+1),$ then one can check that
$P_{\mu}(q)=q^{\mu_1-x+1}[x]_{q}.$ Therefore by (\ref{module1m}) and (\ref{modulem}) we have

\begin{equation}
\label{dmodules 3 and 111}
M_{(3)}=\sum_{\mu} V_{\mu}\frac{q^{-\mu_1+1}[x]_{q}}{(1-q^2)(1-q^3)},\
M_{(1,1,1)}=\sum_{\mu} V_{\mu}\frac{q^{\mu_1-x+5}[x]_{q}}{(1-q^2)(1-q^3)}.
\end{equation}

To compute the character of $M_{(2,1)}$, note that
$$\dim_{q}V_{\mu}=\frac{[\mu_1+2]_{q}[x]_{q}[\mu_1+2-x]_{q}}{[2]_{q}},$$
hence by (\ref{dmodule euchar}) one gets
$$[M_{(3)}-M_{(2,1)}+M_{(1,1,1)}]=\sum_{\mu}\frac{q^{-\mu_1+1}}{[3]_{q}}[V_{\mu}]\dim_{q}V_{\mu}=$$
$$\sum_{\mu}\frac{q^{-\mu_1+1}[x]_{q}}{(1-q^2)(1-q^3)}(1-q^{\mu_1+2})(1-q^{\mu_1+2-x})[V_{\mu}]=$$
$$\sum_{\mu}\frac{[x]_{q}}{(1-q^2)(1-q^3)}(q^{-\mu_1+1}-q^{3}-q^{3-x}+q^{\mu_1-x+5})[V_{\mu}],$$
therefore
\begin{equation}
\label{dmodule 21}
M_{(2,1)}=\sum_{\mu} V_{\mu}\frac{(q^{3}+q^{3-x})[x]_{q}}{(1-q^2)(1-q^3)}=\sum_{\mu} V_{\mu}\frac{q^{3-x}[2x]_{q}}{(1-q^2)(1-q^3)}.
\end{equation}

The character formulas for the corresponding representations of Cherednik algebras immediately follow
from (\ref{dmodules 3 and 111}) and (\ref{dmodule 21}).

\subsection{The ``small'' part of the equivariant D-modules
with trivial central character}

Consider the special case $s=0$, so that $d=m$.
Let $\lambda$ be a partition of $m$, and consider the equivariant
D-module $M_\lambda$ attached to the nilpotent orbit
corresponding to $\lambda$. Also, let $\mu$ be another partition of
$m$, and $V_\mu$ be the corresponding ``small'' representation
of $SL_m$ (in the sense of A. Broer, \cite{Br}), i.e., one occurring in $(\CC^m)^{\otimes m}$.
Consider the isotypic component of $V_\mu$
in $M_\lambda$, and let us compute its character.
\footnote{The dual representations to these $V_\mu$ are also
small representations in the sense of Broer, but we don't have to
consider them, since each $M_\lambda$ is clearly self-dual as a
graded $SL_m$-module (being stable under outer automorphisms of
$SL_m$), so the characters of the multiplicity spaces for $V_\mu$
and $V_\mu^*$ are the same.}
We may take $n=m$, so $n_0=1$, the Verma modules are irreducible,
and thus the formula of Theorem \ref{chardmod}
is greatly simplified:
$$
{\rm
  Ch}_{M_\lambda^{(m)}}(q,g)=(1-q)q^{\frac{m-1}{2}-\kappa(\lambda)}s_\lambda(x_1,...,x_m,qx_1,...,qx_m,q^2x_1,...).
$$
Thus, to compute the character of the multiplicity space for
$V_\mu$, we need to find the coefficient of $s_\mu(x_1,...,x_m)$
in the decomposition of
$s_\lambda(x_1,...,x_m,qx_1,...,qx_m,q^2x_1,...)$
with respect to Schur functions.

Let $\pi_\lambda$ be the representation of $S_m$ corresponding to
$\lambda$, and $E_{\lambda,\mu}(q)$ be the character of the multiplicity
space of $\pi_\lambda$ in $\pi_\mu\otimes S\CC^m$ (with
grading defined by $\deg(\C^m)=1$). It is easy to see that
the desired coefficient equals $E_{\lambda,\mu}(q)$. Thus, we get
$$
{\rm
  Ch}_{\Hom_{SL_m}(V_\mu,M_\lambda)}(q)=(1-q)q^{\frac{m-1}{2}-\kappa(\lambda)}E_{\lambda,\mu}(q).
$$

In particular, consider the special case $V_\mu=\CC$ (i.e., $\mu=(1^m)$), which gives
the character of the invariants $M_\lambda^{SL_m}$.
We have that $E_{\lambda,(1^m)}(q)$ is the character of the
multiplicity space of $\pi_{\lambda^t}$ in $S\CC^m$,
where $\lambda^t$ is the dual partition to $\lambda$. This
character is well known to equal a power of $q$ times the
reciprocal of the hook polynomial (\cite{macd}):
$$
E_{\lambda,(1^m)}(q)=q^{\sum (i-1)\lambda_i^t}
\prod_{x\in    \lambda}(1-q^{h(x)})^{-1}.
$$
This implies that
$$
{\rm
  Ch}_{M_\lambda^{SL_m}}(q)=q^{\frac{m-1}{2}+\sum_{i\ge 1}
  (i-1)\lambda_i}(1-q)\prod_{x\in    \lambda}(1-q^{h(x)})^{-1}.
$$
Thus, we see that
$$
\sum_\lambda {\rm  Ch}_{M_\lambda^{SL_m}}(q)\pi_\lambda
$$
is $q^{\frac{m-1}{2}}\det_{\h_m} (1-q\sigma)^{-1} $, where
$\det{_{\h_m}}(1-q\sigma)^{-1}$ is
the graded character of $S\h_m$ as an $S_m$-module.

\section{The Koszul-BGG complex for rational Cherednik algebras}

\subsection{The definition of the Koszul-BGG complex}
We keep the notation of Section 2.
Let $V\subset S\h^*=M_c(\CC)$ be a representation of $W$
where $\rank(1-s)\le 1$ for all $s\in S$ (this includes, for
instance, the Galois twists of the reflection representation for
complex reflection groups).
Assume that $V$ is singular, i.e., the Dunkl operators act on $V$ by zero.
We will attach to $V$
a complex of $H_c(W,\h)$-modules from category ${\mathcal O}_c$, called the Koszul-BGG complex.

For $s\in S$ let $0\ne \beta_s^*\in V^*$ be such that $s$ acts trivially on $\Ker \beta_s^*$, and
let $s\beta_s^*=\mu_s\beta_s^*$. Let $\beta_s\in V$ be such that
$s\beta_s=\mu_s^{-1}\beta_s$ and $(\beta_s,\beta_s^*)=1$.

We have the Koszul complex $K^\bullet(V)$, where
$K^i(V)=S\h^*\otimes \wedge^iV=M_c(\wedge^iV)$.

\begin{proposition}
\label{Dunkl vs Koszul}
The complex $K^\bullet(V)$:
$$
M_c(\CC)\leftarrow M_c(V)\leftarrow M_c(\wedge^2V)\leftarrow...
$$
is a complex of $H_c(W,\h)$-modules.
\end{proposition}

\begin{proof} By definition, the Koszul complex is a complex
of $\CC W\ltimes S\h^*$-modules. So we need to show that the
Koszul differential $d$ commutes with the Dunkl operators.
Let $f\in S\h^*$, $u\in \wedge^mV$. We have
$$
d(f\otimes u)=\sum_j v_jf\otimes \iota_{v_j^*}u,
$$
where $\lbrace{v_j\rbrace}$ is a basis of $V$ and $\lbrace{ v_j^*\rbrace}$ the dual basis of $V^*$.
Thus,
$$
[D_y,d](f\otimes v)=\sum_j \partial_y(v_j)f\otimes \iota_{v_j^*}u-
\sum_{s\in S}\frac{\tilde c_s(\alpha_s,y)}{\alpha_s}(1-s)(v_j)sf\otimes s\iota_{v_j^*}u.
$$
Since $D_yv_j=0$, this equals to
$$
\sum_{s\in S}\frac{\tilde c_s(\alpha_s,y)}{\alpha_s}(1-s)(v_j)sf\otimes (1-s)\iota_{v_j^*}u,
$$
so our job is to show that the expression
$$
T(u):=\sum_{s\in S}\frac{\tilde c_s(\alpha_s,y)}{\alpha_s}(1-s)(v_j)s\otimes (1-s)\iota_{v_j^*}u.
$$
vanishes. To this end, we note that
$$
\sum_j (1-s)(v_j)\otimes v_j^*=
\sum_j (1-\mu_s)(\beta_s^*,v_j)\beta_s\otimes v_j^*=
(1-\mu_s)\beta_s\otimes \beta_s^*.
$$
This implies that
$$
T(u)=\sum_{s\in S}\tilde c_s(1-\mu_s)(\alpha_s,y)\frac{\beta_s}{\alpha_s}s\otimes (1-s)\iota_{\beta_s^*}u.
$$
So the result follows from the following lemma.

\begin{lemma}\label{vanis}
For any $u\in \wedge^mV$ and $s\in S$, $(1-s)\iota_{\beta_s^*}u=0$.
\end{lemma}

\begin{proof} Let $u=u_1\wedge...\wedge u_m$. If $m=1$, there is nothing to prove, so assume that $m\ge 2$. Then
$$
(1-s)\iota_{\beta_s^*}u=
$$
$$
\Alt_m\sum_{j=1}^{m-1}(u_m,\beta_s^*)
u_1\otimes...\otimes u_{j-1}\otimes (1-s)(u_j)\otimes
su_{j+1}\otimes...\otimes su_{m-1}=
$$
$$
\Alt_m\sum_{j=1}^{m-1}(u_m,\beta_s^*)
u_1\otimes...\otimes u_{j-1}\otimes (1-\mu_s)(u_j,\beta_s^*)\beta_s\otimes
su_{j+1}\otimes...\otimes su_{m-1}.
$$
This is zero, since we have skew-symmetrization with respect to
$j$ and $m$, which now occur symmetrically.
\end{proof}

The proposition is proved.
\end{proof}

In the special case when $W$ is a real irreducible
reflection group, $c=m/h$, where $h$ is the Coxeter number of $W$, $m\in \ZZ_{\ge 1}$, $GCD(m,h)=1$, and
$V=\h$ is the reflection representation, this complex was studied in \cite{BEG} and \cite{Go}. In this case,
this complex is actually a resolution of a finite dimensional $H_c(W,\h)$-module of dimension $m^r$, where $r$ is
the rank of $W$. Later it was studied in \cite{CE} in the case when the representation $S\h^*/(V)$
is finite dimensional (it follows from the fact that the expression in
Theorem 2.3(iii) in \cite{CE} is a polynomial that in this case
$S$ acts by reflections in $V$). This resolution is analogous
to the BGG resolution in Lie theory, so it was called the BGG resolution.
Thus we will call the complex $K^\bullet$ the {\bf Koszul-BGG complex}.

\subsection{The Koszul-BGG complex for $W=S_n$}
It follows from the paper \cite{FS} that
if $W$ is an irreducible real reflection group, and $c=m/h$, where $m=d-1+\ell h$, $\ell\in \ZZ_{\ge 0}$,
and $d$ is a degree of $W$, then there is a singular copy of $V=\h$ in degree $m$ of $S\h^*$, so
the complex $K^\bullet(V)$ is nontrivial.
A similar result for cyclotomic wreath product groups
$G(l,r,n)$ follows from the paper \cite{DO}
(see  also \cite{CE} and \cite{str} for the case $r=1$).

In particular, if $W=S_n$, and $\h$ is the reflection representation, the Koszul-BGG complex has a nonzero differential
for any $c=\frac{m}{n}$, where $m$ is not divisible by $n$. In
this case, it was shown by Dunkl, \cite{dunkl2} (see also \cite{CE}) that
the singular representation $V$ is spanned by partial derivatives of the polynomial
\begin{equation}
\label{potential for singular polys}
F_{m,n}(x_1,...,x_n):=Res_{u=\infty}((u-x_1)...(u-x_n))^{\frac{m}{n}}du.
\end{equation}
Note that this works also if $m$ is divisible by $n$, except that in this case
$F_{m,n}=0$, so the differential in the corresponding complex is zero.
Thus, we have defined a complex for every $n\ge 1$ and $m\ge 1$.
Let us denote this complex by $K_{m,n}^\bullet$.

Now let $m,n$ be positive integers, and $d=GCD(m,n)$.
Write $m=m_0d$, $n=n_0d$. Our main result about the Koszul-BGG complex
for $S_n$ is the following theorem.

\begin{theorem}\label{homolkoz}
(i) The homology $H_i(K_{m,n}^\bullet)$ is nonzero if and only
if $0\le i\le d-1$.

(ii) If $0\le i\le d-1$ then $H_i(K_{m,n}^\bullet)$ is the irreducible representation
$L_{\frac{m}{n}}(\lambda_i)$ of the rational Cherednik algebra $H_c(S_n)$,
where
$
\lambda_i=n_0(d-i,1^i).
$
\end{theorem}

Two proofs of this theorem are contained in the next
three subsections, and a third one in Subsection
\ref{thirdproof}; these proofs are based on different ideas,
so we present all three of them.

\subsection{Proof of Theorem \ref{homolkoz}}

\begin{lemma}\label{comal}
Let $V$ be a finite dimensional subspace of $R:=\CC[x_1,...,x_N]$.
Assume that the zero set $Z(V)$ of $V$ in $\CC^N$ has dimension $k<N$.
Then there are polynomials $f_1,...,f_{N-k}\in V$, which form a regular
sequence.
\end{lemma}

\begin{proof}
We prove by induction in $i$
(for $i\le N-k$) that one can choose a regular sequence $f_1,...,f_i\in V$. The base of induction
is obvious. To make the step of induction, suppose that for some $i<N-k$, the polynomials $f_1,...,f_i$ have been chosen.
Then the zero set $Z_i$ of $f_1,...,f_i$ has pure codimension
$i$. Since the zero set $Z$ of $V$ has codimension $>i$,
none of the components of $Z_i$ is contained in $Z$, so
a generic element $f_{i+1}$ of $V$ does not vanish identically
on any of these components; this completes the step of induction.
\end{proof}

The vanishing of $H_i$ for $i\ge d$ follows from the standard properties of the Koszul complex.
Namely, we know
that the module $H_0=L_{\frac{m}{n}}(\CC)$ has
minimal support (by \ref{wilco}; see also \cite{BE}),
so this support is of dimension $d-1$.
By Lemma \ref{comal}, this means that there exists a basis $f_1,...,f_{n-1}$ of the space $V$ spanned by the partial derivatives of $F_{m,n}$ such that
$f_1,...,f_{n-d}$ is a regular sequence. Define a grading on $K^\bullet_{m,n}$ by the number of $f_i$
in the wedge part with $i>n-d$. Then the differential preserves the filtration
defined by this grading, and the associated graded complex is of the form
$K^\bullet(f_1,...,f_{n-d})\otimes \wedge^\bullet(f_{n-d+1},...,f_{n-1})$
(with the Koszul differential of the first factor).
The first factor is acyclic in positive degrees, so this complex has no homology in degrees $\ge d$.
Hence the same is true for the filtered complex.

Thus, we just need to prove part (ii).

To this end, note that the support of $H_0$
is the union $X_{d,n/d}(n)$ of all the images of
the subspace defined by the equations $x_i=x_j$
when $i-j=0$ modulo $d$ under permutations.
By Theorem \ref{wilco}, this is the minimal support of modules in
category $\mathcal{O}$ for $H_{\frac{m}{n}}(S_n)$.
By the theory of Koszul complexes,
this implies that all the homology
modules $H_i$ are supported on $X_{d,n/d}(n)$,
i.e. have minimal support. By the results of Wilcox,
\cite[Theorem 1.8, Proposition 3.7]{Wi} (see Theorem
\ref{wilco}), the category of such modules is equivalent to the category of
representations of $S_d$, by considering restriction $\Res_b$ to the open stratum of $X_{d,n/d}(n)$ and looking at the monodromy of the
resulting local system. Namely,
this equivalence sends a representation of $S_d$
corresponding to the Young diagram $\mu$ to the representation
$L_{\frac{m}{n}}(n_0\mu)$ over the rational Cherednik algebra
$H_{\frac{m}{n}}(S_n)$ with minimal support (see Theorem \ref{wilco}).
Moreover, this equivalence is compatible with
restrictions to points of $X_{d,n/d}(n)/S_n$ (i.e. restriction from Cherednik
algebra to its parabolic subalgebras corresponds under this
equivalence to the restriction
from the symmetric group $S_d$ to its parabolic subgroups).

Let $\wedge^i_d$ be the $i$-th exterior power
of the reflection representation of $S_d$.
We will need the following simple lemma.

\begin{lemma}\label{symg}
Suppose that $\pi$ is a representation of $S_d$
such that for any $0<k<d$,
$$
\pi|_{S_k\times S_{d-k}}=\oplus_{r-1\le i+j\le
  r}\wedge^i_k\otimes \wedge^j_{d-k},
$$
and $\pi^{S_d}=0$.
Then $\pi=\wedge^r_d$.
\end{lemma}

\begin{proof} Clearly, $\wedge_d^r$ satisfies the condition.
Hence, the character of the difference $\pi-\wedge_d^r$ has zero
restriction to the subgroups $S_k\times S_{d-k}$, i.e., vanishes
on all non-cyclic permutations in $S_d$. Thus, it is an integer multiple
of the virtual character $\chi(g)=\sum_{\lambda} \Tr_{\pi_\lambda}(g)\pi_\lambda$, where
$g$ is a cyclic permutation in $S_d$. This virtual character
involves a copy of the trivial representation.
So $\pi=\wedge_d^r$.
\end{proof}

Now we prove part (ii) of the theorem.
Our job is to show that $H_i=L_{\frac{m}{n}}(n_0(d-i,1^i))$ 
for $0<i<d$ (we already know that $H_0=L_{\frac{m}{n}}((n))$).

We will use the following proposition.
Let $0<k<d$. Let $b$ be a point
in $\h$ with coordinates $x_i=(d-k)z$ for
$i\le n_0k$, and $x_i=-kz$ for $i>n_0k$, for some $z\ne 0$.

\begin{proposition}\label{rest}
We have an isomorphism of complexes of ${\CC}(S_{n_0k}\times
S_{n_0(d-k)})\ltimes \CC[\h]$-modules
\begin{equation}\label{restr}
\Res_b(K^\bullet_{m,n})\cong K_{m_0k,n_0k}^\bullet\otimes K_{m_0(d-k),n_0(d-k)}^\bullet\otimes \Omega^\bullet,
\end{equation}
where $\Omega^\bullet$ is the two-step complex $\CC[t]\leftarrow \CC[t]$ with
the zero differential.
\end{proposition}

\begin{proof}
First of all, if $f_1,...,f_r\in R=\CC[x_1,...,x_n]$ and $K^\bullet(f_1,...,f_r,R)$ is the corresponding Koszul complex,
then by definition, the completion $\hat
K_b^\bullet(f_1,...,f_r,R)$ of $K^\bullet(f_1,...,f_r,R)$ at any point $b\in \CC^n$ is naturally isomorphic
to $K^\bullet(f_1,...,f_r,\hat R_b)$, where $\hat R_b$ is the completion of $R$ at $b$.

Next, suppose $\bar f_i\in R$, $1\le
i\le r$, are linearly independent quasihomogeneous polynomials of
the same degree $D$ (i.e., homogeneous polynomials of degree $D$ for a
grading in which $\deg(x_j)=d_j$ for some positive
integers $d_j$), and assume that $f_i=\bar f_i+\text{higher degree terms}\in \hat
R=\CC[[x_1,...,x_n]]$ are deformations of these polynomials.
Also let $g_p\in \hat R$, $p=1,...,s$, be elements whose lowest
degree is $>D$.

\begin{lemma}\label{le1} Assume that $\bar f_i$ generate the same ideal in $\hat R$ as
$f_i,g_p$. Then
$$
K^\bullet(f_1,...,f_r,g_1,...,g_s,\hat R)\cong K^\bullet(\bar
f_1,...,\bar f_r,\hat R)\otimes \wedge(\xi_1,...,\xi_s),
$$
as complexes of $\hat R$-modules,
where $\partial\xi_i=0$.
\end{lemma}

\begin{proof}
We can choose elements $a_{ij}\in \hat R$ such that
$$
f_i=\sum_j a_{ij}\bar f_j.
$$
Then, since $\bar f_j$ have the same homogeneity degree,
we have
$$
\bar f_i=\sum_j a_{ij}(0)\bar f_j.
$$
This implies that $a_{ij}(0)=\delta_{ij}$
and hence $(a_{ij})$
is invertible. Also, we have the matrix $(c_{pj})$, $c_{pj}\in
\hat R$ such that $g_p=\sum_j c_{pj}\bar f_j$.

We claim that the matrices $(a_{ij}),(c_{pj})$
define the desired isomorphism
$$
\theta: K(f_1,...,f_r,g_1,...,g_s,\hat R)\cong K(\bar f_1,...,\bar
f_r,\hat R)\otimes \wedge(\xi_1,...,\xi_s).
$$
Namely, let $\eta_1,...,\eta_r$ be the odd generators
of $K(\bar f_1,...,\bar f_r,\hat R)$ over $\hat R$, and let
$\eta_1',...,\eta_r',\xi_1',...\xi_s'$
be the odd generators
of $K(f_1,...,f_r,g_1,...,g_s,\hat R)$ over $\hat R$
(so that
$$
K(\bar f_1,...,\bar f_r,\hat R)=\hat R\otimes
\wedge(\eta_1,...,\eta_r),\
K(f_1,...,f_r,g_1,...,g_s,\hat R)=\hat R\otimes
\wedge(\eta_1',...,\eta_r',\xi_1',...,\xi_s'),
$$
and  $\partial \eta_i=\bar f_i$, $\partial \eta_i'=f_i$,
$\partial \xi_p'=g_p$). Then $\theta$ is defined by the formula
$$
\theta(\eta_i')=\sum_j a_{ij}\eta_j, \theta(\xi_p')=\xi_p+\sum_j
c_{pj}\eta_j.
$$
This proves the lemma.
\end{proof}

Now, consider the singular polynomials $f_i$, $i=1,...,n$,
generating the Koszul complex $K_{m,n}^\bullet$.
As explained above, $f_i=\partial_i F_{m,n}$, where
$$
F_{m,n}=\frac{1}{2\pi \ii}\int_\gamma ((u-x_1)...(u-x_n))^{\frac{m}{n}}du,
$$
where the integration is over a large enough circle $\gamma$
in the counterclockwise direction.
This polynomial has degree $m+1$.
Let us consider the completion at the point
$b$, and introduce new variables:
$$
t=\frac{1}{n_0}\sum_{1\le i\le n_0k}x_i-k(d-k)z;\
x_i'=x_i-\frac{t}{k}-(d-k)z, i\le n_0k;
$$
$$
x_i''=x_i+\frac{t}{d-k}+kz, i> n_0k
$$
(thus, $\sum_i x_i'=\sum_j x_j''=0$).

\begin{lemma}\label{le2} We have
$$
F_{m,n}(x)=C'F_{m_0k,n_0k}(x')+C''F_{m_0(d-k),n_0(d-k)}(x'')+\text{higher
  terms}, C',C''\in \CC^\times,
$$
where higher terms are of two kinds:

(1) degree $s'\ge n_0k+1$ in $x'$ and degree $s''$ in $x'',t$
with $s'+s''-(n_0k+1)>0$;

(2) degree $s''\ge n_0(d-k)+1$ in $x''$ and degree $s'$ in $x',t$
with $s'+s''-(n_0(d-k)+1)>0$.
\end{lemma}

\begin{proof} Consider a point $x$ close to $b$. Then
$x_i$ cluster around $z(d-k)$ for $i\le n_0k$, and around $-zk$
for $i>n_0k$. So by the Cauchy integral formula the contour integral defining
$F_{m,n}$ can be represented as a sum
of integrals over two contours going around each of the two
clusters (note that the integrand is single-valued on these
contours). Shifting the integration variable in each of the integrals
to make the contours go around the origin, we get
$$
F_{m,n}=\frac{1}{2\pi \ii}(zd)^{m_0(d-k)}\oint
\prod_{i=1}^{n_0k}(v-x_i')^{m_0/n_0}
\prod_{i=n_0k+1}^n\left(1+\frac{v-x_i''}{zd}+\frac{t}{zk(d-k)}\right)^{m_0/n_0}dv+
$$
$$
\frac{1}{2\pi \ii}(-zd)^{m_0k}\oint
\prod_{i=1}^{n_0k}\left(1-\frac{v-x_i'}{zd}+\frac{t}{zk(d-k)}\right)^{m_0/n_0}
\prod_{i=n_0k+1}^n(v-x_i'')^{m_0/n_0}dv.
$$
This implies the lemma, with $C'=(zd)^{m_0(d-k)}$ and $C''=(-zd)^{m_0k}$ (the two terms in the formula come
from the two resulting integrals, and the form of the higher
terms is clear from the form of these integrals; we just expand
the expressions of the form $(1+u)^{m_0/n_0}$ appearing in the
integrals in a Taylor series with respect to $u$).
\end{proof}

Now let $\deg(x_i')=d-k$, $\deg(x_i'')=k$. Then the polynomials
$\bar f_i=\frac{\partial}{\partial x_i'} F_{m_0k,n_0k}(x')$ for
$i\le n_0k$, and $\bar f_i=\frac{\partial}{\partial x_i''} F_{m_0(d-k),n_0(d-k)}(x'')$ for
$i>n_0k$ are quasihomogeneous of degree $n_0k(d-k)$. Note that
$\sum_{i\le n_0k} \bar f_i=\sum_{i>n_0k}\bar f_i=0$.

So, it suffices to show that $f_i$  generate the same ideal in
$\CC[[x_1', \ldots, x_{n_0k}', x_{n_0k+1}'', \ldots, x_n'',t]]$ as
$\bar f_i$, $i=1,...,n$. Then by virtue of Lemma \ref{le2} the Proposition will follow by applying Lemma \ref{le1}
to $f_i$ for $i\ne n_0k,n$ and $g_1=\sum_{i\le n_0k}f_i$
(as power series in $x',x'',t$).

To this end, note that $f_i$, $1\le i\le n$,
generate an ideal $I$ that is invariant
under the Dunkl operators $D_j$. Let us expand
the Dunkl operators at $b$, with respect to the coordinates
$x_i',x_i'',t$. These ``formal'' Dunkl operators
are non-homogeneous in the variables $x_i',x_i'',t$, and
we have $D_i=\bar D_i+R_i$, where
$\bar D_i$ are the homogeneous parts
(of degree $-1$ in the grading where the degrees of the
$x_i',x_i'',t$ are $1$), and $R_i$ are the regular parts.
Clearly, $I$ is invariant under $R_i$, so it is invariant
under $\bar D_i$, which are the Dunkl operators
of the parabolic subgroup $W_b=S_{n_0k}\times S_{n_0(d-k)}$ stabilizing the point
$b$. Thus, $I$ corresponds to a proper submodule $J$ in the polynomial
module $M_c(\CC)$ over the parabolic Cherednik algebra
$H_c(S_{n_0k}\times S_{n_0(d-k)},\h'\oplus \h'')$,
where $\h',\h''$ are the reflection representations of $S_{n_0k}$
and $S_{n_0(d-k)}$, respectively (here we use the fact that the
restriction of the polynomial module is the polynomial module
over the parabolic subalgebra, which follows from the definition
of the restriction functors in \cite{BE}).
Since $M_c(\CC)/J$ has $d-1$-dimensional (i.e., minimal) support,
$J$ must be the image of $M_c(\h')\boxtimes M_c(\CC')\oplus
M_c(\CC'')\boxtimes M_c(\h'')$ in $M_c(\CC')\boxtimes M_c(\CC'')$,
where $\CC',\CC''$ are the trivial representations of $S_{n_0k}$ and $S_{n_0(d-k)}$,
respectively. This means that $I$ is generated by the elements $\bar f_i$.

Thus, by Lemma \ref{le1}, we have the required isomorphism of complexes
of $\CC[\h]$-modules. Moreover, since all the constructions in
the proof are equivariant under the group \linebreak
$S_{n_0k}\times S_{n_0(d-k)}$,
it follows from the proof of Lemma \ref{le1} that
this is actually an isomorphism of
$\CC(S_{n_0k}\times S_{n_0(d-k)})\ltimes \CC[\h]$-modules, as desired.

The proposition is proved.
\end{proof}

\begin{corollary}\label{chermor}
The two complexes in Proposition \ref{rest}
have isomorphic cohomology groups, as $H_{\frac{m}{n}}(S_{n_0k}\times S_{n_0(d-k)},\h)$-modules.
\end{corollary}

\begin{proof}
By Proposition \ref{rest}, the cohomology groups of the two
complexes are isomorphic as $\CC(S_{n_0k}\times
S_{n_0(d-k)})\ltimes \CC[\h]$-modules.
Also, we know that these cohomology groups are semisimple
modules over $H_{\frac{m}{n}}(S_{n_0k}\times S_{n_0(d-k)},\h)$, by the result
of \cite{Wi} (see Theorem \ref{wilco}), since they have minimal support, and the category
of minimally supported modules is semisimple.
Hence, the corollary is a consequence of the following lemma.
\end{proof}

\begin{lemma}\label{chara}
A semisimple object in
${\mathcal  O}_c(W,\h)$ is uniquely determined, up to an isomorphism,
by its structure of a ${\CC}W\ltimes \CC[\h]$-module.
\end{lemma}

\begin{proof}
Let ${\mathfrak{m}}$ be the augmentation ideal
of $\CC[\h]$ (generated by $\h^*$). If
$N=\oplus_{\tau\in \Irrep\ W}m_\tau L_c(\tau)$,
then $N/{\mathfrak{m}}N=\oplus_{\tau\in \Irrep(W)}m_\tau \tau$
as a $W$-module. So $N$ is determined by its structure of
a ${\CC}W\ltimes \CC[\h]$-module, as desired.
\end{proof}

Now note that for $r>0$, $H_r(K_{m,n}^\bullet)$
cannot contain $L_c(\CC)$, since $L_c(\CC)$ is not a composition factor
of $M_c(\wedge^i\h)$ for $i>0$ (because it has smaller lowest eigenvalue
of $\mathbf{h}$ than any eigenvalue of $\mathbf{h}$ in
$M_c(\wedge^i\h)$, $i>0$).
So, varying $k$ and using Lemma \ref{symg}, we conclude that the statement
follows by induction in $d$ from the known case $d=1$. Namely,
if under the correspondence of \cite{Wi}, [Theorem 1.8 and
Proposition 3.7] (see Theorem \ref{wilco}), $H_r(K_{m,n}^\bullet)$ for some $r>0$ corresponds to some (possibly reducible)
representation $\pi\in \Rep\ S_d$, then we have
$\pi^{S_d}=0$ (as $L_c(\CC)$ does not occur) and
$\pi|_{S_k\times S_{d-k}}=\oplus_{r-1\le i+j\le r}
\wedge^i_k\otimes \wedge^j_{d-k}$, by Corollary \ref{chermor}.
So the result follows from Lemma \ref{symg}.

\subsection{Uniqueness for the Koszul-BGG complex}
It turns out that $K_{m,n}^\bullet$, where $m$ is not divisible
by $n$, is the only complex with nonzero differentials
with terms $M_c(\wedge^i\h)$ that one can write. 

\begin{proposition}\label{uni}
The space $\Hom(M_c(\wedge^{i+1}\h), M_c(\wedge^i\h))$ is one-dimensional for all
$i\le n-1$.
\end{proposition}
\begin{proof}
Let $\mathcal{O}$ denote the direct sum $\bigoplus_{i=0}^\infty \mathcal{O}_i$, where $\mathcal{O}_i$
is the category $\mathcal{O}$ for the algebra $H_c(S_i)$. According to Shan, \cite{Shan},
there is a categorical $\widehat{\sl}_{n_0}$-action on $\mathcal{O}$. This gives rise
to $n_0$ categorical $\sl_2$-actions, one for each simple root in $\widehat{\sl}_{n_0}$.
These categorical actions are highest weight in the sense of \cite{str}.


Now let $\lambda,\lambda'$ be the hooks corresponding to $\wedge^{i+1}\h, \wedge^i\h$. Since $|\lambda|$
is divisible by $n_0$, one obtains $\lambda'$ from $\lambda$ by
moving an $a$-box (where $a$ is a suitable residue mod $n$). It follows that $\lambda$
and $\lambda'$ lie in the same family (in the terminology of \cite[Section 3]{str}) for the subalgebra
$\sl_2\subset \hat{\sl}_{n_0}$ corresponding to the residue $a$. Since $\lambda<\lambda'$
we conclude from \cite[Proposition 7.4, Remark 7.8]{str} that
$\dim\Hom_{\mathcal{O}}(M_c(\wedge^{i+1}\h),M_c(\wedge^i\h))=1$.
%
\end{proof}

Let $\Sh_c: {\mathcal O}_c\to {\mathcal O}_{c+1}$ be the shift functor, which is a right exact functor
defined by
$$
\Sh_c(V)=H_c(S_n)\ee\otimes_{\ee H_c(S_n)\ee=\ee_-H_{c+1}(S_n)\ee_-}\ee_-V,
$$
where $\ee=\ee_{0}$ is the symmetrizer, and
$\ee_-=\ee_{n-1}\in \CC S_n$ is the antisymmetrizer
(see \cite{BEG}).
It is shown in \cite{BE} that $\Sh_c$ is an equivalence of categories
for $c>0$.

\begin{corollary}\label{shift}
One has $\Sh_{\frac{m}{n}}(K^\bullet_{m,n})\cong K^\bullet_{m+n,n}$.
\end{corollary}

\begin{proof}
This follows from Proposition \ref{uni} and the fact that the
shift functor maps Verma modules to Verma modules, see
\cite{GL}. Namely, Lemma 4.3.2 in \cite{GL} says that if a category
is equipped with two highest weight structures such that the classes
of the standard objects  coincide in $K_0$,
then the structures coincide. But the shift functor is clearly the identity on
$K_0$ (with the basis of standard modules), as it is flat with respect
to $c$ for $c>0$, and is obviously the identity for generic $c$
(by looking at the eigenvalues of the scaling element $\bold h$).
\end{proof}

\subsection{Another proof of Theorem \ref{homolkoz}}
\begin{lemma}\label{Lem:koz_groth}
Theorem \ref{homolkoz} holds on
the level of the Grothendieck groups, i.e.,
$$
\bigoplus_{i=0}^{n-1} (-1)^i[M_{\frac{m}{n}}(\wedge^i \h_n)]=
\bigoplus_{i=0}^{d-1}(-1)^i[L_{\frac{m}{n}}(\lambda_i)].
$$
\end{lemma}
\begin{proof}

Let $\lambda_i(k)=(k-i,1^{i}).$ By Theorem \ref{charfor} one has
$$[L_{\frac{m}{n}}(n_0\lambda_i(d))]=\sum_{|\mu|=n}c_{\lambda_i(d),n_0}^{\mu}[M_{\frac{m}{n}}(\mu)],$$
where the coefficients $c_{\lambda_i(d),n_0}^{\mu}$ are defined by the equation
$\Psi_{n_0}(s_{\lambda_i(d)})=\sum_{\mu}c_{\lambda_i(d),n_0}^{\mu}s_{\mu}.$
The statement now follows from a symmetric function identity
$$
\sum_{i=0}^{d-1} (-1)^{i} \Psi_{n_0}(s_{\lambda_i(d)})=\Psi_{n_0} \left( \sum_{i=0}^{d-1} (-1)^{i} s_{\lambda_i(d)}\right) =
\Psi_{n_0} (p_d)= p_{n}=\sum_{i=0}^{n-1} (-1)^{i} s_{\lambda_i(n)}.$$
Here we  used the equation $p_{k}=\sum_{i=0}^{k-1} (-1)^{i} s_{\lambda_i(k)}$ twice: for $k=d$ and $k=n$.
\end{proof}

Let us recall again  that all categories $\mathcal{O}_c$ with $c=\frac{r}{n_0}$, where $GCD(r,n_0)=1$ and $r>0$, are equivalent
as highest weight categories. From Proposition \ref{uni} it follows that the equivalences preserve the Koszul-BGG complexes.
So it is enough to prove the theorem for $c=\frac{1}{n_0}$.

\begin{lemma}\label{Lem:koz_mult}
The multiplicity of $H_i(K^\bullet_{d,n})$ in a generic point of the support of $L_{\frac{1}{n_0}}(n)$ equals
$\binom{d-1}{i}$.
\end{lemma}
\begin{proof}
This follows from Lemma \ref{comal}.
Namely, for $c=1/n_0$, the zero set $Z$ of $f_1,...,f_n$ is
generically reduced,\footnote{In fact, it is reduced, but we use
  only generic reducedness.} so for a suitable generic point $z\in Z$,
the differentials $df_1(z),...,df_{n-d}(z)$
are linearly independent. This implies that
in the formal neighborhood of $z$, there exist functions
$c_{ij}$, $j\le n-d$, $i>n-d$, such that
$f_i=\sum_{j=1}^{n-d}c_{ij}f_j$ for $i>n-d$. This implies
(similarly to the proof of Lemma \ref{le1}) that
the completion of $K^\bullet_{d,n}$ at $z$
is the tensor product of the Koszul complex
defined by $f_1,...,f_{n-d}$ with the exterior algebra
$\wedge(\xi_1,...,\xi_{d-1})$ in $d-1$ generators. This implies the statement, as
the dimension of the degree $i$ component of
$\wedge(\xi_1,...,\xi_{d-1})$ is $\binom{d-1}{i}$.
\end{proof}

\begin{proof}[Proof of Theorem \ref{homolkoz}]
The proof is by induction on $i$. Assume that $H_i(K_{d,n}^\bullet)=L_{\frac{m}{n}}(n_0\lambda_i(d))$ for  all $i<k$, where
$k$ is a fixed number from $0$ to $d$ (for $k=0$ the assumption is vacuous). We want to prove that
$H_k(K_{d,n}^\bullet)=L_{\frac{m}{n}}(n_0\lambda_k(d))$. First of all, for $i<k$, $L_{\frac{m}{n}}(n_0\lambda_k(d))$ is not a composition
factor of $H_i(K_{d,n}^\bullet)$, by the inductive assumption.  We also claim that $L_{\frac{m}{n}}(n_0\lambda(d))$ does not
appear in $H_i(K_{d,n}^\bullet)$ for $i>k$. Indeed, assume the converse. Then $L_{\frac{m}{n}}(n_0\lambda(d))$ has to be a composition
factor of $M_{\frac{m}{n}}(\lambda_i(n))$. However, $\lambda_i(n)\not \geq n_0\lambda_k(d)$ in the dominance ordering as $\lambda_i(n)$
has more rows than $n_0\lambda_k(d)$. So $L_{\frac{m}{n}}(n_0\lambda_k(d))$ cannot appear as a composition factor of $M_{\frac{m}{n}}(\lambda_i(n))$.
But $L_{\frac{m}{n}}(n_0\lambda_i(d))$ has to appear in some $H_i(K_{d,n}^\bullet)$ thanks to Lemma \ref{Lem:koz_groth} and so we must
have $i=k$. Also the generic ranks of $L_{\frac{m}{n}}(n_0\lambda_k(d))$ and $H_k(K_{d,n}^\bullet)$ coincide by Lemma \ref{Lem:koz_mult}.
It follows that $L_{\frac{m}{n}}(n_0\lambda_k(d))=H_k(K_{d,n}^\bullet)$.
\end{proof}

\section{Symmetry for rational Cherednik  algebras of type A}

\subsection{The statement}
We start by recalling the type A rational Cherednik algebra. We
will need a universal version, which is slightly different from
the one defined in the preliminaries section. Let $n\ge 2$ be an integer and $\h$ be the $n-1$
dimensional vector space viewed as the subspace $\{(x_1,\ldots,x_n)| \sum_{i=1}^n x_i=0\}\subset \C^n$.
Then $\h$ can be thought of as the Cartan subalgebra in the Lie
algebra
${\sl}_n$. Let $\Delta_+\subset \h^*$ be the
root system. The corresponding Weyl group is $S_n$. Fix independent variables $h,c$. The (universal) rational Cherednik algebra
${\bf H}$ is the quotient of the semi-direct product $\C S_n\ltimes T(\h\oplus\h^*)[h,c]$ by the relations
$$[x,x']=0, [y,y']=0, [y,x]=h\langle y,x\rangle-c\sum_{\alpha\in \Delta_+} \langle\alpha,y\rangle\langle x,\alpha^\vee\rangle s_\alpha,$$  with $x,x'\in \h^*, y,y'\in \h.$ Here $\alpha^\vee,s_\alpha$ are, respectively, the coroot and the  reflection corresponding to a root $\alpha$.

We remark that the algebra ${\bf H}$ is bigraded: $h,c$ lie in bidegree $(1,1)$, $\h$ is in bidegree $(0,1)$,
$\h^*$ is in bidegree $(1,0)$, and $S_n$ is in bidegree $(0,0)$.

Recall that we introduced the idempotents
$\ee_i$ in $\C S_n$ corresponding to the irreducible
representations $\wedge^i \h$ of $S_n$, and
$e_i:=\ee_i+\ee_{i-1}$, where $\ee_{-1},\ee_n=0$. The element
$e_i\in \C S_n$ should be thought of as the idempotent corresponding
to the $S_n$-module $\wedge^i \C^n=\wedge^i \h\oplus
\wedge^{i-1}\h$. Recall that $\overline{\ee}=\sum_k e_{2k}=\sum_k
e_{2k+1}=\sum_j \ee_j$. Our goal is to understand the structure of
the {\it quasispherical} subalgebra $\overline{\ee}{\bf H}\overline{\ee}\subset {\bf
  H}$. The latter is not a unital subalgebra, instead $\overline{\ee}$ is
the unit in $\overline{\ee}{\bf H}\overline{\ee}$.
We will identify the algebra $\overline{\ee}{\bf H}\overline{\ee}$
with a certain quantum Hamiltonian reduction generalizing
the description of $e_0 {\bf H} e_0$ obtained by Gan and Ginzburg,
\cite{GG}.

Set $\VV_n:=\slf_n(\C)\oplus\C^n$. We will call it $\VV$
if no confusion is possible.
The group $G:=\GL_n(\C)$ acts naturally on $\VV$. Also $G$ acts on
the symplectic vector space $T^*\VV$  with moment map $\mu: T^*\VV\rightarrow \g:=\Lie(G)$
given by $\mu(A,B,i,j)=[A,B]+ i\otimes j, A,B\in \slf_n(\C), i\in \C^n, j\in (\C^n)^*$,
where we identify $\slf_n(\C)$ with its dual via the trace pairing. The space $T^*\VV$ also
carries an action of the two-dimensional torus $(\C^\times)^2$ commuting with $G$:
$$(t_1,t_2).(A,B,i,j)=(t_1^{-1} A, t_2^{-1} B, t_1^{-1} i, t_2^{-1} j).$$
Consider the subtori $T_1:=\{(t,t^{-1})\in (\C^\times)^2\}, T_2:=\{(t,t)\in (\C^\times)^2\}$.

The symplectic vector space $T^*\VV$ admits a natural quantization, the algebra $\D_h(\VV)$ of homogenized
differential operators. The latter can be obtained from the algebra $\D(\VV)$ of differential
operators by using the Rees construction. Namely, the algebra $\D(\VV)$ is filtered by the
subspaces $\D_{\leqslant i}(\VV)$ of differential operators of order $\leqslant i$. Then $\D_h(\VV):=\bigoplus_i
\D_{\leqslant i}(\VV) h^i\subset \D(\VV)[h]$. The algebra $\D_h(\VV)$ is bigraded: its component of bidegree
$(i,j)$, by definition, consists of all elements in $h^j \D_{\leqslant j}(\VV)$ that have degree $i+j$
with respect to the grading induced by the $\C^\times$-action on $\VV$ by $(t,v)\mapsto t^{-1}v$. In particular, $\VV^*\subset \D_{\leqslant 0}(\VV)$ has bidegree $(1,0)$, while $\VV\subset \operatorname{Vect}(\VV)\subset h\D_{\leqslant 1}(\VV)$ has bidegree $(0,1)$
and $h$ has bidegree $(1,1)$. Also the action of $G$ on $\VV$
gives rise to the quantum comoment map $\Phi_h:\g\rightarrow
\D_h(\VV)$, under which
an element $\xi$ maps to the corresponding vector field in
$h\D_{\leqslant 1}(\VV)$ induced by the action of $\g$. We remark that the quantum comoment map has image in bidegree $(1,1)$.

More generally, let $U$ be a $G$-module.  Consider the  tensor product $\D_h(\VV)\otimes \End(U)$ that inherits
the bigrading from $\D_h(\VV)$ with $\End(U)$ being of bidegree
$(0,0)$. There is the quantum comoment
map $\Phi^U_h(\xi):=\Phi_h(\xi)\otimes 1+h\otimes \xi_U: \g\to\D_h(\VV)\otimes \End(U)$, where $\xi_U$ just stands for the image of $\xi$ in $\End(U)$.

Let $\z$ denote the center of $\g$. It is naturally identified with $\C$ via $z\mapsto z \operatorname{id}_{\C^n}$.
Let $\beta$ denote the basis element in $\z$ corresponding to $1$.

The quantum Hamiltonian reduction we are going to consider will be defined first at the level of sheaves.
A sheaf of interest will be on a formal deformation
$\widetilde{X}$ of the Hilbert scheme $X=\operatorname{Hilb}_n$,
to be recalled first.

The variety $X$ can be produced by Hamiltonian reduction as follows. Consider the character
$\theta:=\det$ of $G=\GL_n(\C)$ and let $\overline{\VV}^{ss}$ be the open subset of $\theta$-semistable
points in $\overline{\VV}:=T^*\VV$. Then $\overline{\VV}^{ss}\cap \mu^{-1}(0)=\{(A,B,i,0)| [A,B]=0, \C[A,B]i=\C^n\}$.
By definition, $X$ is the Hamiltonian reduction of $\overline{\VV}^{ss}$ by the action of $G$, i.e.,
$X=(\mu^{-1}(0)\cap \overline{\VV}^{ss})/G$. This is a smooth symplectic variety equipped with
a $(\C^\times)^2$-action and also with a morphism $X\rightarrow
\mu^{-1}(0)//G\cong (\h\oplus\h^*)/S_n$
that is a resolution of singularities.

In fact, we will need to work over a larger scheme. Namely, consider the Hamiltonian reduction $[\overline{\VV}^{ss}\cap \mu^{-1}(\z^*)]/G$.
This is a scheme over $\z^*$. Its non-zero fiber is the so called Calogero-Moser space and the $G$-action
over such fiber is known to be free. Let $\widetilde{X}$ be the completion of
this scheme at the zero fiber, this is a formal scheme over the formal neighborhood $(\z^*)^{\wedge_0}$. The scheme
$\widetilde{X}$ comes equipped with a fiberwise symplectic form, say $\widetilde{\omega}$.

We will define a sheaf $\Dcal^U_h$ of $\C[[\z^*,h]]$-algebras on $\widetilde{X}$ as follows. We sheafify
the $h$-adic completion of $D_h(\VV)$ to $\overline{\VV}$.
Abusing notation, we denote the resulting sheaf again by $D_h(\VV)$. Then set
\begin{equation}\label{eq:loc_red}\Dcal^U_h:=\left[(\End(U)\otimes D_h(\VV))|_{\overline{\VV}^{ss}}/(\End(U)\otimes D_h(\VV))|_{\overline{\VV}^{ss}}\Phi^{U}_h([\g,\g])\right]^G.\end{equation}
The group $(\C^\times)^2$ naturally acts on $\Dcal^U_h$, where we have $(t_1,t_2).h=t_1t_2 h$.
Let $\A_h(\VV,U)$ stand for the subalgebra of $T_2$-finite elements in $\Gamma(\widetilde{X}, \Dcal^U_h)$.
This  is an algebra over $\C[\beta,h]$ equipped with an action of $(\C^\times)^2$ by algebra automorphisms.


\begin{theorem}\label{Thm_iso}
There is a $(\C^\times)^2$-equivariant
$\C[h]$-linear isomorphism
$$
\Upsilon:\A_h(\VV,\wedge^{n-2\bullet}
\C^n)\xrightarrow{\sim} \overline{\ee}{\bf H}
\overline{\ee}
$$
that maps $\beta$ to $c+h$.
\footnote{Here $\wedge^{n-2\bullet}\C^n:=\oplus_j\wedge^{n-2j}\C^n$.}
This isomorphism induces an isomorphisms
$$
\Upsilon_j: \A_h(\VV,\wedge^{n-2j}\C^n)\xrightarrow{\sim}
e_{n-2j}{\bf H}e_{n-2j}
$$
for $j\ge 0$.
\end{theorem}

Theorem \ref{Thm_iso} is proved in the next three subsections.

\subsection{Procesi bundle}
In the proof  we use a remarkable
bundle on $X$, the Procesi bundle $\Pro$ originally constructed by M. Haiman,
\cite{Haiman}; an alternative construction was produced by Ginzburg, \cite{Ginzburg_Pro}.

The Hamiltonian reduction construction equips $X$ and $\widetilde{X}$ with  natural vector bundles $\taut,\widetilde{\taut}$
of rank $n$. Namely, we can consider the $G$-equivariant vector bundle on $T^*\VV$ that is trivial as a vector
bundle, and such that $G$ acts on a fiber as on the tautological $n$-dimensional representation.
We also equip this bundle with the  $(\C^\times)^2$-action that is trivial on the fiber.
The bundle $\taut$ is the descent of the restriction of this bundle to $\mu^{-1}(0)\cap \overline{\VV}^{ss}$.
The bundle ${\mathcal{T}}$ is $(\C^\times)^2$-equivariant. The bundle $\widetilde{\taut}$ on $\widetilde{X}$ is defined in  a similar way.

There is another bundle on $X$, the Procesi bundle $\Pro$. It is a $(\C^\times)^2$-equivariant
bundle with a fiberwise action of $S_n$ having the following properties:
\begin{itemize}
\item[(i)] $\End_{\Str_X}(\Pro)=\C S_n\ltimes\C[\h\oplus\h^*]$ (a
  $\C[\h\oplus\h^*]^{S_n}$- and $S_n$- and $(\CC^\times)^2$-linear
isomorphism).
\item[(ii)] $\Ext^i_{\Str_X}(\Pro,\Pro)=0$ for $i>0$.
\item[(iii)] $e_0 \Pro=\Str_X$.
\item[(iv)] $e_1\Pro=\taut$.
\end{itemize}
Ginzburg generalized (iv): the vector bundle $e_i\Pro=
\Hom_{S_n}(\wedge^i \C^n,\Pro)$
is naturally isomorphic to $\wedge^i\taut$. This follows from \cite[Theorem 1.6.1]{Ginzburg_Pro} and is the main ingredient
in the proof of Theorem \ref{Thm_iso}.
(Note that this
property fails if we replace $\wedge^i$ in both places by a more general Schur
functor!)

Because of (ii), the bundle
$\Pro$ uniquely extends to a $(\C^\times)^2$-equivariant bundle
$\widetilde{\Pro}$ on $\widetilde{X}$.  Moreover,
since $\wedge^i \taut$ is a  direct summand of $\Pro$, we get $\Ext^1_{\Str}(\wedge^i \taut,\wedge^i \taut)=0$. So
 $\wedge^i \taut$ is a unique $(\C^\times)^2$-equivariant extension of $\wedge^i\widetilde{\taut}$.
So we see that $\widetilde{\Pro}$ has the following properties:
\begin{itemize}
\item[($\widetilde{\text{i}}$)] $\End_{\Str_{\widetilde{X}}}(\widetilde{\Pro})/(\z)=\C S_n\ltimes\C[\h\oplus\h^*]$.
\item[($\widetilde{\text{ii}}$)] $\Ext^i_{\Str_{\widetilde{X}}}(\widetilde{\Pro},\widetilde{\Pro})=0$ for $i>0$.
\item[($\widetilde{\text{iii}}$)] $e_0\widetilde{\Pro}=\Str_{\widetilde{X}}$.
\item[($\widetilde{\text{iv}}$)]
  $e_1\widetilde{\Pro}=\widetilde{\taut}$. More generally, the
multiplicity space of the $S_n$-module $\wedge^i \C^n$ in
$\widetilde{\Pro}$ is
isomorphic to $\wedge^i \widetilde{\taut}$.
\end{itemize}

\subsection{Quantization}
Set $\Dcal_h:=\Dcal_h^{\C}$, where $\C$ stands for the trivial $G$-module. This is a quantization of
the structure sheaf $\Str_{\widetilde{X}}$. We remark that $\Dcal_h$ is almost the same as the canonical
quantization of $\widetilde{X}$ studied in \cite{BK} and \cite{quant}. The only difference is that the structure
of the $\C[[\z^*,h]]$-algebra is changed, since here we have used a non-symmetrized quantum comoment
map to define $\Dcal_h$.

To a $G$-module $U$ we can assign a bundle $\widetilde{\U}$
on $\widetilde{X}$ as before. One can construct a quantization $\widetilde{\U}_h$ of $\widetilde{\U}$
as follows:
$$\widetilde{\U}_h:=[(U\otimes D_h(\VV))|_{\overline{\VV}^{ss}}/(U\otimes D_h(\VV))|_{\operatorname{\VV}^{ss}}\Phi_h([\g,\g])]^G.$$
It is clear from the construction that $\widetilde{\U}_h$ is a $(\C^\times)^2$-equivariant right $\Dcal_h$-module.
Let us relate $\widetilde{\U}_h$ to $\Dcal^U_h$.

\begin{lemma}\label{lem:local_end}
There is a natural identification of the sheaves of  algebras $\LEnd_{\Dcal_h}(\widetilde{\U}_h)$
and $\Dcal^U_h$.
\end{lemma}
\begin{proof}
There is a natural action of $\Dcal^U_h$ on $\widetilde{\U}_h$ from the left commuting
with a right action of $\Dcal_h$. This gives rise to a homomorphism
$\Dcal^U_h\rightarrow \LEnd_{\Dcal_h}(\widetilde{\U}_h)$. The endomorphism sheaf is flat modulo $h$
because $\widetilde{\U}_h$ is a locally free right $\Dcal_h$-module. The
 sheaf $\Dcal^U_h$ is complete in the $h$-adic topology. This is
 because the sheaf $\End(U)\otimes D_h(\VV)|_{\overline{\VV}^{ss}}$
is Noetherian and so every left ideal is finitely generated and hence closed with respect to the $h$-adic topology,
compare to \cite[Lemma 2.4.4]{HC}. So it is enough to check that the homomorphism is an isomorphism
modulo $h$. Equivalently, we need to show that the endomorphism sheaf of  $\widetilde{\U}$
is the sheaf induced by $\operatorname{End}(U)$. But this is clear.
\end{proof}

Also, thanks to ($\widetilde{\text{ii}}$) we have a unique quantization $\widetilde{\Pro}_h$ of $\widetilde{\Pro}$,
where $\widetilde{\Pro}_h$ is again a $(\C^\times)^2$-equivariant right $\Dcal_h$-module. We still have a
natural action of $S_n$ on $\widetilde{\Pro}_h$. Consider the endomorphism algebra $\End_{\Dcal_h}(\widetilde{\Pro}_h)$.
This is a $\C[[\z^*,h]]$-algebra equipped with a $(\C^\times)^2$-action by automorphisms. Consider the subalgebra
$\End_{\Dcal_h}(\widetilde{\Pro}_h)_{T_2-fin}$ of $T_2$-finite elements in
$\End_{\Dcal_h}(\widetilde{\Pro}_h)$.
The results of \cite[Section 6]{quant}
 relate the latter algebra to ${\bf H}$. Summarizing these results, we obtain the following
proposition.

\begin{proposition}\label{Prop:Pro_cher}
We have an $S_n$-linear, $(\C^\times)^2$-equivariant isomorphism of $\C[h]$-algebras $\Upsilon: {\bf H}\rightarrow
\End_{\Dcal_h}(\widetilde{\Pro}_h)_{T_2-fin}$. It maps $c\in {\bf H}$ to $-\beta$ or to $\beta-h$.
\end{proposition}

We will see in the next subsection that actually
$\Upsilon(c)=\beta-h$ (i.e. $\beta\mapsto c+h$, as desired).

Now let $\tilde{e}\in \C S_n$ be an idempotent. Then $\Upsilon$ induces an isomorphism $$\tilde{e} {\bf H} \tilde{e}
\xrightarrow{\sim} \tilde{e}\End_{\Dcal_h}(\widetilde{\Pro}_h)_{T_2-fin}\tilde{e}=\End_{\Dcal_h}(\tilde{e}\widetilde{\Pro}_h)_{T_2-fin}.$$

\subsection{Proof of Theorem \ref{Thm_iso}}
Let us remark that if a $G$-module $U$ satisfies $\widetilde{\U}=\tilde{e}\widetilde{\Pro}$, then $\widetilde{\U}_h=\tilde{e}\widetilde{\Pro}_h$. This is because $\Ext^i_{\Str_{\widetilde{X}}}(\tilde{e}\widetilde{\Pro},\tilde{e}\widetilde{\Pro})=0$
for $i>0$ and so $\widetilde{\Pro}$ admits a unique
quantization. So by applying $\ee_-$,
we have a $(\C^\times)^2$-equivariant $\C[[\z^*,h]]$-linear
isomorphism $\Gamma(\widetilde{X},\Dcal^U_h)\rightarrow \End_{\Dcal_h}(\tilde{e}\widetilde{\Pro}_h)$ and hence
a $(\C^\times)^2$-equivariant $\C[\beta,h]$-linear isomorphism $\A_h(\VV,U)\xrightarrow{\sim} \End_{\Dcal_h}(\tilde{e}\widetilde{\Pro}_h)_{T_2-fin}$.
So, to prove Theorem \ref{Thm_iso}, it remains to take $U=\wedge^{n-2\bullet}\CC^n$ (where we use Ginzburg's result on $\wedge^i{\mathcal \T}$)
and verify that in Proposition \ref{Prop:Pro_cher}, we have $\Upsilon(c)=\beta-h$.

Assume the converse, $\Upsilon(c)=-\beta$. Consider the
determinant representation $\wedge^n \C^n$. Then $\widetilde{\U}=\ee_-\widetilde{\Pro}$,
where $\ee_-$ is the idempotent corresponding to the sign representation. So we have an isomorphism
$\ee_-{\bf H}\ee_-\cong \A_h(\VV,\wedge^n \C^n)$. Consider the specialization of this isomorphism at $h=1$ and $c=p$.
Since $\Upsilon(c)=-\beta$, we get
$\ee_-{\bf H}_{1,p}\ee_-\cong \A_{1,-p}(\VV,\wedge^n \C^n)$.
It is known (\cite{BEG}) that there is an isomorphism $\sigma_1:\ee{\bf
  H}_{1,p-1} \ee\xrightarrow{\sim} \ee_-{\bf H}_{1,p}\ee_-$.
Also, by the definition of quantum hamiltonian reductions,
there is an isomorphism
$\sigma_2:\A_{1,-1-p}(\VV,\C)\xrightarrow{\sim}\A_{1,-p}(\VV;\wedge^n
\C^n)$. So, we have an isomorphism
$\sigma_2^{-1}\circ \ee_-\Upsilon \circ\sigma_1:\ee {\bf H}_{1,p-1} \ee\rightarrow \A_{1,-1-p}(\VV,\C)$.
On the other hand, we have $\ee\Upsilon : \ee{\bf H}_{1,p+1}\ee\rightarrow \A_{1,-1-p}(\VV,\C)$.
This gives rise to an isomorphism $\ee{\bf H}_{1,p+1}\ee\cong \ee {\bf H}_{1,p-1} \ee$ for all $p$.
It is clear, however, that such an isomorphism cannot exist
(for example, from considering dimensions of irreducible finite dimensional representations, see
\cite{BEG}).

\subsection{Local  vs. global quantum Hamiltonian reductions}\label{locglob}
We also can form the global hamiltonian reduction
$$A_h(\VV,U):=[(\End(U)\otimes D_h(\VV))/(\End(U)\otimes D_h(\VV))\Phi^{U}_h([\g,\g])]^G,$$
where the algebra $D_h(\VV)$ is not completed. This definition yields a $(\C^\times)^2$-equivariant
$\C[\beta,h]$-linear algebra homomorphism $\varphi: A_h(\VV,U)\rightarrow \A_h(\VV,U)$.

We do not know whether this homomorphism is an isomorphism with one exception:
$U=\C$. In this case modulo $(h,\beta)$, the homomorphism $\phi$ is the map $\C[\h\oplus\h^*]^{S_n}\rightarrow \C[X]$. This is an isomorphism. Since the algebras in consideration are graded and flat over $\C[h,\beta]$,
the homomorphism $\varphi$ is an isomorphism.

Also let us point out that our construction is independent of $U$ in the following sense.
Let $e',e''$ be two commuting idempotents in $\C S_n$ such that $e'e''=0$. Assume that
$U',U''$ be $G$-modules such that $\U'\cong e'\Pro, \U''\cong e''\Pro$. Then we have homomorphisms
$$A_h(\VV,U')\rightarrow \A_h(\VV,U')\xrightarrow{\sim} e'{\bf H}e', A_h(\VV,U'')\rightarrow \A_h(\VV,U'')\xrightarrow{\sim} e''{\bf H}e'',$$
$$A_h(\VV,U'\oplus U'')\rightarrow \A_h(\VV,U'\oplus U'')\xrightarrow{\sim} (e'+e''){\bf H}(e'+e'').$$
The latter homomorphisms map idempotents corresponding to $U',U''$ to the analogous idempotents.
The induced homomorphisms between  $A_h(\VV,U^?),\A_h(\VV,U^?), e^{?}{\bf H}e^{?}$ coincide with
the homomorphisms above.

Specializing to $h=1, \beta=c$ we get a homomorphism $A_c(\VV,U)\rightarrow e_U {\bf H}_c e_U$.
We note that for $U=\C$ the homomorphism $A_c(\VV,U)\rightarrow \A_c(\VV,U)=\ee H_c \ee$
is an isomorphism.

\subsection{The CEE construction}

Let $V$ be an $m$-dimensional vector space, and $M$ a $D$-module
on ${\mathfrak{\g}}={\sl}(V)={\sl}_m$.
Let $G=SL(V)$ (note that we now
consider ${\sl}_m$ instead of ${\mathfrak{gl}}_m$,
and use $m$ instead of $n$ for the size of matrices).

Consider the vector space
$$
F_n(M):=(M\otimes V^{*\otimes n})^G
$$

\begin{remark} Let $M_f$ be the locally finite part of $M$ under
the ${\sl}(V)$-action by conjugation. It is clear that
$F_n(M)=F_n(M_f)$, so we may assume that $M=M_f$, i.e. that $M$
is $G$-equivariant. In this case, $M=\oplus_{s=0}^{m-1}M(s)$,
where $M(s)$ is the subspace on which the center
of $SL(V)$ acts as it does in $V^{\otimes s}$. It is easy to see
that $F_n(M)=F_n(M(\bar n))$, where $\bar n$ is the remainder
under division of $n$ by $m$.
\end{remark}

Following \cite{CEE}, Subsection 9.6 \footnote{Note that in \cite{CEE}, the
  parameter $m$ is denoted by $N$.}
and replacing $V$ with $V^*$ (using the isomorphism ${\sl}(V)\cong {\sl}(V^*)$
given by $A\mapsto -A^*$), we
obtain the following proposition

\begin{proposition}
\label{CEEaction}
The space $F_n(M)$ carries a natural action
of the rational Cherednik algebra $H_{\frac{m}{n}}(n)$.
\end{proposition}

If $M=M(s)$, we say that $M$ has central character $s$.
Let $\mu$ be a partition of $d=GCD(m,n)$, and $M_\mu$
be the irreducible $G$-equivariant $D$-module on ${\sl}(V)$
with central character $\bar n$ supported on the nilpotent orbit $O_\mu$
corresponding to $\mu$, as in \cite{CEE}. Then, as shown in
\cite{CEE}, $F_n({\mathcal
  F}(M_\mu))=L_{\frac{m}{n}}(n\mu/d)$. Thus,
$\overline{\ee}F_n({\mathcal{F}}(M_\mu))=
\overline{\ee}L_{\frac{m}{n}}(n\mu/d)$.

Thus, applying $\overline{\ee}$ to the statement of
Proposition \ref{CEEaction}, we immediately obtain the following corollary.

\begin{corollary}\label{coronatact}
The space
$$
\overline{\ee}F_n(M)=(M\otimes (\oplus_{j\text{ even  }}S^{n-j}V^*\otimes \wedge^jV^*))^G
$$
carries a natural action of the algebra $\overline{\ee}H_{\frac{m}{n}}(n)\overline{\ee}$.
If $M={\mathcal{F}}(M_\mu)$, where $\mu$ is a partition of $d$, with central
character $\bar n$, then this space is naturally isomorphic to
$\overline{\ee}L_{\frac{m}{n}}(n\mu/d)$.
\end{corollary}

\subsection{Matching of representations with minimal support for spherical Cherednik algebras}
\label{matc}

Let $m,n$ be positive integers with $GCD(m,n)=d$.
By $H_c(n)$ we will mean the rational Cherednik algebra
$H_c(S_n,\CC^{n-1})$. It is known \cite{Lo} that if $c$ has denominator $d$, then
the proper two-sided ideals in $H_c(S_n)$ form a chain $0=I_0\subset I_1\subset...\subset I_{[n/d]}$;
so $I_{[n/d]}$ is a maximal ideal. We will denote it by
$I(c)$. Let $\ee=\ee_0$ be the symmetrizing idempotent for $S_n$, and
$I_\ee (c)=\ee I(c)\ee\subset \ee H_c(n)\ee$.

Recall that the spherical Cherednik algebra $\ee H_c(n)\ee$
contains the $x$-subalgebra generated by the power sums $p_j(x_1,...,x_n)$,
the $y$-subalgebra generated by the power sums $p_j(y_1,...,y_n)$,
and the ${\sl}_2$-subalgebra generated by $\sum x_i^2$ and $\sum y_i^2$.
The following proposition is proved in \cite[Proposition 9.5]{CEE}:

\begin{proposition}\label{CEEiso}
There is an isomorphism $\phi:
\ee H_{\frac{m}{n}}(n)\ee/I_\ee(\frac{m}{n})\to \ee H_{\frac{n}{m}}(m)\ee/I_\ee(\frac{n}{m}),$
which preserves the Bernstein filtration, the $x$-subalgebra, the $y$-subalgebra,
and the ${\sl}_2$-subalgebra.
\end{proposition}

Consider the isomorphism $\phi$ in more detail.
Let $x_1,...,x_n$ be the $x$-variables
for the first spherical Cherednik algebra, and $x_1',...,x_m'$ be
the $x$-variables for the second one.
Let $p_{r,n}=x_1^r+...+x_n^r$. By \cite[Proof of Proposition 9.7, line 3]{CEE},
we have $\phi(p_{r,n}(x))=\frac{n}{m}p_{r,m}(x')$.
Denote the $x$-subalgebras of these two spherical Cherednik algebras
(i.e., the subalgebras generated by $p_{r,n}$ and $p_{r,m}$ respectively)
by $A_{m,n}$ and $A_{n,m}$, respectively, and define the affine schemes
${\mathcal{X}}_{m,n}=\Spec\ A_{m,n}$ and ${\mathcal{X}}_{n,m}=\Spec\ A_{n,m}$. Then we have an isomorphism
$\phi: A_{m,n}\to A_{n,m}$, and hence an isomorphism of schemes
$\phi^*: {\mathcal{X}}_{n,m}\to {\mathcal{X}}_{m,n}$. This isomorphism induces an isomorphism
of the corresponding reduced schemes (i.e., affine varieties)
$\phi^*: \overline{\mathcal{X}}_{n,m}\to \overline{\mathcal{X}}_{m,n}$.
These varieties, by the results of \cite[Section 9]{CEE},
are just the minimal supports of category ${\mathcal O}$ modules over the corresponding spherical Cherednik algebras.
This means that $\overline{\mathcal{X}}_{m,n}=X_{d,n/d}(n)/S_n$
is the image (under taking the quotient by permutations) of the
locus where $x_i=x_j$ when $i-j$ is divisible by $d$ ($i,j\in [1,n]$),
and $\overline{\mathcal{X}}_{n,m}=X_{d,m/d}(m)/S_m$ is the image of the locus where $x_i'=x_j'$ when $i-j$ is divisible by
$d$ ($i,j\in [1,m]$). Set $z_i=x_i$ and $z_i'=x_i'$ for $i=1,...,d$. Then
$p_{r,n}(x)=\frac{n}{d}p_{r,d}(z)$ on $\overline{\mathcal{X}}_{m,n}$, and
$p_{r,m}(x')=\frac{m}{d}p_{r,d}(z')$ on
$\overline{\mathcal{X}}_{n,m}$. Thus,
$\overline{\mathcal{X}}_{m,n}$ and $\overline{\mathcal{X}}_{n,m}$
are $d$-dimensional affine spaces with coordinates $p_{r,d}(z)$, $p_{r,d}(z')$, respectively
($r=1,...,d$), and we have $\phi(p_{r,d}(z))=p_{r,d}(z')$.
Hence $\phi(\Delta^2(z))=\Delta^2(z')$, where $\Delta$ is the Vandermonde determinant.
Let $Z\subset {\mathcal{X}}_{n,m}$ be the zero locus of $\Delta(z)$ (i.e. the image of the locus where $z_i=z_j$ for some $i\ne j$), and $Z'\subset {\mathcal{X}}_{m,n}$
be the zero locus of $\Delta(z')$ (i.e., the image of the locus where $z_i'=z_j'$ for some $i\ne j$).
We see that $\phi^*(Z')=Z$.

Now consider the category ${\mathcal O}$
of modules over $\ee H_{\frac{m}{n}}(n)\ee/I_\ee(\frac{m}{n})$
(i.e., the category of modules which are finitely generated over
$\CC[x_1,...,x_n]^{S_n}$ and locally nilpotent under the
action of the augmentation ideal of $\CC[y_1,...,y_n]^{S_n}$).
This category is equivalent to the category of minimal support modules
in the category $\mathcal{O}$ for $H_{\frac{m}{n}}(n)$: namely, an
equivalence is given by $M\mapsto \ee M$. So by the results of
\cite[Theorem 1.8 and Proposition 3.7]{Wi} (see Theorem \ref{wilco}), it is a semisimple
category with the simple objects
$\ee L_{\frac{m}{n}}(\frac{n}{d}\mu)$, where $\mu$ is a partition of $d$.

We will need the following proposition on how the isomorphism $\phi$ acts on these modules.

\begin{proposition}\label{match}
For any partition $\mu$ of $d$, the pushforward map $\phi_*$
under the isomorphism $\phi$ of Proposition \ref{CEEiso}
sends the module $\ee L_{\frac{m}{n}}(\frac{n}{d}\mu)$ to the module
$\ee L_{\frac{n}{m}}(\frac{m}{d}\mu)$.
\end{proposition}

\begin{proof} It is clear that $\phi_*$ sends the module $\ee L_{\frac{m}{n}}(n\mu/d)$ to the module
$\ee L_{\frac{n}{m}}(m\sigma(\mu)/d)$, where $\sigma$ is a certain permutation of the set of partitions of $d$, and our job
is to show that $\sigma=\Id$.
Let us localize our algebras and modules with respect to the loci
$Z$ and $Z'$ (i.e., to the complements of these loci), and denote the corresponding localizations by the subscript
``$\loc$''. Since $\phi^*(Z')=Z$ (as shown above),
we have an isomorphism $\phi_{\loc}: (\ee
H_{\frac{m}{n}}(n)\ee/I_\ee(\frac{m}{n}))_{\loc}\to (\ee H_{\frac{n}{m}}\ee/I_\ee(\frac{n}{m}))_{\loc}$,
which maps the module $\ee L_{\frac{m}{n}}(n\mu/d)_{\loc}$ to the module
$\ee L_{\frac{n}{m}}(m\sigma(\mu)/d)_{\loc}$.

On the other hand, we see from the results of
\cite{Wi} (see Section 4 of \cite{Wi}, in particular Theorem 4.4)
that the algebra
$(\ee H_{\frac{m}{n}}(n)\ee/I_\ee(\frac{m}{n}))_{\loc}$ can be naturally identified with
the algebra $(D(\CC^d\setminus\text{diagonals})\otimes \End(Y(m,n)^{\otimes d}))^{S_d}$,
where $Y(m,n)$ is the spherical part of the irreducible finite dimensional
representation of $H_{\frac{m}{n}}(n_0)$. But it is shown in \cite{CEE} (see Section 9, in particular, Proposition 9.5) that
there is a natural identification of graded spaces $\gamma: Y(m,n)\cong Y(n,m)$, and upon this identification
the map $\phi_{\loc}$ becomes the identity map.
But it follows from \cite{Wi} (see Section 4 of \cite{Wi} and Theorem \ref{wilco} above)
that the module $\ee L_{\frac{m}{n}}(\frac{n}{d}\mu)$ corresponds under the above identification to
the local system on $(\CC^d/S_d)\setminus Z$ which is attached to the representation
$\pi_\mu\otimes Y^{\otimes d}$ of $S_d$ (where $Y=Y(m,n)=Y(n,m)$ and $\pi_\mu$ is the irreducible
representation of $S_d$ attached to the partition $\mu$).
Thus, we see that $\pi_\mu\cong\pi_{\sigma(\mu)}$, which implies that $\mu=\sigma(\mu)$, as desired.
\end{proof}

\subsection{The generalized Gan-Ginzburg construction}

Recall the setting of quantum hamiltonian reduction introduced above
(but now for numerical values of parameters, and $n$ replaced
with $m$).
Let $\g={\sl}_m$, $V=\CC^m$, $\VV_m=\g\times V$. We will denote
$\VV_m$ by $\VV$ for brevity. Let $0\le i\le m$, and consider
the algebra $\overline{A}:=D(\VV)\otimes \End(\wedge^{m-2\bullet}V)$.
For $a\in {\mathfrak{gl}}_m$, let $X_a$ be the vector field
on $\VV$ corresponding to the action of $a$,
and let us consider the homomorphism (the quantum moment map)
$\mu: {\mathfrak{gl}}_m\to \overline{A}$ defined by $\mu(a):=X_a\otimes 1+1\otimes a$.
Let $c\in \CC$ and $\chi_c: {\mathfrak{gl}}_m\to \CC$
be the character defined by $\chi_c(a)=c\Tr(a)$.
Let $A_c(\VV,\wedge^{m-2\bullet}V):=(\overline{A}/\overline{A}(\mu-\chi_c)({\mathfrak{gl}_m}))^{\mathfrak{gl}_m}$
be the (global) quantum hamiltonian reduction.

\begin{proposition}\label{reducedaction}
Let $M$ be a $D(\g)$-module. Then the algebra
$A_c(\VV,\wedge^{m-2\bullet}V)$ acts naturally on the space
$(M\otimes SV^*\otimes \wedge^{m-2\bullet}V\otimes
\chi_{-c})^{{\mathfrak{gl}}_m}$.
\end{proposition}

\begin{proof} This follows directly from the definition of the quantum
  Hamiltonian reduction.
\end{proof}

Note that the space $(M\otimes SV^*\otimes \wedge^{m-j} V\otimes
\chi_{-c})^{{\mathfrak{gl}}_m}$ is nonzero if and only if
$-mc+m-j$ is a nonnegative integer $\ell$, and
in this case this space is
$(M\otimes S^\ell V^*\otimes \wedge^{m-j}V)^\g$.
Thus, replacing $c$ by $1-c$ and $\wedge^{m-j}V$ by $\wedge^jV^*$
(which are isomorphic ${\mathfrak{sl}}_m$-modules),
we obtain the following corollary.

\begin{corollary}\label{co1}
For a nonnegative integer $\ell$,
on the space $(M\otimes (\oplus_{j\text{ even}}S^{mc-j} V^*\otimes \wedge^{j}V^*))^\g$,
there is a natural action of the algebra $A_{1-c}(\VV,\wedge^{m-2\bullet}V)$.
\end{corollary}

We are now ready to state and prove the main theorem of this section.
Recall from Subsection \ref{idem}
that $\overline{\ee}=\sum_{i=0}^{n-1}\ee_{i}=
\sum_{i\ge 0} e_{2i}=\sum_{i\ge 0}e_{2i+1}$
is an idempotent for $S_n$. We have: $\overline{\ee}\CC S_n=\wedge \h_n$, where $\h_n$
is the reflection representation of $S_n$.

\begin{theorem}\label{the1}
We have a natural isomorphism of algebras
$$
\overline{\ee}(H_{\frac{n}{m}}(m)/I(\frac{n}{m}))\overline{\ee}\cong
\overline{\ee}(H_{\frac{m}{n}}(n)/I(\frac{m}{n}))\overline{\ee}
$$
preserving the filtration and the grading, and mapping $\ee_j$ to $\ee_j$ and
$e_j$ to $e_j$.
\end{theorem}

A proof of Theorem \ref{the1} is given in the next subsection.

\begin{corollary}\label{speci} We have natural isomorphisms of algebras
$$
e_j(H_{\frac{n}{m}}(m)/I(\frac{n}{m}))e_j\cong
e_j(H_{\frac{m}{n}}(n)/I(\frac{m}{n}))e_j,
$$
$$
\ee_j(H_{\frac{n}{m}}(m)/I(\frac{n}{m}))\ee_j\cong
\ee_j(H_{\frac{m}{n}}(n)/I(\frac{m}{n}))\ee_j
$$
\end{corollary}

\begin{proof}
The Corollary follows from Theorem \ref{the1} by applying $e_j$,
respectively $\ee_j$ on
both sides.
\end{proof}

\begin{remark}
If $n>m$ then for $m<j\le n$, $e_j\in I(\frac{n}{m})$, so for any $i$,
$$e_i(H_{\frac{m}{n}}(n)/I(\frac{m}{n}))e_j=e_j(H_{\frac{m}{n}}(n)/I(\frac{m}{n}))e_i=0.$$
\end{remark}

\subsection{Proof of Theorem \ref{the1}}

We will show that there is a homomorphism
$$
\Phi_{m,n}:
\overline{\ee}(H_{\frac{m}{n}}(n)/I(\frac{m}{n}))\overline{\ee}\to
\overline{\ee}(H_{\frac{n}{m}}(m)/I(\frac{n}{m}))\overline{\ee}
$$
preserving the Bernstein filtration.
This homomorphism must be injective since
the algebra $\overline{\ee}(H_{\frac{m}{n}}(n)/I(\frac{m}{n}))\overline{\ee}$ is
simple \footnote{Indeed, this algebra is Morita equivalent to the algebra
  $H_{\frac{m}{n}}(n)/I(\frac{m}{n})$, which is simple, and simplicity is a Morita invariant property.}.
This implies that we have a self-inclusion $\Phi_{n,m}\circ \Phi_{m,n}$ of
$\overline{\ee}(H_{\frac{m}{n}}(n)/I(\frac{m}{n}))\overline{\ee}$ preserving the Bernstein
filtration. Since the Bernstein filtration has finite dimensional
quotients, this self-inclusion must be an isomorphism, which
implies the theorem.

To construct $\Phi_{m,n}$, recall that by Proposition \ref{coronatact}, we have a map
$$\tau: \overline{\ee}H_{\frac{m}{n}}(n)\overline{\ee}\to \End\left((M\otimes
  \oplus_{j\text{ even }}S^{n-j}V^*\otimes\wedge^jV^*)^{\mathfrak{g}}\right),$$ which is obtained
by applying $\overline{\ee}$ on both sides to the map
$\widetilde{\tau}:
H_{\frac{m}{n}}(n)\to \End((M\otimes (V^*)^{\otimes n})^{\mathfrak{g}})$ provided by \cite{CEE}, see Proposition \ref{CEEaction} above  (on the right hand
side, $e_j$ symmetrizes with respect to the first $n-j$ indices
and antisymmetrizes with respect to the last $j$
indices, and $\overline{\ee}=\sum_{j\text{ even
  }}e_j$).\footnote{Here $M$ can be taken to be any D-module for
  which the corresponding spaces of invariants are nonzero; for
  example, one can take $M={\mathcal{F}}(M_\lambda)$ for $\lambda$ being a
  partition of $d$, as in Section 5.} Moreover, we know from \cite{CEE}, Section 9, that this map
kills the ideal $\overline{\ee}I(\frac{m}{n})\overline{\ee}$, so it defines a map
$$
\bar \tau:
\overline{\ee}(H_{\frac{m}{n}}(n)/I(\frac{m}{n}))\overline{\ee}\to
\End((M\otimes (\oplus_{j\text{ even }}S^{n-j}V^*\otimes\wedge^jV^*))^{\mathfrak{g}}).
$$

Now, according to \cite{CEE}, the action $\bar \tau$ is given by global differential operators
with values in $U=\oplus_{j\text{ even }}\wedge^{m-j}V^*$. Let
$\xi: A_{1-\frac{n}{m}}(\VV,U)\to \End((M\otimes (\oplus_{j\text{ even }}S^{n-j}V^*\otimes\wedge^jV^*))^{\mathfrak{g}})$
(where $\VV={\mathfrak{sl}}(V)\oplus V$)
be the action of the global hamiltonian reduction from Corollary \ref{co1}.
Let $K=\Ker(\xi)$, and $\bar \xi$ be the
corresponding injective map
$$
\bar\xi: A_{1-\frac{n}{m}}(\VV,U)/K\to
\End((M\otimes(\oplus_{j\text{ even
  }}S^{n-j}V^*\otimes\wedge^jV^*))^{\mathfrak{g}}).
$$
Since the action $\bar \tau$ is given by global differential operators,
it must factor through $\bar \xi$, i.e., there exists a
unique homomorphism $\bar\theta: \overline{\ee}(H_{\frac{m}{n}}(n)/I(\frac{m}{n}))\overline{\ee}\to
A_{1-\frac{n}{m}}(\VV,U)/K$ such that $\bar \tau=\bar \xi\circ \bar\theta$.

Now recall that for any $c$ we have an algebra homomorphism from the global
hamiltonian reduction to the global sections of the local
hamiltonian reduction, $\pi: A_c(\VV,U)\to
{\mathcal  A}_c(\VV,U)$ (see Subsection \ref{locglob}). This descends to $\bar\pi: A_c(\VV,U)/K\to
{\mathcal  A}_c(\VV,U)/\langle\pi(K)\rangle$ (where $\langle S\rangle$ denotes the
ideal generated by $S$). Also, since $e_iH_ce_i\cong e_{m-i}H_{-c}e_{m-i}$,
by Theorem \ref{Thm_iso} we have
an isomorphism $\varphi: {\mathcal  A}_{1-\frac{n}{m}}(\VV,U)\to
\overline{\ee}(H_{\frac{n}{m}}(m))\overline{\ee}$, which induces isomorphisms
$\varphi_j: {\mathcal  A}_{1-\frac{n}{m}}(\VV,\wedge^{m-2j}\CC^m)\to
e_jH_{\frac{n}{m}}(m)e_j$. This isomorphism descends to an isomorphism
$$\varphi': {\mathcal  A}_{1-\frac{n}{m}}(\VV,U)/\langle\pi(K)\rangle\to
\overline{\ee}H_{\frac{n}{m}}(m)\overline{\ee}/\varphi(\langle\pi(K)\rangle).$$ Since
$\overline{\ee}I(\frac{n}{m})\overline{\ee}$ is a maximal ideal, and since ideals in the
Cherednik algebra form a chain (see the beginning of Subsection
\ref{matc}),
we see that $\overline{\ee}I(\frac{n}{m})\overline{\ee}\supset
\varphi(\langle\pi(K)\rangle)$, and hence we have a projection $\gamma:
\overline{\ee}H_{\frac{n}{m}}(m)\overline{\ee}/\phi(\langle \pi(K)\rangle)
\to
\overline{\ee}(H_{\frac{n}{m}}(m)/I(\frac{n}{m})
)\overline{\ee}.$
So the map $\varphi'$ gives rise to a map
$\bar \varphi: {\mathcal  A}_{1-\frac{n}{m}}(\VV,U)/\langle\pi(K)\rangle\to
\overline{\ee}(H_{\frac{n}{m}}(m)/I(\frac{n}{m}))\overline{\ee}$.

So altogether we have a map
$$
\Phi_{m,n}:=\bar \varphi\circ
\bar \pi\circ \bar \theta:
\overline{\ee}(H_{\frac{m}{n}}(n)/I(\frac{m}{n}))\overline{\ee}\to
\overline{\ee}(H_{\frac{n}{m}}(m)/I(\frac{n}{m}))\overline{\ee},
$$
as desired.

It remains to show that the map $\Phi_{m,n}$ preserves the
Bernstein filtration. To show this, note that
the map $\bar \theta$ preserves the Bernstein
filtration by the CEE construction, the map $\bar\varphi$
preserves the Bernstein filtration by the generalized
Gan-Ginzburg construction, and the map $\bar\pi$
preserves the Bernstein filtration because the
corresponding map of the Rees algebras is $(\CC^\times)^2$-equivariant. This implies the required statement.

\subsection{Correspondence between modules over quasi-spherical subalgebras}

\begin{proposition} \label{compa}
For $j=0$, the isomorphisms of Corollary \ref{speci}
coincide with the isomorphism constructed above in Proposition \ref{CEEiso}.
\end{proposition}

\begin{proof}
It suffices to show that these isomorphisms coincide on
the elements $\sum_i x_i^k$ and $\sum_i y_i^2$, since
by the results of \cite{BEG}, such elements generate the
corresponding algebras.

To this end, note that it follows from \cite{CEE}, Section 9, that
in the proof of Theorem \ref{the1} one has
$\bar\theta(\sum_{i=1}^n x_i^r)=\frac{n}{m}\Tr(X^r)$ where $X\in
{\mathfrak{sl}_m}$. Also, it is clear that
$\bar\pi(\Tr(X^r))=\Tr(X^r)$, and
$\bar \phi(\Tr(X^r))=\sum_{i=1}^m x_i^r$.
Thus, $\Phi_{m,n}(\sum_{i=1}^n x_i^r)=\frac{n}{m}\sum_{i=1}^m
x_i^r$. Similarly,
it follows from \cite{CEE}, Section 9, that one has
$\bar\theta(\sum_{i=1}^n y_i^2)=\frac{n}{m}\Delta_\g$ (the
Laplacian of $\g$). Also, it is clear that
$\bar\pi(\Delta_\g)=\Delta_\g$, and it is known from \cite{EG}
that $\bar \phi(\Delta_\g)=\sum_{i=1}^m y_i^2$.
Thus, $\Phi_{m,n}(\sum_{i=1}^n y_i^2)=\frac{n}{m}\sum_{i=1}^m
y_i^2$, as desired.
\end{proof}

\begin{corollary}\label{co3}
The isomorphism of Theorem \ref{the1}
maps $\overline{\ee}L_{\frac{n}{m}}(\frac{m}{d}\mu)$ to $\overline{\ee}L_{\frac{m}{n}}(\frac{n}{d}\mu)$,
and the isomorphisms of Corollary \ref{speci} map
 $e_j L_{\frac{n}{m}}(\frac{m}{d}\mu)$ to $e_j
L_{\frac{m}{n}}(\frac{n}{d}\mu)$,
and $\ee_j L_{\frac{n}{m}}(\frac{m}{d}\mu)$ to $\ee_j
L_{\frac{m}{n}}(\frac{n}{d}\mu)$.
Thus, we have natural isomorphisms of vector spaces
preserving the gradings and the filtrations:
$\overline{\ee} L_{\frac{n}{m}}(\frac{m}{d}\mu)\cong \overline{\ee}
L_{\frac{m}{n}}(\frac{n}{d}\mu)$,
$e_j L_{\frac{n}{m}}(\frac{m}{d}\mu)\cong e_j
L_{\frac{m}{n}}(\frac{n}{d}\mu)$,
$\ee_j L_{\frac{n}{m}}(\frac{m}{d}\mu)\cong\ee_j
L_{\frac{m}{n}}(\frac{n}{d}\mu)$.
\end{corollary}

\begin{proof}
The algebra $\overline{\ee} H_c(n) \overline{\ee}$
is Morita equivalent to both $\ee H_c(n)\ee$ and $H_c(n)$, and
the symmetrizer $\ee$ can be regarded as an idempotent in
$\overline{\ee} H_c(n)
\overline{\ee}$. So, we see by Proposition \ref{match} and Proposition \ref{compa} that the
pullback of $\overline{\ee} L_{\frac{m}{n}}(\frac{n}{d}\mu)$ to $\overline{\ee} H_{1/c}(m) \overline{\ee}$
is $\overline{\ee} L_{\frac{n}{m}}(\frac{m}{d}\mu)$.
This proves the first statement of the Corollary.
The second statement is obtained from the first one
by applying $e_j$ and $\ee_j$, respectively.
\end{proof}

\begin{remark}
Note that by virtue of the above results, the algebras
$\overline{\ee}H_{\frac{m}{n}}(n)\overline{\ee}$ and
$\overline{\ee}H_{\frac{n}{m}}(m)\overline{\ee}$
act on the same space $(M\otimes (\oplus_{j\text{ even}}S^{n-j}V^*\otimes\wedge^jV^*))^\g$
(the first algebra via the CEE construction and the second one
via the generalized Gan-Ginzburg construction), and the images
of these algebras under these actions coincide.
\end{remark}



\section{Symmetrized Koszul-BGG complexes}



\subsection{Quasiisomorphism of Koszul-BGG complexes}

Consider the BGG resolution $K^{\bullet}_{m,n}$. As a vector space, it is just the space $\Omega^{\bullet}\h_n$ of differential forms on the reflection representation, and the homological degree is given by the degree
 of a form. The differential is a contraction $\iota_{\xi}$ with an $S_n$-invariant vector field $\xi=\sum_{i} f_i \frac{\partial}{\partial x_i}$, where $f_i$ are the singular polynomials for Dunkl operators.
In other words, $\iota_{\xi}$ is defined by the identity
$$\iota_{\xi}(dx_{i}\wedge \alpha)=f_i\alpha-dx_i\wedge \iota_{\xi}(\alpha).$$
By definition, $K^{\bullet}_{m,n}$ is a Koszul complex associated to the polynomials $f_i$.

\begin{lemma}
\label{sym}
The symmetrized complex $(K^{\bullet}_{m,n})^{S_n}$ coincides with the
Koszul complex associated to the polynomials $\sum_{i} x_i^{j}f_i,\ 1\le j\le n-1$.
\end{lemma}

\begin{proof}
By a theorem of Solomon \cite{solomon} we have $$(\Omega^{\bullet}\h_n)^{S_n}=\Omega^{\bullet}(\h_n/S_n).$$
The functions on $\h_n/S_n$ are symmetric functions on $\h_n$ and form a polynomial ring in power sums
$p_2,\ldots,p_{n}$. The statement now follows from the identity
$$\iota_{\xi}(dp_{j})=\iota_{\xi}(j\sum_{i} x_{i}^{j-1}dx_i)=j\sum_{i}x_i^{j-1}f_i.$$
\end{proof}

Recall that by (\ref{potential for singular polys}) the singular polynomials are given by the equation
$$f_i=\frac{\partial}{\partial x_i}\Coef_{m+1}\prod_{i}(1-zx_i)^{\frac{m}{n}}.$$

\begin{theorem}
\label{symmetrized BGG}
The symmetrized complex $(K^{\bullet}_{m,n})^{S_n}$ is quasi-isomorphic to the Koszul complex
associated to the sequence of polynomials
\begin{equation}
\label{polys}
\Coef_{j}\left[\left(1+\sum_{k=2}^{n} u_kz^{k}\right)^{m}-\left(1+\sum_{k=2}^{m} v_kz^{k}\right)^{n}\right], 2\le j\le m+n-1
\end{equation}
in variables $u_2,\ldots,u_n,v_2,\ldots,v_m$.
\end{theorem}

\begin{proof}
By Lemma \ref{sym} $(K^{\bullet}_{m,n})^{S_n}$ is isomorphic to the Koszul complex
associated to the polynomials $g_j=\sum_{i} x_i^{j}f_i,\ 1\le j\le n-1$.
Similarly to the proof of  \cite[Theorem 4.3]{Gor}, one can deduce from (\ref{potential for singular polys}) that $g_j$ are related to
$\hat{g}_j=\Coef_{m+j}\prod_{i}(1-zx_i)^{\frac{m}{n}},\  1\le j\le n-1$ by a triangular change. Therefore the Koszul complexes associated to $g_j$ and $\hat{g}_j$ are quasi-isomorphic. If we denote $\prod_{i}(1-zx_i)=1+\sum_{k=2}^{n} u_kz^{k}$, we conclude that
$(K^{\bullet}_{m,n})^{S_n}$ is quasi-isomorphic to the Koszul complex associated to the polynomials
$$\Coef_{m+j}(1+\sum_{k=2}^{n} u_kz^{k})^{\frac{m}{n}},\  1\le j\le n-1$$
in variables $u_k$.
The latter complex is quasi-isomorphic to the Koszul complex associated to the polynomials
$$\Coef_{j}\left[(1+\sum_{k=2}^{n} u_kz^{k})^{\frac{m}{n}}-(1+\sum_{k=2}^{m} v_kz^{k})\right],\  2\le j\le m+n-1$$
in variables $u_k, v_k$, while this set of polynomials is related to (\ref{polys}) by a triangular change
which does not affect its quasi-isomorphism class.
\end{proof}

\begin{corollary}\label{quasii}
The complexes $(K^{\bullet}_{m,n})^{S_n}$ and $(K^{\bullet}_{n,m})^{S_m}$
are quasi-isomorphic as complexes of modules over the ring of symmetric functions.  In particular,
$(L_{\frac{m}{n}}((n)))^{S_{n}}\cong (L_{\frac{n}{m}}((m)))^{S_{m}}.$
\end{corollary}

\begin{proof}
The first statement follows from Theorem \ref{symmetrized BGG}.
The second statement follows from the first one,
since $H_0(K^{\bullet}_{m,n})=L_{\frac{m}{n}}((n))$.
\end{proof}

\subsection{Action of the Hamiltonian}

Recall that the quantum Calogero-Moser Hamiltonian is defined by the formula $H_2=\sum_{i=1}^{n} D_i^{2}.$
Let us compute the action of $H_2$ on
$$\Hom_{S_n}(\wedge^{\bullet}\hh_n,\CC[V])=(\Omega^{\bullet}\hh_n)^{S_n}=\Omega^{\bullet}(\hh_n/S_n).$$
This action is defined because the space
$\Hom_{S_n}(\wedge^{\bullet}\hh_n,\CC[V])$
is the $S_n$-invariants in the Verma module
$M_c(\wedge^\bullet \h_n^*)$ over the rational Cherednik algebra.
We have to compute \linebreak
$H_2(f(x_1,\ldots,x_n)dp_{\alpha_1}\wedge\ldots \wedge dp_{\alpha_k})$,
where $p_i$ are the power sum symmetric functions (providing a coordinate system on $\hh_n/S_n$),
$f$ is a symmetric polynomial in $x_i$ (and thus a polynomial in $p_i$),
and $dp_{\alpha_1}\wedge\ldots \wedge dp_{\alpha_k}$ denotes a copy of $\wedge^{k}\hh_n$ in $\CC[\hh_n]$
spanned by the coefficients of $dp_{\alpha_1}\wedge\ldots dp_{\alpha_k}$ in its expansion in
$dx_i$. Since $H_2$ commutes with the action of $S_n$, its action on $S_n$-equivariant differential forms
is well-defined and preserves the exterior degree.

Recall that $H_2$ is a second order differential operator with $\sum \left(\frac{\partial}{\partial x_i}\right)^2$
as second order part, so one has the identity
$H_2(fg)=H_2(f)g+fH_2(g)+2(\nabla f,\nabla g),$
where $(\nabla f,\nabla g)=\sum_{i}\frac{\partial f}{\partial x_i}\frac{\partial g}{\partial x_i}.$

\begin{lemma}
\label{wedge product rule}
The following equation holds:
$$H_2(dp_{i}\wedge dp_{j})=H_2(dp_i)\wedge dp_{j}+dp_{i}\wedge H_2(dp_{j}).$$
\end{lemma}

\begin{proof}
By definition,
$dp_{i}\wedge dp_{j}$ is a copy of $\wedge^{2}V$ spanned by $\frac{\partial p_{i}}{\partial x_{\mu}}\frac{\partial p_{j}}{\partial x_{\nu}}-\frac{\partial p_{i}}{\partial x_{\nu}}\frac{\partial p_{j}}{\partial x_{\mu}}.$ Therefore

$$H_2(dp_{i}\wedge dp_{j})-H_2(dp_i)\wedge dp_{j}-dp_{i}\wedge H_2(dp_{j})=\sum_{l}\left(\frac{\partial^2 p_{i}}{\partial x_{\mu}\partial x_{l}}\frac{\partial^2 p_{j}}{\partial x_{\nu}\partial x_{l}}-\frac{\partial^2 p_{i}}{\partial x_{\nu}\partial x_{l}}\frac{\partial^2 p_{j}}{\partial x_{\mu}\partial x_{l}}\right)$$

Note that $\frac{\partial^2 p_{i}}{\partial x_{\mu}\partial x_{l}}$ vanish for $\mu\neq l$, so the right hand side  can be rewritten as
$$\sum_{l}\left(\frac{\partial^2 p_{i}}{\partial x_{l}^2}\frac{\partial^2 p_{j}}{\partial x_{l}^2}-\frac{\partial^2 p_{i}}{\partial x_{l}^2}\frac{\partial^2 p_{j}}{\partial x_{l}^2}\right)=0.$$
\end{proof}

\begin{lemma}
\label{H2 for dpk}
The following identity holds:
$$H_2(dp_k)=(1+c)k(k-1)dp_{k-2}-2kc\sum_{s=0}^{k-2}p_{s}dp_{k-2-s}.
$$
\end{lemma}

\begin{proof}
By \cite[Lemma 2.6]{GORS} $H_2(p_k)=(1+c)k(k-1)p_{k-2}-kc\sum_{s=0}^{k-2}p_{s}p_{k-2-s},$ and $dp_{k}$ denotes a copy of $\h$
spanned by $\frac{\partial p_k}{\partial x_\mu}=D_\mu(p_k)$. Therefore $H_2(dp_k)$ is spanned by
$$H_2(dp_k)=\langle H_2(D_{\mu}(p_k))\rangle =\langle D_{\mu}(H_2(p_k))\rangle =dH_2(p_k)=$$ $$(1+c)k(k-1)dp_{k-2}-2kc\sum_{s=0}^{k-2}p_{s}dp_{k-2-s}.$$
\end{proof}

\begin{lemma}
\label{linear term}
The following equation holds:
 $$(\nabla f,\nabla dp_{k})=\sum_{s}\frac{sk(k-1)}{k+s-2}\frac{\partial f}{\partial p_{s}}dp_{k+s-2}.$$
\end{lemma}

\begin{proof}
By definition,
$(\nabla f,\nabla (dp_{k})_{\mu})=\sum_{l}\frac{\partial f}{\partial x_{l}}\frac{\partial^2 p_k}{\partial x_{\mu}\partial x_{l}}=k(k-1)x_{\mu}^{k-2}\frac{\partial f}{\partial x_{\mu}}.$
This is a first order differential operator in $f$, so it is sufficient to compute it for $f=p_s$:
$$(\nabla p_{s},\nabla dp_{k})_{\mu}=k(k-1)x_{\mu}^{k-2}\frac{\partial p_{s}}{\partial x_{\mu}}=sk(k-1)x_{\mu}^{k+s-3}=\frac{sk(k-1)}{k+s-2}(dp_{k+s-2})_{\mu}.$$
\end{proof}

\begin{theorem}
The action of $H_2$ on the $S_n$-invariant differential forms is given by the equation
\begin{multline}
\label{H2 for forms}
H_2(fdp_{\alpha_1}\wedge\ldots dp_{\alpha_k})=H_2(f)dp_{\alpha_1}\wedge\ldots dp_{\alpha_k}+\\2\sum_{j=1}^{k}\sum_{s}\frac{s\alpha_{j}(\alpha_{j}-1)}{\alpha_{j}+s-2}\frac{\partial f}{\partial p_{s}}dp_{\alpha_1}\wedge\ldots \wedge dp_{\alpha_j+s-2}\wedge \ldots dp_{\alpha_k}-\\
2cf\sum_{j=1}^{k}\sum_{s}p_{s}\alpha_{j}dp_{\alpha_1}\wedge\ldots \wedge dp_{\alpha_j-2-s}\wedge \ldots dp_{\alpha_k}+\\
(1+c)f\sum_{j=1}^{k}\alpha_{j}(\alpha_{j}-1)dp_{\alpha_1}\wedge\ldots \wedge dp_{\alpha_j-2}\wedge \ldots dp_{\alpha_k}.\\
\end{multline}
\end{theorem}

\begin{proof}
By  Lemma \ref{wedge product rule} one has
 $$H_2(f\cdot dp_{\alpha_1}\wedge\ldots dp_{\alpha_k})=H_2(f)dp_{\alpha_1}\wedge\ldots dp_{\alpha_k}+\sum_{j=1}^{k}fdp_{\alpha_1}\wedge\ldots \wedge H_2(dp_{\alpha_j})\wedge \ldots dp_{\alpha_k}$$
$$+2\sum_{j=1}^{k}(-1)^{j-1}(\nabla f,\nabla dp_{\alpha_j})dp_{\alpha_1}\wedge\ldots \wedge \widehat{dp_{\alpha_j}}\wedge \ldots dp_{\alpha_k}. $$
Now the theorem follows from Lemma \ref{linear term} and Lemma \ref{H2 for dpk}.
\end{proof}

\begin{corollary}
Let us consider two sets of coordinates $\{x_i\}$, $\{\widetilde{x_i}\}$
such that $\widetilde{p_i}=cp_{i}$. Then $H_2^{1/c}(\widetilde{p_i})=\frac{1}{c}H_2^{c}(p_i)$.
\end{corollary}

\begin{proof}
The statement was proved in \cite[Theorem 2.9]{GORS} for symmetric functions. Let us extend it to the differential forms.
Indeed, $d\widetilde{p_k}=c\cdot dp_{k}$, and
$$H_2^{1/c}(f)=\frac{1}{c}H_2^c(f),\
\frac{\partial f}{\partial \widetilde{p_{s}}}=\frac{1}{c}\frac{\partial f}{\partial p_{s}},$$
$$\frac{1}{c}\widetilde{p_{k}}=\frac{1}{c}(cp_{k}),\
{1+\frac{1}{c}}=\frac{1}{c}(1+c).$$
Therefore every term in (\ref{H2 for forms}) is multiplied by $\frac{1}{c}$.
\end{proof}

It follows from Proposition \ref{Dunkl vs Koszul} that
the actions of $H_2$ and $\iota_{\xi}$ commute.

\begin{theorem}\label{symBGGH2}
The quasi-isomorphism of Corollary \ref{quasii} between the complexes
$(K^{\bullet}_{m,n})^{S_n}$ and
$(K^{\bullet}_{n,m})^{S_m}$
commutes with the action of $H_2$.
In other words, if $Y$ is the algebra freely
generated by the symbol $H_2$ and symmetric functions
in infinitely many variables, then $(K^{\bullet}_{m,n})^{S_n}$ and
$(K^{\bullet}_{n,m})^{S_m}$ are complexes of $Y$-modules, and
the quasi-isomorphism of Corollary \ref{quasii}
is a quasi-isomorphism of complexes of $Y$-modules.
\end{theorem}

\begin{proof}
Let us extend the action of $H_2$ to the constructions
of Theorem \ref{symmetrized BGG}.
Consider the polynomial ring in variables $u_2,\ldots, u_n, v_2,\ldots, v_m$.
We identify $u_i$ with elementary symmetric polynomials in variables $x_1,\ldots, x_n$ and $v_i$ with the elementary symmetric polynomials in variables $\widetilde{x_1},\ldots,\widetilde{x_m}$. Consider the operator
$H:=nH_2+m\widetilde{H_2}.$
It is sufficient to check that $H$ preserves the Koszul complex associated with the equations
$$\Coef_{i}\left[(1+\sum(-1)^{i}u_iz^{i})^{m}-(1+\sum(-1)^{i}v_iz^{i})^{n}\right],\ i=2\ldots m+n-1.$$
We can change variables and consider instead power sums in $x_i$ and $\widetilde{x_i}$:
the generators will be $p_2,\ldots, p_n, \widetilde{p_2},\ldots, \widetilde{p_{m}}$, and the equations
$E_i:=mp_i-n\widetilde{p_i}=0,\ i=2\ldots m+n-1.$
The corresponding Koszul complexes will be quasi-isomorphic, and it follows from
\cite[Lemma 2.6]{GORS} that
$$\frac{1}{mn}H(E_i)=H_2(p_i)-\widetilde{H_2}(\widetilde{p_i})=
\left(\frac{m+n}{n}i(i-1)p_{i-2}-\frac{m+n}{m}i(i-1)\widetilde{p}_{i-2}\right)-$$
$$\left(i\frac{m}{n}\sum_{s=0}^{i-2}p_{s}p_{i-2-s}-i\frac{n}{m}\sum_{s=0}^{i-2}\widetilde{p_{s}}\widetilde{p}_{i-2-s}\right)=$$
$$\frac{m+n}{mn}i(i-1)E_{i-2}-\frac{i}{mn}\sum_{s=0}^{i-2}\left(mp_{i-2-s}E_{s}+n\widetilde{p_{s}}E_{i-2-s}\right).$$
Since $H(E_i)$ belongs to the ideal generated by $E_j$ with $j<i$, the Koszul complex associated with $E_i$ is invariant under $H$.
\end{proof}

\subsection{Yet another proof of Theorem \ref{homolkoz}(ii)}\label{thirdproof}

Here is a third proof of Theorem \ref{homolkoz}(ii),
based on Theorem \ref{symBGGH2}.
First, note that the statement holds if $m$ is divisible by $n$. In this case, $d=n$, the differential is zero,
so the statement is trivial. Next, by the results of \cite{BEG},
the spherical subalgebra $\ee H_c(n)\ee$ is generated by
symmetric functions of the $x_i$ and $H_2:=\sum_i y_i^2$.
Therefore,  Theorem \ref{symBGGH2} and Proposition \ref{match} imply
that if the statement of Theorem \ref{homolkoz}(ii) holds
for $(m,n)$ then it holds for $(n,m)$. Finally, by Corollary
\ref{shift}, using the fact that the shift functor is an equivalence (\cite{BE}),
we see that if the statement holds for $(m,n)$ with $m\ge n$, then it holds for
$(m-n,n)$. This implies the result by using the Euclidean algorithm (more precisely, any pair $(m,n)$
can be reduced to one with $m$ divisible by $n$ by transformations $(m,n)\mapsto (n,m)$ for $m<n$ and $(m,n)\mapsto (m-n,n)$ for $m\ge n$).

\begin{remark}
Instead of the shift functors (i.e., Corollary \ref{shift}), we could have used Rouquier
equivalences of highest weight categories ${\mathcal
 O}_{\frac{m}{n}}\cong {\mathcal O}_{\frac{m'}{n}}$, where $m,m'>0$ and
$GCD(m,n)=GCD(m',n)$ (\cite{rouqqsch}).
Note that for $GCD(m,n)=2$, these equivalences were
constructed later in \cite{mult_cher}.
\end{remark}


\begin{thebibliography}{APK}

\bibitem[A]{aiston} A. Aiston.  A skein theoretic proof of the hook formula for quantum dimension.
arXiv:q-alg/9711019

\bibitem[AM]{amorton} A. Aiston, H. Morton. Idempotents of Hecke algebras of type A. J. Knot
Theory Ramifications {\bf 7} (1998), no. 4, 463--487.

\bibitem[BEG]{BEG}
Yu. Berest, P. Etingof, V. Ginzburg,
Finite-dimensional representations of rational Cherednik
algebras. Int. Math. Res. Not.  2003,  no. 19, 1053--1088.

\bibitem[BGS]{BGS} C. Berkesch, S. Griffeth, S. V. Sam,
Jack polynomials as fractional quantum Hall states and the Betti numbers of the (k+1)-equals ideal,
 arXiv:1303.4126.

\bibitem[BE]{BE}
R. Bezrukavnikov, P. Etingof. Parabolic
induction and restriction functors
for rational Cherednik algebras. Selecta Math.,
{\bf 14} (2009), 397--425.

\bibitem[BK]{BK}
R. Bezrukavnikov and D. Kaledin, Fedosov quantization in algebraic context, Moscow Math. J. 4 (2004), 559-592.

\bibitem[Br]{Br}
A. Broer, The sum of generalized exponents
and Chevalley's restriction theorem for modules of
covariants, Indag. Mathem.,v. 6
(1995), issue 4, pp. 385--396.

\bibitem[CE]{CE} T. Chmutova; P. Etingof,
On some representations of the rational Cherednik algebra.
Represent. Theory {\bf 7} (2003), 641--650.

\bibitem[CEE] {CEE} D. Calaque, B. Enriquez, P. Etingof. Algebra,
  arithmetic, and geometry: in honor of Yu. I. Manin. Vol. I,
  165--266, Progr. Math., {\bf 269}, Birkh\"auser Boston, Inc.,
  Boston, MA, 2009,
 arXiv:math/0702670; the references in the text are to the arXiv
 numbering.

\bibitem[CR]{CR} J. Chuang and R. Rouquier, {\it Derived equivalences for symmetric groups and
$\mathfrak{sl}_2$-categorifications}. Ann. Math. (2)  167(2008), n.1, 245-298.

\bibitem[DLT]{DLT} J. D\'esarm\'enien, B.  Leclerc, J.-Y. Thibon. Hall-Littlewood functions and Kostka-Foulkes polynomials in representation theory. S\'em. Lothar. Combin. {\bf 32} (1994), Art. B32c, approx. 38 pp.

\bibitem[D]{dunkl2} C. Dunkl. Intertwining operators and polynomials associated with the symmetric group. Monatsh. Math. {\bf 126} (1998), no. 3, 181--209.


\bibitem[DO]{DO}
C. F. Dunkl and E. M. Opdam,
Dunkl operators for complex reflection groups, Proc. London Math. Soc. (3) 86 (2003), 70--108.

\bibitem[EG]{EG} P. Etingof and V. Ginzburg, Symplectic re∞ection algebras, Calogero-Moser space, and deformed Harish-
Chandra homomorphism, Invent. Math., 147 (2002), 243-348.

\bibitem[ES]{ES}
P. Etingof, E. Stoica, Unitary representations of rational Cherednik algebras (with an appendix of Stephen Griffeth),
arXiv:0901.4595, Represent. Theory 13 (2009), 349--370.

\bibitem[FS]{FS}
M. Feigin, A. Silantyev,
Singular polynomials from orbit spaces, arxiv:1110.1946,
Compositio Mathematica, 2012, v. 148 (6), pp. 1867-1879.

\bibitem[FH]{fuha} W. Fulton, J. Harris. Representation theory. A first course. Graduate Texts in Mathematics, {\bf 129}. Readings in Mathematics. Springer-Verlag, New York, 1991.


\bibitem[GG]{GG}
W. L. Gan and V. Ginzburg, Almost-commuting variety, D-modules, and Cherednik algebras, with an appendix by Ginzburg. IMRP Int. Math. Res. Pap. (2006).


\bibitem[Gi1]{Ginzburg_prim}
V.~Ginzburg, On primitive ideals, Selecta Math. {\bf 9} (2003), n. 3, 379-407.

\bibitem[Gi2]{Ginzburg_Pro} V.~Ginzburg, Isospectral commuting
  variety, the Harish-Chandra D-module, and principal nilpotent
  pairs, Duke Mathematical Journal. Volume 161, Number 11 (2012),
  2023-2111.

\bibitem[GGOR]{GGOR}
V. Ginzburg, N. Guay, E. Opdam, R. Rouquier,
On the category $\mathcal{O}$ for rational Cherednik algebras.
Invent. Math.  {\bf 154}  (2003),  no. 3, 617--651.

\bibitem[Go]{Go}
I. Gordon, On the quotient ring by diagonal invariants,
Invent. Math. {\bf 153}, no. 3, 503--518.

\bibitem[GL]{GL}  I. Gordon, I. Losev,
On category O for cyclotomic rational Cherednik
algebras, arXiv:1109.2315.

\bibitem[Gor1]{adams} E. Gorsky.  Adams Operations and Power Structures. Mosc. Math. J. {\bf 9} (2009), no. 2, 305--323.

\bibitem[Gor2]{Gor}
E. Gorsky, Arc spaces and DAHA representations,
Selecta Math. {\bf 19} (2013), no. 1, 125--140.	


\bibitem[GORS]{GORS} E. Gorsky, A. Oblomkov, J. Rasmussen, V. Shende. Torus knots and the rational DAHA. arXiv: 1207.4523

\bibitem[Gr]{Groth} A. Grothendieck. Cohomologie locale de faisceaux coh\'{e}rent et
th\'{e}or\`{e}mes de Lefschetz locaux et globaux, North Holland, 1968.

\bibitem[GS]{GS} S. Gukov, M. Stosic. Homological algebra of knots and BPS states. arXiv:1112.0030

\bibitem[Gu]{Gupta}  R. K. Gupta. Characters and the $q$-analog of weight multiplicity. J. London Math. Soc. (2) {\bf 36} (1987), no. 1, 68--76.

\bibitem[Hai]{Haiman} Mark Haiman, Hilbert schemes,
  polygraphs, and the Macdonald positivity conjecture. J. Amer.
  Math.  Soc. 14 (2001) 941-1006.


\bibitem[Hur]{Hur} A. Hurwitz, \"Uber die Anzahl der Riemann'schen Fl\"aschen mit gegebenen Verzweigungspunkten. Math. Ann. {\bf 55} (1902), 53--66.

\bibitem[K]{Kostant} B. Kostant. Lie group representations on polynomial rings. Amer. J. Math. {\bf 85} (1963), 327--404.



\bibitem[KP]{KP} A. Kirillov, I. Pak. Covariants of the symmetric group and its analogues in A. Weil algebras. Funct. Anal. Appl. {\bf 24} (1990), no. 3, 172--176.

\bibitem[LZ]{LZ} X.-S. Lin, H. Zheng. On the Hecke algebras and the colored HOMFLY polynomial. Trans. Amer. Math. Soc. {\bf 362} (2010), no. 1, 1--18.


\bibitem[L1]{Lo} I. Losev.  Completions of Symplectic reflection algebras. Selecta Math. {\bf 18} (2012),  no. 1, 179--251.

\bibitem[L2]{cryst} I. Losev, {\it Highest weight $\mathfrak{sl}_2$-categorifications I: crystals}.
arXiv:1201.4493. Accepted by Math. Zeitschrift.

\bibitem[L3]{str} I. Losev. Highest weight $\sl_2$-categorifications II: structure theory.
arXiv:1203.5545.

\bibitem[L4]{mult_cher} I. Losev,  Towards multiplicities for cyclotomic rational Cherednik algebras.
arXiv:1207.1299.

\bibitem[L5]{HC} I. Losev. Finite dimensional representations of
W-algebras. Duke Math J. {\bf 159} (2011), n.1, 99--143.  arXiv:0807.1023.

\bibitem[L6]{quant}
I. Losev. Isomorphisms of quantizations via quantization of resolutions. Adv. Math. {\bf 231} (2012), 1216-1270. arXiv:1010.3182.


\bibitem[LM]{morton1} S. Lukac, H. Morton, The Homfly polynomial of the decorated Hopf link,
J. Knot Theory Ramifications {\bf 12} (2003), no. 3, 395-416.

\bibitem[Lu]{Lusztig} G. Lusztig. Singularities, character formulas, and a $q$-analog of
 multiplicities. Analysis and topology on singular spaces, II, III (Luminy, 1981), 208--229, Ast\'erisque, 101--102, Soc. Math. France, Paris, 1983.

\bibitem[M]{macd} I. G. Macdonald. Symmetric functions and Hall polynomials. Second edition.  The Clarendon Press, Oxford University Press, New York, 1995.

\bibitem[Mir]{Mir} I. Mirkovic.
Character sheaves on reductive Lie algebras.,
Mosc. Math. J. {\bf 4}, no. 4, 897–910, (2004).

\bibitem[MMS]{MMS}  A. Mironov, A. Morozov, S. Shakirov. Torus HOMFLYPT as the Hall-Littlewood polynomials. J. Phys. A {\bf 45} (2012), no. 35, 355202.

\bibitem[Mol]{molch}
V. Molchanov. On the Poincar\'e series of representations of finite reflection groups.
Funct. Anal. Appl. {\bf 26} (1992), no. 2, 143--145.

\bibitem[MM]{morton2} H. Morton, P. Manch\'on, Geometrical relations and plethysms in the
Homfly skein of the annulus. J. Lond. Math. Soc. (2) {\bf 78} (2008), no. 2,
305--328.

\bibitem[ORS]{ORS}
A. Oblomkov, J. Rasmussen, V. Shende, The Hilbert scheme of a
plane curve singularity and the HOMFLY homology of its link,
with an appendix of E. Gorsky,
arXiv:1201.2115.

\bibitem[OS]{OS} A. Oblomkov, V. Shende. The Hilbert scheme of a plane curve singularity and the HOMFLY polynomial of its link. Duke Math. J. {\bf 161} (2012), no. 7 , 1277--1303.

\bibitem[Resh]{resh} N. Y. Reshetikhin. Quantized universal enveloping algebras, the Yang-
Baxter equation and invariants of links I and II. Preprint, E-4-87 E-17-87
LOMI 1987.

\bibitem[RJ]{RJ} M. Rosso, V. Jones. On the invariants of torus knots derived from quantum groups.
J. Knot Theory Ramifications {\bf 2} (1993), no. 1, 97--112.

\bibitem[R]{rouqqsch} R. Rouquier. $q$-Schur algebras for complex reflection groups. Mosc. Math. J. {\bf 8} (2008),
119-158.

\bibitem[Re]{Re} G. A. Reisner,
Cohen-Macaulay Quotients of Polynomial Rings
, Advances in Math. 21
(1976), 30–49.

\bibitem[Sh]{Shan} P. Shan. Crystals of Fock spaces and cyclotomic rational double affine Hecke algebras.
Ann. Sci. \'Ecole Norm. Sup. {\bf 44} (2011), 147-182.

\bibitem[SV]{SV} P. Shan, E. Vasserot. Heisenberg algebras and rational double affine Hecke algebras. J. Amer. Math. Soc. {\bf 25} (2012),
959–1031. arXiv:1011.6488

\bibitem[S]{solomon} L. Solomon, Invariants of finite reflection groups. Nagoya Mathem. J. {\bf 22} (1963), 57--64.

\bibitem[Sta]{stanley} R. Stanley. Some combinatorial properties of Jack symmetric functions. Adv. Math. {\bf 77} (1989), no.1, 76--115.

\bibitem[Ste]{stevan} S. Stevan. Chern-Simons invariants of torus links. Ann. Henri Poincar\'e {\bf 11} (2010), no. 7, 1201--1224.

\bibitem[V]{vdb}
M. Van den Bergh, Differential operators on semi-invariants for tori and weighted projective spaces, in
Topics in Invariant Theory, Lecture Notes in Math. 1478, Springer, Berlin, 1991, pp. 255--272.

\bibitem[Wi]{Wi} S. Wilcox, Supports of representations of the
rational Cherednik algebra of type A, arXiv:1012.2585.

\end{thebibliography}
\end{document}